\documentclass[11pt, letterpaper]{article}
\usepackage[margin=.8in]{geometry}
\usepackage[sort, numbers]{natbib}
\usepackage{amssymb}
\usepackage{amsmath}
\usepackage{amsthm}
\usepackage{dsfont}
\usepackage{graphicx}
\usepackage{color}
\usepackage{bm}
\usepackage{soul}
\usepackage{enumitem}
\usepackage{titling}
\usepackage{hyperref}
\hypersetup{
	pdftitle={Occupation of random walk cut points},
	pdfauthor={Y. Gao, X. Li, P. Panov, D. Shiraishi},
	bookmarksnumbered=true,
	bookmarksopen=true,
	bookmarksopenlevel=1,
	breaklinks=false,
	backref=false,
	colorlinks=true
}

\numberwithin{equation}{section}
\numberwithin{figure}{section}

\newtheorem{theorem}{Theorem}[section]
\newtheorem{definition}[theorem]{Definition}
\newtheorem{lemma}[theorem]{Lemma}
\newtheorem{proposition}[theorem]{Proposition}
\newtheorem{corollary}[theorem]{Corollary}

\newtheorem{remark}[theorem]{Remark}

\theoremstyle{remark}

\newcommand{\Ac}{\mathcal{A}}
\newcommand{\Bc}{\mathcal{B}}

\newcommand{\Dc}{\mathcal{D}}
\newcommand{\Ec}{\mathcal{E}}
\newcommand{\Fc}{\mathcal{F}}
\newcommand{\Nc}{\mathcal{N}}
\newcommand{\Mc}{\mathcal{M}}
\newcommand{\eM}{\mathcal{M}}
\newcommand{\Pc}{\mathcal{P}}

\newcommand{\Uc}{\mathcal{U}}
\newcommand{\Vc}{\mathcal{V}}
\newcommand{\Xc}{\mathcal{X}}
\newcommand{\Zc}{\mathcal{Z}}

\newcommand{\Wc}{\mathcal{W}}
\newcommand{\Kc}{\mathcal{K}}

\newcommand{\Eb}{\mathbb{E}}

\newcommand{\Nb}{\mathbb{N}}
\newcommand{\Pb}{\mathbb{P}}
\newcommand{\Rb}{\mathbb{R}}
\newcommand{\Zb}{\mathbb{Z}}

\newcommand{\Sc}{\mathcal{S}}
\newcommand{\Af}{\mathfrak{A}}
\newcommand{\gf}{\mathfrak{g}}

\newcommand{\Qf}{\mathbf{Q}}

\newcommand{\Gt}{\widetilde{G}}
\newcommand{\Kt}{\widetilde{K}}

\newcommand{\baK}{\overline{K}}
\newcommand{\baS}{\overline{S}}
\newcommand{\baW}{\overline{W}}

\newcommand{\babet}{\overline{\beta}}

\newcommand{\bagam}{\overline{\gamma}}
\newcommand{\baGam}{\overline{\Gamma}}

\newcommand{\dist}{\mathrm{dist}}
\newcommand{\cut}{\mathrm{cut}}
\newcommand{\Cont}{\mathrm{Cont}}

\newcommand{\len}{\mathrm{len}}

\newcommand{\eps}{\varepsilon}
\newcommand{\wh}{\widehat}
\newcommand{\wt}{\widetilde}
\newcommand{\ol}{\overline}

\newcommand{\state}{\mathcal{X}}
\newcommand{\ball}{\mathcal{B}}
\newcommand{\op}{\operatorname}
\newcommand{\Rd}{\Rb^d}
\newcommand{\Zd}{\Zb^d}
\newcommand{\upinf}{\uparrow \infty}
\newcommand{\limup}[1]{\lim_{#1 \upinf}}

\newcommand{\eZn}{\Zc_n}
\newcommand{\norm}[1]{||#1||}

\newcommand{\floo}[1] {\lfloor #1 \rfloor}

\title{Scaling limit of the occupation measure of random walk cut points}
\author{Yifan Gao\thanks{City University of Hong Kong. \href{mailto:yifangao@cityu.edu.hk}{yifangao@cityu.edu.hk}} \and Xinyi Li\thanks{Beijing International Center for Mathematical Research, Peking University. \href{mailto:xinyili@bicmr.pku.edu.cn}{xinyili@bicmr.pku.edu.cn}} \and Petr Panov\thanks{Optiver Holding B.V. \href{mailto:pvpanov93@gmail.com}{pvpanov93@gmail.com}} \and Daisuke Shiraishi\thanks{Graduate School of Informatics, Kyoto University. \href{mailto:shiraishi@acs.i.kyoto-u.ac.jp}{shiraishi@acs.i.kyoto-u.ac.jp}}}
\date{\today}
\begin{document}	
\maketitle
\begin{abstract}
	We consider the occupation measure of the cut points of a simple random walk on a $d$-dimensional cubic lattice for $d=2,3$, and we show that the scaling limit of the occupation measure in weak topology is the natural fractal measure on the Brownian cut points defined via its Minkowski content. 
\end{abstract}

\tableofcontents

\section{Introduction}
A point $x$ is a \emph{cut point} for a curve $\gamma$ in $\mathbb{R}^d$, $d\geq 2$, if $x$ has been visited only once by $\gamma$ and $\gamma\setminus\{x\}$ is not connected. In a similar fashion, one can define cut points for discrete paths in $\mathbb{Z}^d$. The set of cut points of Brownian motion and simple random walk in two and three dimensions form important examples of random fractals in the continuum and the discrete respectively. It is very natural to wonder if the latter is the scaling limit of the former, just as Brownian motion is the scaling limit of simple random walk. In this work, we answer this question positively by showing that the rescaled occupation measure of the random walk cut points converges weakly to the occupation measure of Brownian cut points.

\medskip

To describe our results more precisely and explain how they relate to previous results in this direction,
we start by discussing the intersection exponents and we recall some facts about the behavior of the Brownian and the random walk cut points.
In \cite{RWcuttimes}, Lawler showed that for $d\leq 3$, 
\begin{equation} \label{eq:nonint_exp_RW}
	\Pb \{ S^1[0, n] \cap S^2(0, n] = \emptyset  \} \asymp
	n^{-\xi / 2},
\end{equation}
where $S^{1}$ and $ S^{2}$ are independent simple random walks in $\mathbb{Z}^{d}$ started at the origin, $\xi=\xi_d$ is the \emph{intersection exponent},  and $\asymp$ means ``within multiplicative constants of'' (see Section \ref{sec:def} for a precise definition). {We also remark that the probability above is $\ge c$ for $d \ge 5$  and $\asymp (\log n)^{-1/2}$ for $d=4$; see e.g.\ \cite[Section 10]{RWintro}.

We can write an expression similar to \eqref{eq:nonint_exp_RW} for Brownian motion, although it is a bit more subtle. {If $W:[0,\infty)\mapsto \Rd$ is a standard Brownian motion and $d \le 3$, then for any $0 < s < t$, almost surely, $W(s) \not \in \Ac_{W[0, t]}$} (here and below, we denote the set of cut points of a curve $\gamma$ by $\Ac_{\gamma}$). If $W^1$ and $W^2$ are independent Brownian motions with $W^{1} (0) = 0$ and $ W^{2} (0)  \in \partial \Dc$ where $\Dc$ stands for the unit open ball  around the origin, then it is proved that 
\begin{equation} \label{eq:nonint_exp_BM}
	\Pb\{W^1[0, n] \cap W^2[0, n] = \emptyset \} \asymp
	n^{-\xi / 2}.
\end{equation}
Indeed, the classical gambler's ruin estimate shows that $\xi = 2$ for $d = 1$. {The estimate \eqref{eq:nonint_exp_BM} has been obtained in \cite{hausdorff} for $d=2, 3$. Moreover, in \cite{exp2}, Lawler, Schramm and Werner determined that $\xi = 5 / 4$ for $d = 2$ by making use of techniques developed in studying Schramm-Loewner evolution (SLE).} While the exact value of $\xi$ for $d = 3$ is not known, {bounds $1/2 < \xi < 1$ have been established rigorously} in \cite{burdzy1990, BMinvar}, and numerical simulations in \cite{3expsim} suggest that $\xi$ is close to $.58$. We also mention that the probability in \eqref{eq:nonint_exp_BM} is equal to $1$ for $d \ge 4$; see e.g.\ \cite[Section 9]{morters2010brownian}.

The estimate \eqref{eq:nonint_exp_BM} suggests that the set of Brownian cut points enjoys fractal nature, and in fact, Lawler in \cite{hausdorff} showed that $\Ac_{W[0, 1]}$ has box and Hausdorff dimension $\delta$, where 
\begin{equation}\label{eq:eta}
	\delta := d - \eta, \quad  \eta := \xi + d - 2.
\end{equation}
It is then a natural task to construct the occupation measure of Brownian cut points. Recently it has been showed in \cite{mink_cont} that  the $\delta$-Minkowski content of Brownian cut points exists, it is non-trivial and it induces a natural fractal measure on Brownian cut points. Here we state a specific version tailored for our setup (see Section \ref{subsec:review} for more discussions). Consider the standard Brownian motion $\gamma$ started from $0$ and  stopped when it exits $\Dc$, and let $H_{r} (x) = \{ \dist(x, \Ac_\gamma) \le e^{-r} \}$. We set
\begin{equation*}
\Cont_\delta(\Ac_\gamma) :=
\limup{r} \, \int_{\Rd}  e^{r\eta} \, 1_{H_{r} (x)} dx 
\end{equation*}
for the $\delta$-dimensional Minkowski content of the cut points of path $\gamma$. Using results in \cite{mink_cont}, one can define almost surely a regular non-atomic finite Borel measure $\nu$ by letting $\nu(\cdot) = \Cont_\delta (\cdot \cap \Ac_\gamma)$. 

We now define the (discrete) occupation measure of random walk cut points. For $n\geq 0$ (not necessarily an integer), we write
\begin{equation} \label{eq:def_rwocc}
	\nu_n :=
	c_1 \, e^{-n(2 - \xi)} \sum_{x \in {\eZn \cap \Ac_{\lambda}}} \delta_{x},
\end{equation}
where $\delta_{x}$ is a unit point mass at $x$, $\lambda$ is a simple random walk in $\eZn := e^{-n} \Zd$ started at $0$ and stopped when it exits $\Dc\cap\eZn$ (see Section~\ref{sec:def}, in particular the part on SRW, for a precise definition), {and $c_1$ is the universal constant provided by Theorem~\ref{t:one_pt}.}

We are now ready to state our main result.
\begin{theorem}\label{thm:REALMAIN}
The law of $\nu_{n}$ converges weakly to that of $\nu$ with respect to the topology of weak convergence of finite measures. 
\end{theorem}

\medskip

This result is quite intuitive since if one couples a simple random walk and a Brownian motion together via Skorokhod embedding, then the cut points of both trajectories should appear roughly at the same locations and hence the continuous occupation measure should be well approximated by an appropriately renormalized discrete counterpart. In fact, Theorem \ref{thm:REALMAIN} is a corollary of the following $L^2$-convergence of occupation measures. 

\begin{theorem}\label{thm:vg}
        There is a coupling (see Section \ref{subsec:dcc} for more details) for the simple random walk $\lambda$ and the Brownian motion $\gamma$ defined above such that for any real-valued bounded continuous function $g$ on $\overline\Dc$, 
	\begin{equation}\label{eq:nug}
	\lim_{n\rightarrow\infty}\Eb[ (\nu_n(g)-\nu(g))^2 ] =0.
	\end{equation}
	Moreover, $\nu_n$ converges in probability for the weak topology towards $\nu$. 
\end{theorem}

In its turn, Theorem \ref{thm:vg} relies on the following local $L^2$-bound for occupation measures. 

\begin{theorem}
	\label{thm:rw_occ}
	Under the same coupling as Theorem \ref{thm:vg}, there is a constant $u>0$ such that for any ``nice'' box $V$ in $\Dc$ (see \eqref{eq:nice-box} for a precise definition),
	\[
	\Eb [(\nu_n(V) - \nu(V))^2] \le c \, e^{-un},
	\]
	where $c=c(V)$ is a positive constant that depends on $V$. 
\end{theorem}

We now discuss the proof of Theorem \ref{thm:rw_occ}. 

In order to relate the discrete occupation measure to the continuum, it is necessary to take one step back by looking at the cut balls (which can be regarded as a kind of approximation for cut points; see Section \ref{sec:gcb} for a precise definition). By adapting the ideas of \cite{mink_cont}, we are able to show in Propositions \ref{prop:one-point} and \ref{prop:two-point} that the Brownian cut-point Green's function, which is a key quantity in showing the existence of the Minkowski content of cut points, can also be interpreted as the Green's function of the Brownian cut balls. Similar arguments also work in the discrete case, see Propositions \ref{prop:f2} and \ref{prop:AKAA}, with the help of the coupling of pairs of non-intersecting random walks with different ``initial configurations'', which we discuss in Section \ref{sec:coupling}. 

We are left with the most difficult part in the $L^2$-control in Theorem \ref{thm:rw_occ}: the cross-terms. More precisely, under the Skorokhod embedding that couples together the random walk and the Brownian motion, we need to compare discrete and continuous cut-ball events, with the presence of a random walk cut point at mesoscopic distance apart. This crucial part is carried out  in  Proposition \ref{prop:ct-3}.

Our strategy relies on $L^2$-approximation to establish the convergence of discrete occupation measures  (in order to prove the convergence in natural parametrization) or the existence of natural measures for random fractals. This strategy is pretty stable and has been applied in various models, e.g.\ critical planar percolation, Schramm-Loewner evolutions, loop-erased random walks, Brownian cut points; see \cite{garban2013pivotal, lawler2015minkowski, LSNPLERW,holden2022natural,mink_cont} respectively. Indeed the scheme of our work can be adapted to another type of random fractals - the frontier (outer boundary) of  random walk and Brownian motion in two dimensions. In a forthcoming work \cite{NPfrontier},  three of the authors of this work are going to show that the frontier of planar random walk converges in  the natural parametrization towards  that of the planar Brownian motion, which is a variant of ${\rm SLE}_{8/3}$. In the course of the proof, a key step is to establish the convergence of the occupation measure of the frontier of the random walk, which will be tackled in a similar fashion as in this work. 

However, there is a major difference between our work and those cited above that deal with convergence of occupation measures of discrete random fractals where one builds $L^2$-approximation upon the knowledge that the scaling limit of the corresponding discrete fractal viewed as a set exists in the {\it Hausdorff} sense. In contrast, in this work we skip this procedure and work directly from a strong coupling of random walks and Brownian motion. It is then very natural to wonder  if one can also establish the convergence of the set of random walk cut points to that of Brownian cut points with respect to the Hausdorff distance (which does not trivially follow from our approximation scheme). We plan to tackle this question in a future work.
\medskip

The structure of this paper is as follows. 
We start by setting up the notation and basic denfinitions in Section~\ref{sec:def}. We then recall some standard facts on the Brownian motion and the simple random walk and introduce the path measures (Section~\ref{sec:prelims}) which will be used throughout this work. In Section~\ref{sec:nonint} we review some well-known results about the non-intersecting simple random walks together with the non-intersecting  Brownian paths. In Section~\ref{sec:moment}, we give some moment bounds on the number of random walk cut points. In Section~\ref{sec:gcb}, we review the cut-point Green's function and give another description for it via the cut-ball event. Section~\ref{sec:pt_estim} is dedicated to showing sharp one- and two-point convergence between random walk and Brownian cut points via the strong approximation. Then, we compare the cut-point event with the discrete cut-ball event, and the discrete cut-ball event with the continuous cut-ball event (Section~\ref{sec:cpcb}). For this purpose, we prepare in advance an inward coupling result for non-intersecting random walks in Section~\ref{sec:coupling}. Finally, we wrap up proofs of the main theorems Theorems~\ref{thm:vg} and \ref{thm:rw_occ} in Section~\ref{sec:measure_conv}.

{\bf Acknowledgments:} The authors thank Gregory F. Lawler for helpful and inspiring discussions. Part of this work was carried out when the second and third authors were at the University of Chicago.
YG is supported by a GRF grant from the Research Grants Council of the Hong Kong SAR (project CityU11306719).
XL is supported by National Key R\&D Program of China (No.\ 2021YFA1002700 and No.\ 2020YFA0712900) and NSFC (No.\ 12071012). DS is supported by JSPS Grant-in-Aid for Scientific Research (C) 22K03336, JSPS Grant-in-Aid for Scientific Research (B) 22H01128
and 21H00989.
}

\section{Notation and basic definitions}\label{sec:def}
In this section we {introduce basic notation, conventions and key  objects of our investigation.} 

\paragraph{Notation and conventions.} We let $\Nb:=\{0,1,2,3,\cdots\}$, $\Zb$ and $\Rb$ stand for the set of natural numbers, integers and reals resp. For $A, B \subset \Rd$, we define the distance between them by
$$\dist(A,B) = \inf_{x \in A, \, y \in B }|x - y|,$$
where $| \cdot |$ stands for the Euclidean distance. 
We extend this definition to a sequence of sets in $\Rb^d$ as follows. Suppose $A_i\subset \Rd$ for $1\le i\le m$, then we define 
\[
\dist(A_1,\ldots,A_m):=\min_{1\le i<j\le m} \dist(A_i,A_j).
\]
We allow the set to be a single point in $\Rb^d$, in which case we just write $x\in \Rb^d$ for $\{x\}$. For example, we write $\dist(x,B):=\dist(\{x\},B)$ for brevity. 
For $A \subset \Rd$ and $r>0$, we let
\begin{equation}\label{eq:sausage}
	\Theta(A,r):=\{ x\in\Rb^d: \dist(x,A)\le r \},
\end{equation}
which is the set of all points within a distance $r$ of the set $A$. We call $\Theta(A,r)$ the $r$-sausage of $A$. 
The sausage also applies to a pair of sets in $\Rb^d$:
\[
\Theta((A,B),r):=(\Theta(A,r),\Theta(B,r)) \quad \text{ for } A,B\subset \Rb^d.
\]
For two sets $A$ and $B$ in $\Rb^d$, the Hausdorff distance between $A$ and $B$ is given by
\begin{equation}\label{eq:d-H}
	d_H\left(A, B\right)=\inf \left\{r>0: A \subset \Theta(B,r) \text { and } B \subset \Theta(A,r) \right\}.
\end{equation}
For $z\in \Rb^d$ and $r>0$, we write 
\[
D_r(z)=D(z,r):=\{ x \in \Rd : |x-z| < r \},\quad B_r(z)=B(z,r)=D_r(z)\cap \Zd.
\]
Sometimes it will be more convenient to use exponential scales so that we abbreviate
\[
\Dc_r(z):=D_{e^r}(z), \quad  \Bc_r(z):= D_{e^r}(z)\cap \Zd. 
\]
We omit the dependence on $z$ from the notation when $z=0$.
We will use $\Dc$ to denote the unit disc $D(0,1)$.

We also give some notations in the discrete setting.
If $A \subset \Zd$, we let $\partial A$ and $\partial_{i} A$ be the outer and inner boundary of $A$ respectively, i.e., 
\begin{align*}
&\partial A = \{ x \in \mathbb{Z}^{d} \setminus A  :  \exists y \in A \text{ such that } |x- y| = 1 \} \\
&\partial_{i} A = \{ x \in  A  :  \exists y \in \mathbb{Z}^{d} \setminus A \text{ such that } |x- y| = 1 \}.
\end{align*}
We define $\overline{A} = A \cup \partial A$. For $x \in \mathbb{R}$, $\floo{x}$ denotes the floor function that gives as output the greatest integer less than or equal to $x$. When $x = (x_{1}, \cdots x_{d} ) \in \mathbb{R}^{d}$, we write $\floo{x} = ( \floo{x_{1}}, \cdots , \floo{x_{d}} )$. Furthermore, for $x \in \mathbb{R}^{d}$, we let $x_n=\floo{e^{n}x}$ and $x^{(n)} = e^{-n}x_n$ be the (discrete) blow-up and the discretized approximation of $x$ respectively.
Let $\eZn=e^{-n}\Zb^d$ be the grid with mesh size $e^{-n}$.
For a domain $D$ in $\Rb^d$, we let $D^{(n)}$ denote the largest connected subset of $D \cap \eZn$. 

We will often use positive finite universal (except the dimension $d$) constants which will be denoted by $c$, $c'$, $c''$, $u$ or $u'$, and whose values may change between lines. If a constant depends on a parameter we will use a bracket to indicate it. For example, $c=c(z)$ means the constant $c$ depends on $z$. If a constant has a subscript larger than $0$, then its value is fixed throughout the paper, while the subscript $0$ is reserved for constants that are only fixed within a given proof. If $(a_n)$ and $(b_n)$ are positive sequences, then we write 
\begin{itemize}
	\item {$a_n = O(b_n)$ or $a_n\lesssim b_n$ if there exists $c > 0$ such that $a_n \le c \, b_n$ for all $n$;}
	\item $a_n \asymp b_n$ if $a_n \lesssim b_n$ and $b_n \lesssim a_n$;
	\item {$a_{n}\simeq b_{n}$ if there exists $u > 0$ such that $a_{n}=b_{n}[1+O(e^{-un})]$.}
\end{itemize}
{If $(a_n)$ are real sequences, we still write $a_n=O(b_n)$ if $|a_n|=O(b_n)$.} 
If we add {subscripts} to these symbols above it means that the implied constant $c$ {or $u$} depends on the subscripts. For example, $a_n = O_z(b_n)$ means $a_n \le c(z) \, b_n$.

\paragraph{Paths.}
We let $\Pc$ denote the set of continuous (continuous-time) paths, that is, continuous mappings $\gamma : [0, t_\gamma] \mapsto \Rd$, where $t_\gamma\in (0,\infty)$ is referred to as the \emph{duration} of $\gamma$.   
Define the hitting time of $D\subset \Rb^d$ for any path $\gamma$ in $\Pc$ as follows:
\begin{equation}\label{eq:hitting}
	H_D(\gamma):=\inf \{ t \ge 0 : \gamma(t) \in D \}.
\end{equation}
We will write $H_r(\gamma)=H_{\partial \Dc_r}(\gamma)$ for simplicity.
Define the metric on $\Pc$ by 
\[
\rho(\gamma,\beta) = |t_{\gamma}-t_{\beta}| + \sup_{0\le s\le 1} |\gamma(t_{\gamma}s)-\beta(t_{\beta}s)|,
\]
under which $\Pc$ is a separable metric space (see \cite[Section 2.4]{Kara-Shre} for this).
If $\gamma_1,\gamma_2 \in \Pc$, then we define their concatenation 
\begin{equation}\label{eq:oplus0}
\gamma := \gamma_1 \oplus \gamma_2
\end{equation}
as an element of $\Pc$, such that $\gamma(t) = \gamma_1(t)$ for $t \in [0, t_{\gamma_1}]$ and 
\[
\gamma(t) =
\gamma_2(t - t_{\gamma_1}) - \gamma_2(0) + \gamma_1 (t_{\gamma_1})
\]
for $t > t_{\gamma_1}$. 
We use bars to denote pairs of sets in $\Rd$. If $\bagam = (\gamma_1, \gamma_2) \in \Pc \otimes \Pc$ and $\babet = (\beta_1, \beta_2) \in \Pc \otimes \Pc$, 
we let $\bagam \oplus \babet = (\gamma_1 \oplus \beta_1, \gamma_2 \oplus \beta_2)$. Finite paths can be reversed: if $\gamma \in \Pc$, then its reversal $\gamma^R \in \Pc$ has duration $t_\gamma$, and $\gamma^R(t) = \gamma(t_\gamma - t)$ for $t \in [0, t_\gamma]$.
We use the Brownian scaling to multiply paths by scalars, i.e., if $L > 0$, $\gamma \in \Pc$ and $\beta = L \circ\gamma$, then $t_\beta = L^2 t_\gamma$ and $\beta(L^2 s) = L\gamma(s)$ for $s \in [0, t_\gamma]$. 
For $\gamma, \beta \in \Pc$, we say $\gamma$ intersects $\beta$ if their traces intersect, i.e. 
\[
\gamma \cap \beta \ne \emptyset \iff
\exists s \in [0, t_\gamma], \, q \in [0, t_\beta] : \, \gamma(s) = \beta(q).
\]

The continuous paths will be used to represent realizations of Brownian motions. In respect to random walks, we also introduce discrete (discrete-time) paths. A discrete path $\lambda$ in $\Zb^d$ of length $\len(\lambda)\in \Nb$ is an ordered sequence of nearest-neighbor vertices in $\Zb^d$, denoted by $[\lambda(0),\lambda(1),\ldots,\lambda(\len(\lambda))]$.
We view a discrete path $\lambda$ as an element of $\Pc$ by interpolating $\lambda$ between neighboring vertices linearly such that it spends $1/d$ units of time to traverse each edge (this particular choice of $1/d$ will become clear when we compare the Brownian motion and simple random walk in Section~\ref{subsec:coup}). Thus, $\lambda$ has duration $t_{\lambda}=\len(\lambda)/d$ under this special interpolation. In this way, the previous notions also apply for discrete paths in $\Zb^d$. With a slight abuse of notation, we will use $H_D(\gamma)$ to denote $H_{D\cap \Zb^d}(\gamma)$ when $\gamma$ is obtained from a discrete path via interpolation as above.

\paragraph{Path measures. \label{par:path_measures}}

Let $\eM$ denote the set of finite positive measures on $\Pc$, where each element $\mu \in \eM$ is described by its total mass $\norm{\mu} \in (0, \infty)$ and a probability measure $\widehat{\mu} := \mu / \norm{\mu}$. All the operations that can be applied to $\Pc$ can also be applied to $\eM$ via a pushforward. We can restrict $\mu \in \eM$ to any Borel set $F \subset \Pc$ by letting $\mu|_F(\cdot) = \mu( \cdot \cap F)$, or equivalently $\mu|_F(d \gamma) = \mu(d\gamma) \, 1_F(\gamma)$. 
If $\mu \in \eM$, then we can sample ``random'' paths from it even though it is not necessarily a probability measure. {To see what we mean by sampling here, suppose that $f$ is a measurable function on $\Pc$ and we want to compute $\mu[f] := \int_{\Pc} f (\gamma)  \mu (d \gamma)$. To this end, we sample $\gamma$ from the probability measure $\widehat{\mu}$, compute $\Eb[f (\gamma)]$, and then multiply the result by $\norm{\mu}$. 
	
	If $\mu_1$ and $\mu_2$ are two measure in $\Mc$, we use $\mu_1\otimes\mu_2$ to denote the product measure on $\Pc\otimes\Pc$ such that
\begin{equation}\label{eq:otimes}
		\mu_1\otimes\mu_2((d\gamma_1,d\gamma_2))=\mu_1(d\gamma_1)\,\mu_2(d\gamma_2).
\end{equation}
We also introduce the concatenation for a special class of path measures. If $\mu_1$ and $\mu_2$ are measures supported on $\Gamma_1=\{ \gamma\in \Pc: \gamma(t_{\gamma})=x \}$ and $\Gamma_2=\{ \gamma\in \Pc: \gamma(0)=x \}$, respectively, then we 
\begin{equation}\label{eq:oplus}
	\text{use $\mu_1\oplus\mu_2$ to denote the image of $\mu_1\otimes\mu_2$ under the continuous map $(\gamma_1,\gamma_2)\mapsto\gamma_1\oplus\gamma_2$}
\end{equation}
(see \eqref{eq:oplus0} for the definition of path concatenations for paths).
The notion of $\otimes$ and $\oplus$ for path measures extends naturally to any number of measures. 
}

\paragraph{Brownian motion (BM).}
A standard Brownian motion in $\Rd$ is denoted by $W = \{ W(t) \}_{t \ge 0}$.
If $A$ is a $W$-measurable event, i.e., if $A \in \sigma(W)$, and $x \in \Rd$, then we let $\Pb^x\{A\}$ denote the probability of $A$ {assuming} $W(0) = x$. 
If $D \subset \Rd$, we denote the hitting time of $D$ for the Brownian motion $W$ by
\begin{equation}\label{eq:hitting-W}
T_D:=H_D(W) =
\inf \{ t \ge 0 : W(t) \in D \}.
\end{equation}
Abbreviate $T_r:=T_{\partial\Dc_r}(W)$.
We write $\baW = (W_1, W_2)$ for a pair of independent Brownian motions. 
Similarly, if $A \in \sigma(\baW)$ and $x, y \in \Rd$, then $\Pb^{x, y}(A)$ is the probability of $A$ given $\baW(0) = (x, y)$, i.e., given
that $W_1(0) = x$ and $W_2(0) = y$.

\paragraph{Simple random walk (SRW).}\label{para:RW}
We use $S = \{ S(t) \}_{t \ge 0}$ to denote the discrete-time simple random walk $S = \{ S(t) \}_{t \ge 0}$ in $\Zd$, where we interpolate $S$ as we did earlier for general discrete paths so that $S$ and $W$ have the same covariance. 
We use $\Pb^x$ to indicate that $S$ starts from $x$.
If $D \subset \Rd$, we define the hitting time of $D$ for the simple random walk by
\begin{equation}\label{eq:hitting-S}
\tau_D:=H_D(S) =
\inf \{ t \ge 0 : S(t) \in D \}.
\end{equation}
We write $\tau_r$ for $\tau_{\partial\Bc_r}(S)$.
Recall the Brownian scaling for paths. We now let $S^{(n)} = e^{-n} S$ stand for the simple random walk in $\eZn$ that spends $d^{-1}e^{-2n}$ units of time to traverse each edge of $\eZn$. 
Let $\baS = (S_1, S_2)$ denote a pair of independent random walks and $\baS^{(n)}$ denote the corresponding walks in $\eZn$.

\section{Preliminary facts on simple random walks and Brownian motions \label{sec:prelims}}
In this section, we first collect some standard facts about simple random walks and Brownian motions in Section~\ref{subsec:bt}, then review couplings {between} them in Section~\ref{subsec:coup}. Finally in Section~\ref{subsec:pm} we introduce {several types of path measures that we will use later.}

\subsection{Basic tools}\label{subsec:bt}
\paragraph{Intersections.}
If $\bagam =(\gamma_{1},\gamma_{2})=\overline W$ or $\bagam =\overline S$ and if the stating point of $\gamma_{1}$ is close to that of $\gamma_{2}$, then with high probability $\gamma_{1}$ intersects $\gamma_{2}$ very quickly when $d=2,3$. In this subsection, we will discuss  this kind of ``hittability'' of $\gamma_{i}$. 

We begin with the $d=2$ case, in which we have the Beurling projection theorem as follows.

\begin{proposition}[\emph{Beurling estimates:} Theorem 3.76, \cite{conf_proc}; Theorem
	2.5.2, \cite{RWintersect}]
	\label{p:Beurling} There exists $c > 0$ such that if $d=2$, $r\geq0$, and $\gamma\in\Pc$ connects $0$ to $\Dc^c_{r}$,
	then for the Brownian motion $W$,
	\[
	\Pb^{x}\{T_{r}(W)<T_{\gamma}(W)\}\leq c\sqrt{|x|/e^{r}},
	\]
	and also for the simple random walk $S$,
	\[
	\Pb^{x}\{\tau_{r}(S)<\tau_{\gamma}(S)\}\leq c\sqrt{|x|/e^{r}}.
	\]
\end{proposition}

Note that we can take the constant $c$ in Proposition \ref{p:Beurling} uniformly in $\gamma$. In contrast to this, a similar uniform estimate does not hold for $d=3$ since  with probability one a three-dimensional Brownian motion cannot intersect a line. However, thanks to the fact that the Hausdorff dimension of the trace of a three-dimensional Brownian motion is equal to 2 almost surely, 
it is a ``hittable'' set in the following sense. For $\epsilon \in (0,1)$ and $b \in (0, \infty)$, we define 

\[
Z_{r}^{\Gamma}(\gamma,\epsilon,b):=\sup_{y}\Pb^{y}\{\Gamma[0,H_{r+1}]\cap\gamma[0,H_{r+1}]=\emptyset\},
\]
where 
\begin{itemize}
\item $\gamma$ is a path from $0$ to infinity,

\item $\Pb^{y}$ stands for the probability law of a random path $\Gamma$ assuming that $\Gamma(0)=y$,

\item the supremum is taken over all $y\in\Dc_{r}$ satisfying that $\dist(y,\gamma[0,H_{r+1}])\leq be^{(1-\epsilon)r}$.

\end{itemize}
Then we have the following lemma which is an analogue of Proposition \ref{p:Beurling}  for $d=3$. 

\begin{lemma}[\emph{Freezing lemmas:} Lemmas 2.4 and 2.6, \cite{RWcuttimes}]
	\label{l:freezing}  Let $d=3$.  
        For every $b, M \in (0, \infty)$ and $\epsilon \in (0, 1)$,
	there exist finite positive constants $u, C$ such that for any $r > 0$,
	$x\in\Dc_{r}$ and $\baGam =\overline W$ or $\overline S$,
	\[
	\Pb^{x}\{Z_{r}^{\Gamma_{1}}(\Gamma_{2},\epsilon,b)\geq\exp(-u r)\}\leq C\exp(-Mr).
	\]
	(Notice that $Z_{r}^{\Gamma_{1}}(\Gamma_{2},\epsilon,b)$ is a function of $\Gamma_{2}$.)
\end{lemma}

\paragraph{Some estimates on hitting probabilities}

The next lemma gives some useful estimates on a solution of the Dirichlet problem on an annulus, which will be used repeatedly in this paper. 

\begin{lemma}[Theorem 3.18, \cite{morters2010brownian}; Propositions 6.4.1 and 6.4.2, \cite{RWintro}]
	\label{l:trans_recur} For $x\in\partial\Dc$ and $k, l >0$, we have that 
	\[
	\Pb^{x}\{T_{-l}(W)<T_{k}(W)\}=\begin{cases}
		k/(k+l) & \text{ if }d=2,\\
		(1-e^{-k})/(e^{l}-e^{-k}) & \text{ if }d=3.
	\end{cases}
	\]
	For $m<n$ and $y\in\Bc_n\setminus\overline\Bc_m$,
	\[
	\Pb^{y}\{\tau_{m}(S)<\tau_{n}(S)\}=\begin{cases}
		(n-|y|+O(m^{-1}))/(n-m) & \text{ if }d=2,\\
		m|y|^{-1}[1+O(m^{-1})] & \text{ if }d=3.
	\end{cases}
	\] 
\end{lemma}

\subsection{Couplings}\label{subsec:coup}
{ Couplings between random walks and Brownian motions play a central role in the investigation of the cut points of random walk through those of Brownian motion, which is one of the main goals of this work. In this subsection, we  introduce two types of such couplings, both with their own  advantages and drawbacks.

\paragraph{The KMT coupling or strong approximation.}
The so-called Koml\'os-Major-Tusn\'ady (KMT) coupling is arguably one of the most powerful couplings between random walks and Brownian motion.} We refer to Theorem 7.1.1 of \cite{RWintro} for a reference, which only gives the proof for $d=1,2$. One can also see Theorem 1.3 of \cite{zaitsev1998multidimensional} for a rigorous proof for any $d$ dimensions.

\begin{theorem}[KMT coupling]\label{thm:kmt}
	There exists a constant $c>0$ and a coupling $\Pb$ of the simple random walk $S$ in $\Zb^d$ and the Brownian motion $W$ in $\Rb^d$ such that for all $\lambda>0$ and each $n\in\Nb$,
	\[
	\Pb\{ \max_{0\le t\le n} |S_t-W_t|>c(\lambda+1)\log n \} < c n^{-\lambda}.
	\]
\end{theorem}

Note that the simple random walk $S$ {that we use} in this paper spends $1/d$ units of time on each edge in $\Zb^d$, so that we do not need to rescale the Brownian motion by $\frac{1}{\sqrt{d}}$. Since we will consider these two processes until the first exit time of some ball, we also need the following strong approximation, which is a corollary of the KMT coupling. We refer to Corollary 3.2 of \cite{MR2191635} for the case $d=2$, and {the $d=3$ case can be proved similarly.} 
 
 We also recall the following version for stopped processes, which will be repeatedly used in this work.
 \begin{corollary}[Strong approximation]\label{cor:sa}
 	There exists a constant ${\cal K}>0$ and a coupling $\Pb$ of the simple random walk $S$ in $\Zb^d$ and the Brownian motion $W$ in $\Rb^d$ such that 
 	\[
 	\Pb\{ \max_{0\le t\le \tau_{n+1}\vee T_{n+1}} |S_t-W_t|>{\cal K} n \} = O(e^{-10 n}).
 	\]
 \end{corollary}

\medskip

However, there is one instance where we cannot apply Corollary \ref{cor:sa} for technical reasons (see Remark~\ref{sesa} for a detailed explanation). Instead we use the Skorokhod embedding.
\paragraph{Skorokhod embedding.}\label{para:se}
We refer readers to Section 3 of \cite{RWcuttimes} for details.
Let $X^1,\ldots, X^d$ be $d$ independent one-dimensional standard Brownian motions.
Define the stopping times $\xi^j_n:=\{ t>\xi^j_{n-1}: |X^j(t)-X^j(\xi^j_{n-1})|=1 \}$ for all $n\ge 1$ with $\xi^j_0=0$.
Let $Z_n=(Z_n^1,\ldots,Z_n^d)$ be a {$d$-dimensional} process independent of $X^j$'s {satisfying that} $Z_0=(0,\ldots,0)$ and $Z_n-Z_{n-1}$ for $n\ge 1$ are independent with distribution $\Pb\{ Z_n-Z_{n-1}=e_j \} = 1/d$ for all $1\le j\le d$, where $e_j$ is the unit vector whose $j$-th exponent is equal to $1$. 
Let  
\[
W(t)=(X^1(t),\ldots, X^d(t)),\quad \quad S(t)=(X^1(\xi^1(Z_{[td]}^1)),\ldots, X^d(\xi^d(Z_{[td]}^d))).
\]
Then, $W(t)$ is a $d$ dimensional BM and $S(t)$ is a $d$-dimensional SRW. Using exponential estimates (in the beginning of Section 3 in \cite{RWcuttimes}), one can derive the following result.

\begin{lemma}[Lemma 3.2, \cite{RWcuttimes}]\label{Skoro}
	Let $W$ and $S$ be coupled as above. Then for any $\eps>0$ there exists $u>0$ such that 
	\begin{equation}\label{eq:se}
	\Pb\{ \max_{0\le t\le \tau_{n+1}\vee T_{n+1}} |S_t-W_t|>e^{(1/2+\eps)n} \} = O(e^{-e^{un}}).
	\end{equation}
\end{lemma}

Compared with Corollary \ref{cor:sa}, the benefit of Lemma \ref{Skoro} is that one can get a certain joint Markov property from this coupling, which is known to fail for the strong coupling in Corollary \ref{cor:sa}. Now, we present a version that we will use later.

Let $\epsilon, b > 0$ with $1/2+\eps<b<1$. Let $x\in\Zb^d$ with $\dist(0,x,\partial\Dc_n)>e^{bn}$. Let $V(x)$ be the event that $S$ visits $x$ only once before $\tau_n$. Define the following random times:
\begin{itemize}
	\item Let $\iota_1 = \tau_{\partial \Bc_{bn-2}(x)}$.
	\item Let $\iota_2$ be the first time after $\tau_{x}$ that $S$ hits $\partial\Bc_{bn-2}(x)$.
	\item Let $\widetilde{\iota}_1 = T_{\partial\Dc_{bn-1}(x)}$.
	\item Let $\widehat{\iota}_{1}$ be the last time before $T_n$ that $W$ visits $\partial\Dc_{(1/2+\eps)n}(x)$.
	\item Let $\widetilde{\iota}_2$ be the first time after $\widehat{\iota}_{1}$ that $W$ hits $\partial\Dc_{bn-1}(x)$.
\end{itemize}
See Figure \ref{fig:random-times} for an illustration.

\begin{figure}[!h]
		\begin{center}
			\includegraphics[scale=0.6]{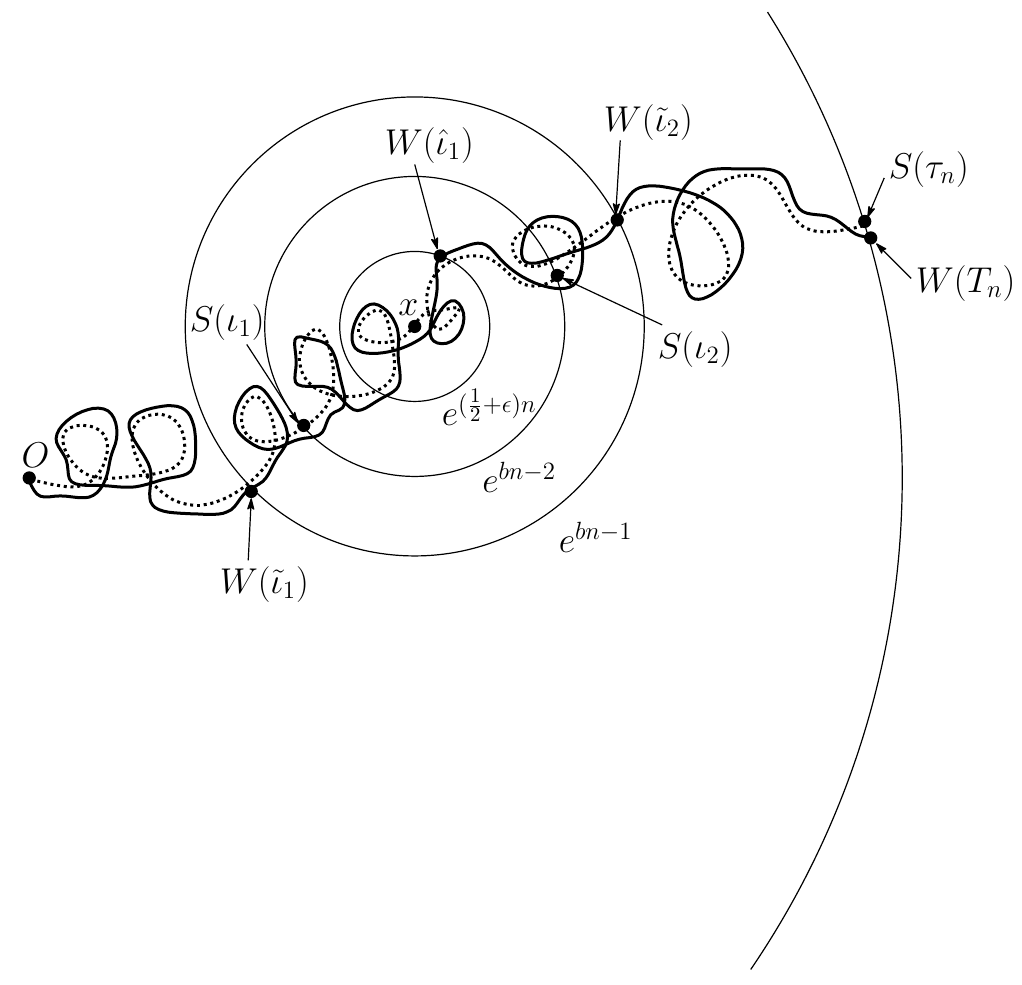}
			\caption{Illustration for random times used in Lemma \ref{lem:sem}. The thick solid curve stands for BM while SRW is depicted by the thick dotted curve.}
			\label{fig:random-times}
		\end{center}
	\end{figure}

\begin{lemma}\label{lem:sem}
	Let $W$ and $S$ be coupled under the Skorokhod embedding.
	For any $\eps, b >0$ with $1/2+\eps<b<1$, there exists $u>0$ such that the following holds for all $x$ satisfying $\dist(0,x,\partial\Dc_n)>e^{bn}$. There is an event $\Upsilon$ 
	with
	\begin{equation}\label{caution-5}
	\Pb\{ \Upsilon^c\cap V(x) \} = O(e^{-e^{un}})
	\end{equation}
	such that on $\Upsilon\cap V(x)$, $S[\iota_1,\iota_2]$ is measurable with respect to $W[\widetilde{\iota}_1, \widetilde{\iota}_2]$ and the processes $Z_n$. 
\end{lemma}
\begin{proof} 
        Take $\eps, b >0$ with $1/2+\eps<b<1$.
        Under the Skorokhod embedding,
        the event 
        \[
        E:=\Big\{ \max_{0\le t\le \tau_{n+1}\vee T_{n+1}} |S_t-W_t|>e^{(1/2+\eps)n} \Big\} 
        \] 
        satisfies $\Pb\{E^c\}\lesssim e^{-e^{-un}}$.       
	Define
	\[
	t_1 := T_{\partial D(x,e^{bn-2}+e^{(1/2+\eps)n})},
	\quad 
	t_2 := \inf\{ t\ge \widehat{\iota}_{1}: W(t)\in \partial D(x,e^{bn-2}+e^{(1/2+\eps)n}) \}.
	\]
	Then, on the event $E$, we have  
	\[
	\widetilde{\iota}_1\le t_1\le \iota_1\le \iota_2\le t_2 \le \widetilde{\iota}_2.
	\]	
	By standard exponential estimates (see the estimates above Lemma 3.1 in \cite{RWcuttimes}), we have 
	\begin{equation}\label{caution-2}
	 	\Pb \big\{ \{ e^{(2b-\eps) n}\le t_1\le e^{(2+\eps) n} \}^{c} \cap E \cap V (x) \big\} \lesssim  e^{-e^{un}}.
 \end{equation}
	In a similar way,
	we have for each $1\le j\le d$, 
	\[
	\Pb\big\{ \{ |\xi^j(Z_{[t_1d]}^j))-t_1|\ge e^{(1+\eps)n} \}\cap  E \cap V (x)  \big\} \lesssim e^{-e^{un}},
	\]
	and 
	\begin{equation}\label{caution-3}
	\Pb\big\{ \{ t_1-\widetilde{\iota}_1 \le e^{(2b-\eps)n} \} \cap  E \cap V (x)  \big\} \lesssim e^{-e^{un}}.
	\end{equation}
	The above estimates combined with the fact that $1/2+\eps<b$ show that for every $j$,
	\[
	\Pb\big\{  \{ \xi^j(Z_{[t_1d]}^j) \le \widetilde{\iota}_1 \} \cap  E \cap V (x)  \big\} \lesssim e^{-e^{un}}.
	\]
	Furthermore, by union bound, we have
	\[
	\Pb\big\{ \cup_{j=1}^d \{ \xi^j(Z_{[t_1d]}^j) \le \widetilde{\iota}_1 \} \cap V (x) \big\}  \lesssim e^{-e^{un}}.
	\]
	Similarly, one can show that
	\[
	\Pb\big\{ \cup_{j=1}^d \{ \xi^j(Z_{[t_2d]}^j) \ge \widetilde{\iota}_2 \} \cap V (x) \big\}  \lesssim e^{-e^{un}}.
	\]
	Letting 
	\begin{equation}
		\Upsilon:= \{ \widetilde{\iota}_1 \le \inf \{\xi^j(Z_{[t_1d]}^j): 1\le j\le d\} \le  \sup \{\xi^j(Z_{[t_2d]}^j): 1\le j\le d\} \le \widetilde{\iota}_2  \}\cap E,
	\end{equation}
    we know from the construction of the Skorokhod embedding that $S[\iota_1,\iota_2]$ is measurable with respect to $W[\widetilde{\iota}_1, \widetilde{\iota}_2]$ and the processes $Z_n$ on the event $\Upsilon\cap V(x)$.
	Moreover, we have 
	\begin{equation}\label{eq:iotas}
	\Pb\big\{ \Upsilon^c  \cap V (x) \big\} \lesssim e^{-e^{un}}.
	\end{equation}
	This finishes the proof.
\end{proof}

\subsection{Path measures}\label{subsec:pm}
We introduce various path measures for random walk and Brownian motion in this subsection.

\paragraph{Random walk measures}

For any two subsets $A,B$ in $\Zb^d$,
let $\Gamma_{A,B}$ be the set of (discrete) paths in $\Zb^d$ from $A$ to $B$. For $D\subseteq \Zb^d$,  we write $\Gamma^D_{A,B}$ for the set of paths in $\Gamma_{A,B}$ that stay in $D$, except maybe for the endpoints. Furthermore, let $\nu_{A,B}$ and $\nu^D_{A,B}$ be the \emph{random walk path measure} which assigns weight $(2d)^{-\mathrm{len}(\lambda)}$ to each path $\lambda$ in $\Gamma_{A,B}$ and $\Gamma^D_{A,B}$, respectively. In other words, $\nu^D_{A,B}$ is $\nu_{A,B}$ restricted to $\Gamma^D_{A,B}$.
If $A$ or $B$ is just a single point, we write $\Gamma_{x, B}$, $\Gamma_{A, y}$ and $\Gamma_{x, y}$ for $\Gamma_{\{x\}, B}$, $\Gamma_{A, \{y\}}$ and $\Gamma_{\{x\}, \{ y \}}$ respectively. These abbreviations also apply to $\Gamma^D_{A,B}$, $\nu_{A,B}$ and $\nu^D_{A,B}$. For example, $\nu^D_{x,B}$ represents $\nu^D_{\{x\}, B}$.

With the above notion, one can describe the law of a simple random walk in different ways in terms of random walk path measures.
For example, if $D$ is finite, one has 
\begin{equation}\label{eq:S-stop}
\Pb^x\{ S[0,\tau_{\partial D}]\in\cdot \}=\nu^D_{x,\partial D} (\cdot) \quad \text{ for any } x\in D.
\end{equation}
Furthermore, if $\overline{B}\subseteq D$ and $x\in D\setminus \ol B$, we can decompose each $\lambda\in \Gamma^D_{x,\partial D}$ that intersects $\ol B$ by its first and last visits to $\partial B$ as follows:
\begin{equation}\label{eq:decom1}
\lambda=\lambda_1\oplus\omega\oplus\lambda_2,
\end{equation}
where 
\begin{itemize}
	\item $\lambda_{1} = \lambda [0, a_{1}] $ with $a_{1} := H_{\ol B}(\lambda)= \inf \{ t \ge 0 :  \lambda (t) \in \overline{B} \}$,
	
	\item $\omega =  \lambda [a_{1}, a_{2}] $ with $a_{2} := H_{\ol B}(\lambda^R)= \sup \{ t \le t_\lambda :  \lambda (t) \in \overline{B} \}$,
	
	\item $\lambda_{2} = \lambda [a_{2}, t_\lambda ]$.
\end{itemize}
Since the above decomposition is unique, which defines a map from $\Gamma^D_{x,\partial D}$ to $\Gamma^{D\setminus \overline{B}}_{x,\partial B}\times \Gamma^D_{\partial B,\partial B}\times \Gamma^{D\setminus \overline{B}}_{\partial B,\partial D}$ by $\lambda\mapsto (\lambda_1,\omega,\lambda_2)$, and thanks to the strong Markov property of the simple random walk, we can write (cf. \eqref{eq:oplus} for the meaning of $\oplus$ on measures)
\begin{equation}\label{eq:decom2}
	\nu^D_{x,\partial D} 
        =\sum_{y\in \partial B}\sum_{z\in \partial B} \nu^{D\setminus \overline{B}}_{x,y}\oplus \nu^D_{y,z}\oplus \nu^{D\setminus \overline{B}}_{z,\partial D}.
\end{equation}
Therefore, to sample a simple random walk $\lambda$ from $x$ to its first hitting of $\partial D$ that is restricted to visit $\ol B$ (note that the total mass is not $1$ with this restriction), we can 
\begin{itemize}
	\item first sample a path $\lambda_1$ from $\nu^{D\setminus \overline{B}}_{x,\partial B}$,
	\item and sample another path $\lambda_2$, independently of $\lambda_1$, from $\nu^{D\setminus \overline{B}}_{\partial B,\partial D}$, 
	\item then given $\lambda_1,\lambda_2$ with ending and starting points $y,z$ on $\partial B$, we sample the third path $\omega$ from $\nu^D_{y,z}$,
	\item finally we concatenate these three paths in the way of \eqref{eq:decom1} to recover $\lambda$.
\end{itemize}
We refer the reader to the paragraph in Section~\ref{sec:def} on path measures for exact meaning of sampling here. 
This kind of decomposition is quite useful when we deal with the cut ball event later. 
In fact, one can also consider the first-entry or last-exit decomposition as well, which will decompose a path into two parts according to the first hitting time $a_1$ or last-exit time $a_2$ as above. We omit the details.

The path measure satisfies the reversibility in the following sense. If 
\begin{equation*}
\text{$\lambda$ and $\eta$ are sampled according to $\nu^{D\setminus \overline{B}}_{z,\partial D} \Big/ \| \nu^{D\setminus \overline{B}}_{z,\partial D} \|$ and $\nu^{D\setminus \overline{B}}_{\partial D,z} \Big/ \| \nu^{D\setminus \overline{B}}_{\partial D,z} \|$ respectively,}
\end{equation*}
 then $\lambda$ has the same distribution as that of $\eta^{R}$.

For $x,y$ in $D$, the Green's function for the simple random walk is defined by 
\[
G_D(x,y):=\|\nu^D_{x,y}\|.
\]
Abbreviate $G_n(x,y)=G_{\Bc_n}(x,y)$ and $G_n(y)=G_n(0,y)$.

\paragraph{Brownian measures \label{subsec:br_meas}}

We now introduce a Brownian analogue of random walk measures defined above. 
Such measures have already been introduced and intensively investigated in Lawler's book \cite{conf_proc}, especially in Section 5.2, ibid. We will briefly review some of them, and refer the reader to \cite{conf_proc} for further details and proofs.
 
We start with notation for sets of (continuous) paths. We adapt the notation used for discrete paths to the continuous case with a tilde above to emphasize this. For example, we use $\wt\Gamma^D_{A,B}$ to denote the set of (continuous) paths in $\Pc$ from $A$ to $B$ that stay in $D$, except maybe for the endpoints. Here $A,B$ and $D$ are sets in $\Rb^d$.

For $x,y\in\Rd$ and $t>0$, we define the \emph{(Brownian) bridge
	measure} $\mu_{x,y;t}$ by 
	$$ \mu_{x,y;t} = p(x,y;t) \, \widehat{\mu}_{x,y;t},$$
	where $\widehat{\mu}_{x,y;t}$ stands for the law of the Brownian bridge from $x$ to $y$ with duration $t$, and we denote the heat kernel by  $p(x,y;t):=(2\pi t)^{-d/2}e^{-|x-y|^{2}/2t}$. 
 Let $\mu_{x,y}:=\int_{0}^{\infty}\,\mu_{x,y;t}\, dt$ be the \emph{Brownian path measure} from $x$ to $y$.
 
 Next, we consider the path measure inside a domain as before.  Although the underlying domain can be made more general, we will restrict ourselves to \emph{nice} domains to avoid some boundary issues. 
 A domain $D$ in $\Rb^d$ is called nice if it is connected, bounded and its boundary $\partial D$ is piecewise analytic. 
If $x,y\in D$, then $\mu_{x,y;t}^{D}$ is the
restriction of $\mu_{x,y;t}$ to $\widetilde{\Gamma}^{D}_{x,y}$, with total mass
denoted by $p_{D}(x,y;t)$. 
For $x,y\in D$
with $x\neq y$, then the
\emph{(interior Brownian) path measure} $\mu_{x,y}^{D}:=\int_{0}^{\infty}\,\mu_{x,y;t}^{D}\, dt$
has total mass given by the \emph{Green's function} (for the Brownian motion) in $D$, defined as follows:
\[
\Gt_{D}(x,y)=\int_{0}^{\infty}\,p_{D}(x,y;t)\,dt.
\]
Abbreviate $\Gt_n(x,y)=\Gt_{\Dc_n}(x,y)$ and $\Gt_n(y)=\Gt_n(0,y)$.

	We will also make use of the Brownian \emph{interior-to-boundary} and \emph{boundary-to-boundary} excursion measures.   
For any $x\in D$,
	let $\mu^D_{x,\partial D}$ denote the probability measure of $W[0,T_{\partial D}]$,
	with $W$ starting at $x$. 
Since $D$ is a nice domain, we can write $\mu^D_{x,\partial D}=\int_{\partial D}\,\mu_{x,y}^{D}\,\sigma(dy)$,
	with $\sigma(dy)$ being the surface measure on $\partial D$ (area if $d = 3$ and length if $d = 2$), 
	where $\mu_{x,y}^{D}$ for $x\in D$ and $y\in \partial D$ denotes a measure supported on $\tilde\Gamma^D_{x,y}$ with total mass given by the usual Poisson kernel. The normalized probability measure $\hat\mu_{x,y}^{D}$ is the law of Brownian motion conditioned to exit $D$ “at $y$”.
	By taking the reverse of path from $\mu_{x,y}^{D}$ we get the \emph{ boundary-to-interior excursion measure}. In a similar fashion we can also define the  \emph{boundary-to-boundary excursion measure}.

Recall that $\gamma^R$ stands for the reversal of $\gamma$. Then,
we have $(\mu_{x,y}^{D})^R=\mu_{y,x}^{D}$ for all the Brownian
measures defined thus far, where $(\mu_{x,y}^{D})^R$ is the pushforward of $\mu_{x,y}^{D}$ under the ``reverse'' function $(\cdot)^R$. This allows for last-exit decomposition
formulas, which are the main reason for the path measure formalism.
Some of these formulas can be found in \cite[Section 5.2] {conf_proc}.
For our goals it is sufficient to point out, that, thanks to the strong
Markov property of Brownian motion, if $B$
and $D$ are some nice domains with $\overline{B}\subset D$,
then we can do the following decompositions (recall \eqref{eq:oplus}): 
\begin{equation}\label{eq:dec1}
\mu_{x,\partial D}^{D}=\int_{\partial B} \,\mu_{x,y}^{B}\oplus\mu_{y,\partial D}^{D}\,\sigma(dy)\  \text{ for any } x\in B,
\end{equation}
\begin{equation}\label{eq:dec-last}
	\mu_{x,\partial D}^{D}=\int_{\partial B} \,\mu_{x,y}^{D}\oplus\mu_{y,\partial D}^{D\setminus \overline{B}}\,\sigma(dy)\  \text{ for any } x\in B,
\end{equation}
and 
\begin{equation}\label{eq:dec2}
	\mu^D_{x,\partial D}\mid_{\gamma\cap B\neq\emptyset}
	\,=\int_{\partial B} \,\int_{\partial B} \, \mu^{D\setminus \overline{B}}_{x,y}\oplus \mu^D_{y,z}\oplus \mu^{D\setminus \overline{B}}_{z,\partial D} \,\sigma(dy)\,\sigma(dz) \  \text{ for any } x\in D\setminus \overline{B}.
\end{equation}
The above three formulas can be obtained by using the first-entry, last-exit, and first-entry and last-exit decomposition, respectively. 
Any other formulae that we use in the sequel follows immediately by
iterating these observations, and, perhaps, by also using the reversibility
of the interior Brownian measures.
\bigskip

Next, we will show that although the total mass of point-to-point path measure in the plane is infinity, it will become finite if one restricts it to the collection of non-disconnecting or non-intersecting paths.
\begin{lemma}
	\label{l:paths_discon_beurling} There exist positive constants $c,u$ such that the
	following is true. Suppose that $d=2$, $x,y,z\in\Rb^2$ and $2|x|<r:=|y|\land|z|$. Then,
	\begin{equation}
		\mu_{x,y}\{\gamma\in\widetilde{\Gamma}_{x,y}:\gamma\text{ does not disconnect }0\text{ from }z\}\leq c(r/|x|)^{-u}.\label{eq:path_measure_disconnection}
	\end{equation}
	This is also true if $\gamma'\in\widetilde\Gamma_{0,z}$ and we have an additional
	restriction to $\gamma$ such that $\gamma\cap\gamma'=\emptyset$. Moreover, it still holds if we transfer to the discrete setting, that is, $\nu_{x,y}$ in place of $\mu_{x,y}$ with $x,y,z\in \Zb^2$.
\end{lemma}

\begin{proof}
	We will prove \eqref{eq:path_measure_disconnection} in the continuous setting for an illustration, and the other case can be proved in a similar way.
	We assume without loss of generality that $|x|=1$ and $2<r=|z|<|y|$. For each path $\gamma$ in $\widetilde{\Gamma}_{x,y}$, we decompose it into crossings between $\partial B_2$ and $\partial B_r$, that is, 
	let $s_0=t_0=0$, and for $k\ge 1$,
	\[
	s_k=\inf\{ t\ge t_{k-1}: \gamma(t)\in \partial B_r \}, \quad
	t_k=\inf\{ t\ge s_{k-1}: \gamma(t)\in \partial B_2 \}.  
	\]
	Let $U_k$ be the set of paths in $\widetilde{\Gamma}_{x,y}$ with $2k-1$ crossings, i.e., $s_k<\infty$ and $t_k=\infty$. For some universal constant $c>0$, each crossing has a probability at most $r^{-c}$ not to disconnect $0$ from $z$. Moreover, the last part of the decomposition $\gamma[s_k,t_{\gamma}]$ is a path in $\wt\Gamma^{\overline{B_r}^c}_{\partial B_r,y}$, which has total mass bounded by $1$. Therefore, by the strong Markov property,
	\[
	\mu_{x,y}\{ U_k \}\le r^{-c(2k-1)}.
	\]
	Therefore,
	\[
	\mu_{x,y}\{\gamma\in\widetilde{\Gamma}_{x,y}:\gamma\text{ does not disconnect }0\text{ from }z\}
	\le \sum_{k\ge 1}\mu_{x,y}\{ U_k \} \le \sum_{k\ge 1}  r^{-c(2k-1)}.
	\]
	This concludes the proof of  \eqref{eq:path_measure_disconnection}.
	
	If one replaces the non-disconnection requirement by non-intersection with another path $\gamma'$, one can easily conclude the proof by noting that $\gamma\cap\gamma'=\emptyset$ implies that
	$\gamma$ does not disconnect $0$ from $z$.
\end{proof}

\section{Non-intersecting paths \label{sec:nonint}}
In this section, we discuss the probability measure of two random walks (resp.\ two Brownian motions) conditioned not to intersect each other and introduce various estimates, in particular {\it separation lemmas} (see the paragraph above Lemma \ref{lem:sep-SRW} for more explanations), regarding these objects.
Note that from this section onwards, we always call this pair of walks non-intersecting (NI) random walks (resp.\  Brownian motions), abbreviated as NIRW's (resp.\  NIBM's).

\subsection{NIRW's}

We first discuss NIRW's. Let $\Gamma_m$ be the set of paths $\gamma$ such that $\Pb\{ S[0, \tau_m]=\gamma \}>0$ where $S$ is the simple random walk started from $0$. Define the set of NI pairs of such paths:
\begin{equation}\label{eq:tXr}
	\Xc_{m}=\, \{\bagam=(\gamma^{1},\gamma^{2})\in\Gamma_m\times\Gamma_m: \gamma^{1}(s)\neq \gamma^{2}(t) \text{ for all } (s,t)\neq (0,0)\}.
\end{equation}
For any $w\in \Zb^d$, write $ \Xc_{r}(w):=w+ \Xc_{r}$ for the translation of NI paths in $\Xc_{r}$ to the point $w$. 

Define the NI event for RW's started from the origin as
\begin{equation}\label{eq:tilde-A'}
	A_{m}:=\{  (S^{1}[0,\tau_{m}],S^{2}[0,\tau_{m}])\in \Xc_{m} \},
\end{equation} 
where $S^i$ above is a random walk started from $0$ for each $i=1,2$.
For $m>l>0$, define the NI event with initial configuration $\overline\gamma\in \Xc_l$ as
\begin{equation}\label{eq:tilde-A}
A_{m}(\overline \gamma):= \{S^{1}[0,\tau_{m}]\cap(S^{2}[0,\tau_{m}]\cup \gamma^2)=(S^{1}[0,\tau_{m}]\cup \gamma^{1})\cap S^{2}[0,\tau_{m}]=\emptyset\},
\end{equation}
where $S^i$ is a random walk started from $\gamma^i(t_{\gamma^i})$.

We now discuss a very important tool, \emph{the separation lemma},  that will be used a lot of times when we analyze NIRW's. Roughly speaking, it says that if two independent random walks do not intersect each other, then with uniformly positive probability they will be ``well-separated'' at the end. To be more precise, we first define the ``quality of separation'' as follow:
\begin{equation}\label{eq:RW-quality}
\Delta_m:=e^{-m} \min_{i=1,2} \dist\big( S^i(\tau_m),S^{3-i}[0,{\tau_m}] \big).
\end{equation}
Then, we say {the pair $(S^{1}[0,\tau_{m}], S^{2}[0,\tau_{m}] )$} is \emph{well-separated} if $\Delta_m\ge 1/10$. 

\begin{lemma}[Separation lemma for NIRW's: Proposition 2.1, \cite{MR3787377}]\label{lem:sep-SRW}
	There exists a universal constant $c>0$ such that for all $m\ge 10 l$ and any initial configuration $\overline\gamma\in {\Xc}_{l}$,
	\begin{equation}
		\Pb\{ \Delta_m\ge 1/10 \mid A_{m}(\overline \gamma) \}\ge c.
	\end{equation}
\end{lemma}

We also have the sharp estimate for the probability of the NI event $A_{m}(\overline \gamma)$.

\begin{proposition}[{Corollary 4.2, \cite{2sidedwalk}}]
	\label{prop:sharp-rw}
	For each $l\ge 0$ and $\overline\gamma\in {\Xc}_{l}$, there exists $0<q(\overline\gamma)<\infty$ such that for all $m\ge 10 l$,
	\[
	\Pb\bigl({A}_{m}(\overline \gamma)\bigr)\simeq q(\overline\gamma)e^{-\xi(m-l)}.
	\]
	In particular, it is true when $\ol\gamma$ reduced to the origin, that is, for some $0<q<\infty$, 
	\[
	\Pb\bigl({A}_{m}\bigr)\simeq q\, e^{-\xi m}.
	\]
\end{proposition}

	The referenced paper \cite{2sidedwalk} has error bound $O(e^{-u\sqrt{m}})$ for $d=2$, i.e., $\Pb\bigl({A}_{m}(\overline \gamma)\bigr)= q(\overline\gamma)e^{-\xi(m-l)}[1+O(e^{-u\sqrt{m}})]$, in which the authors referred to Theorem 1.2 of \cite{2dBMnonint}
	in establishing Proposition 2.9 of \cite{2sidedwalk}; by referring to Proposition 4.3 from \cite{analyticity}
	instead, we get the desired exponential convergence rate.

\subsection{NIBM's from the origin}\label{subsec:qi}

We now discuss non-intersecting Brownian motions (abbreviated as NIBM's below) started from $0$. Let $\widetilde\Gamma_r$ be the set of paths $\gamma$ such that $\gamma(0)=0$, $\gamma(t_{\gamma})=e^r$ and $|\gamma(t)|<e^r$ for all $0\le t<t_{\gamma}$. Define the set of NI pairs of paths:
\begin{align*}
	\widetilde\Xc:=\, \{\bagam=(\gamma^{1},\gamma^{2})\in\widetilde\Gamma_0\times\widetilde\Gamma_0: \gamma^{1}(s)\neq \gamma^{2}(t) \text{ for all } (s,t)\neq (0,0) \}.
\end{align*}
We now define the quasi-invariant probability measure $\Qf$ on $\widetilde\Xc$,
which was introduced in \cite{2dBMnonint} for $d=2$, and in \cite{MR1645225}
for $d=3$. We will describe $\Qf$ and its properties using stronger
results from \cite{analyticity,BMinvar}. Informally, $\Qf$ is the
distribution {of} a pair of independent Brownian motions from $0$ to
$\partial\Dc$ ``conditioned to avoid intersection''. Since we cannot
condition on events of zero probability, we define $\Qf$ as a limit by using the procedure given below.

Let $\baK=(K^{1},K^{2})$ be a pair of two compact sets in $\Rb^d$ (either $K^1 \cap K^2 \neq \empty$ or not). Let $\overline x=(x^{1},x^{2})$ be a pair of points in $\Rb^d$. We call $(\baK,\overline x)$ an \emph{initial configuration}.
For any $z\in \Rb^d$ and $r\in \Rb$ such that $(K^1\cup K^2\cup\{x^1,x^2\})\subset \Dc_r(z)$, we define the NI event
\begin{equation}\label{eq:Ar}
	\widetilde A_{z,r}(\baK,\overline x)=\{W^{1}[0,T_{\Dc_r(z)}]\cap(W^{2}[0,T_{\Dc_r(z)}]\cup K^{2})=(W^{1}[0,T_{\Dc_r(z)}]\cup K^{1})\cap W^{2}[0,T_{\Dc_r(z)}]=\emptyset\},
\end{equation}
where $W^1$ and $W^2$ are independent Brownian motions started from $x^1$ and $x^2$ respectively. We abbreviate $\widetilde A_{r}(\baK,\overline x)=\widetilde A_{0,r}(\baK,\overline x)$.
For simplicity we assume that $z=0$, $r>2$, $(K^1\cup K^2)\subset \Dc_2$ and $x^i\in K^{i}\cap\partial\Dc$ below. The general results can be obtained by using translation- and scaling-invariance of Brownian motion.
We only consider initial configurations such that $\widetilde A_{2}(\baK,\overline x)$ has positive probability.
We also let $\widetilde A_{r}(\bagam)=\widetilde A_{r}(\bagam,\overline x)$ for $\bagam\in\widetilde\Xc$, where $\overline x=(x^{1},x^{2})$
with $x^{i}=\gamma^{i}\cap\partial\Dc$. We define $Q_{r}(\baK,\overline x)$ as the distribution of
$e^{-r}\circ\overline W[0,T_{r}]$ conditionally on $\widetilde A_{r}(\baK,\overline x)$, and $Q_{r}(\bagam)$ as
the distribution of $\bagam_{r}$ conditionally on $\widetilde A_{r}(\bagam)$, where $\bagam_{r}=e^{-r}\circ\bigl((\bagam\oplus\overline W)[0,T_{r}]\bigr)$.
We will abuse notation: if an operation other than the shrinking and
elongation of $\bagam$, which we denote by $\bagam_{r}$, is defined
for curves, then we also define it via a pushforward for distributions
on pairs of curves, e.g., if $\bagam$ has distribution $Q$, then $Q[0,T_{-r}]$
is the distribution of $\bagam[0,T_{-r}]$. Quasi-invariant distribution
$\Qf$ is defined to be such a probability measure on $\widetilde\Xc$,
that for some $u>0$, 
\begin{equation}\label{eq:qi_bm_definition}
	\norm{Q_{r}(\baK,\overline x)[T_{-r/2},T_{0}]-\Qf[T_{-r/2},T_{0}]}_{{\rm TV}}=O(e^{-ur}),
\end{equation}
uniformly in $(\baK,\overline x)$. 

We also have a separation {lemma} for NIBM's, similar to that of NIRW's. Define the separation quality by 
\begin{equation}\label{eq:quality}
\widetilde \Delta_r:=e^{-r} \min_{i=1,2} \dist\big( W^i(T_r),W^{3-i}[0,T_r] \big).
\end{equation}
Then, we say it is \emph{well-separated} if $\widetilde\Delta_r\ge 1/10$. 

\begin{lemma}[Separation lemma for NIBM's: Lemma 3.4, \cite{2dBMnonint}; Lemma 3.2, \cite{BMinvar}]\label{lem:sep-BM}
	There exists a universal constant $c>0$ such that for any initial configuration $(\baK,\overline x)$ and any $r>0$ such that $(K^1\cup K^2)\subseteq \Dc_{r/2}$,
	\begin{equation}
		\Pb\{ \widetilde\Delta_r\ge 1/10 \mid \widetilde A_r(\baK,\overline x) \}\ge c.
	\end{equation}
\end{lemma}

Analogous to Proposition~\ref{prop:sharp-rw}, it is also known that there is a bounded
function $\widetilde q:\widetilde\Xc\to(0,\infty)$ such that 
\[
\Pb\bigl(\wt A_{r}(\bagam)\bigr)\simeq \widetilde q(\bagam)\, e^{-\xi r} \quad \text{ for all }   \bagam\in\widetilde\Xc.
\]
 Similarly, for any initial configuration $(\baK,\overline x)$, there exists $0<\wt q(\baK,\overline x)<\infty$ such that
\begin{equation}\label{eq:An}
	\Pb\bigl(\widetilde A_{r}(\baK,\overline x)\bigr)\simeq \widetilde q(\baK,\overline x)\, e^{-r\xi}.
\end{equation}
We also describe a version for excursions which will be used in Section \ref{sec:gcb}. 

\begin{lemma}\label{lem:excursion}
	Let $W^1$ and $W^2$ be two independent standard Brownian motions. Let $s+1\le r$. Denote by $\sigma_s$ the last visit of $\Dc_s$ by the Brownian motion before $T_r$. Then,
	\[
	\Pb\{ W^1[\sigma_s,T_r] \cap W^2[\sigma_s,T_r]=\emptyset \} \asymp (r-s)^{2(3-d)} e^{-\xi (r-s)}.
	\]
\end{lemma}
\begin{proof}
	When $d=2$, it can be shown by following the proof of Theorem 3.1 in \cite{MR1901950}. Although the referred Theorem 3.1 actually deals with the case that a packet of $2$ excursions do not intersect a packet of ``$\lambda$'' excursions, as already explained in Section 7 of \cite{MR1901950}, it holds for all general cases. One excursion does not intersect another excursion can be viewed as a special case, which corresponds to the exponent $\xi(1,1)$ in \cite{MR1901950}, we thus get Lemma~\ref{lem:excursion} when $d=2$. 
	
	When $d=3$, we only need to show that 
	\begin{equation}\label{eq:exc3}
	\Pb\{ W^1[T_s,T_r] \cap W^2[T_s,T_r]=\emptyset \}\lesssim 
	\Pb\{ W^1[\sigma_s,T_r] \cap W^2[\sigma_s,T_r]=\emptyset \}
	\lesssim \Pb\{ W^1[T_{s+1},T_r] \cap W^2[T_{s+1},T_r]=\emptyset \}.
	\end{equation}
	The first inequality is trivial. For the second inequality, we note that the law of  $W^i[\sigma_s,T_r]$ from its first visit of $\partial \Dc_{s+1}$ to first visit of $\partial \Dc_{r}$ is just a Brownian motion started uniformly from $\partial \Dc_{s+1}$ conditioned to hit $\partial \Dc_{r}$ before $\partial \Dc_{s}$. Therefore, by Lemma~\ref{l:trans_recur},
	\begin{align*}
	\Pb\{ W^1[\sigma_s,T_r] \cap W^2[\sigma_s,T_r]=\emptyset \}
	&\le \left( \frac{e^{-s}-e^{-(s+1)}}{e^{-s}-e^{-r}} \right)^{-2} \Pb\{ W^1[T_{s+1},T_r] \cap W^2[T_{s+1},T_r]=\emptyset \} \\
	&\le (1-e^{-1})^{-2}\, \Pb\{ W^1[T_{s+1},T_r] \cap W^2[T_{s+1},T_r]=\emptyset \}.
	\end{align*}
This concludes the proof of \eqref{eq:exc3}, and thus implies the lemma.
\end{proof}

Note that the distribution of $W^i[\sigma_s,T_r]$ is $\widehat\mu^{\Dc_r\setminus\overline\Dc_s}_{\partial\Dc_s,\partial\Dc_r}$. Moreover, using the last-exit decomposition \eqref{eq:dec-last}, we have 
\begin{equation*}
	\mu_{0,\partial\Dc_r}^{\Dc_r}=\int_{\partial \Dc_s} \,\mu_{0,y}^{\Dc_r}\oplus\mu_{y,\partial \Dc_r}^{\Dc_r\setminus\overline\Dc_s}\,\sigma(dy).
\end{equation*}
Note that $\mu_{0,\partial\Dc_r}^{\Dc_r}$ has total mass $1$ since it can be viewed as the law of a Brownian motion started from $0$ stopped upon reaching $\partial\Dc_r$, similar to the discrete version \eqref{eq:S-stop}. Moreover, by rotation invariance, for any $y\in\partial\Dc_s$,
\[
\|\mu_{0,y}^{\Dc_r}\|=\widetilde G_{\Dc_r}(0,(e^s,0))\asymp (r-s)^{-2(d-3)} e^{-2(d-2)s}.
\]
 Therefore, the total mass of boundary-to-boundary excursion measure has the following up-to-constants estimate:
 \[
\|\mu^{\Dc_r\setminus\overline\Dc_s}_{\partial\Dc_s,\partial\Dc_r}\|=\widetilde G_{\Dc_r}(0,(e^s,0))^{-1} \asymp (r-s)^{2(d-3)} e^{2(d-2)s}.
\] 
Combining these observations with Lemma~\ref{lem:excursion}, we obtain the following result.
\begin{lemma}\label{lem:ex-ni}
	For any $s+1\le r$, we have
	\[
	\mu^{\Dc_r\setminus\overline\Dc_s}_{\partial\Dc_s,\partial\Dc_r}\otimes \mu^{\Dc_r\setminus\overline\Dc_s}_{\partial\Dc_s,\partial\Dc_r} \{ \gamma^1\cap\gamma^2=\emptyset \} \asymp e^{2(d-2)s} e^{-\xi (r-s)}.
	\]
\end{lemma}

\subsection{NIBM's from infinity}\label{subsec:qi-infty}
In this subsection, we briefly review the results on NIBM's from infinity, which is a variant of the results in the previous section. Most results in this subsection can be found in Section 3 of \cite{mink_cont}, so we will be quite brief below.

Let $r\ge 1$. Suppose $\overline K=(K_1,K_2)$ is a pair of compact sets in $\Dc^c_r$ and $\overline x=(x_1,x_2)$ is a pair of points such that $x_i\in K_i$. Let $W_i$ be the Brownian motion started from $x_i$. Let $Q^*_r(\baK,\overline x)$ be the probability measure of $\overline W[T_{r/2},T_{0}]$, conditioned on the event that 
\[
	\widetilde A^*_{r}(\baK,\overline x):=\Big\{W_{1}[0,T_{0}]\cap(W_{2}[0,T_{0}]\cup K_{2})=(W_{1}[0,T_{0}]\cup K_{1})\cap W_{2}[0,T_{0}]=\emptyset\Big\}\bigcap \Big\{ T_0(W_1),T_0(W_2)< \infty \Big\}.
\]
Let $\wt\Gamma_{\infty,0}$ be the set of paths that started from infinity and stopped when they reach $\partial\Dc$. 
Define the set of non-intersecting pairs in $\wt\Gamma_{\infty,0}$ as follows
\begin{align*}
	\widetilde\Xc^*:=\, \{\bagam=(\gamma_{1},\gamma_{2})\in\wt\Gamma_{\infty,0}\times\wt\Gamma_{\infty,0}: \gamma_{1} \cap \gamma_{2}=\emptyset \}.
\end{align*}
Then, there exists a probability measure $\Qf^*$ supported on $\widetilde\Xc^*$ such that for some $u>0$,
\begin{equation}\label{eq:qi_bm_definition-inv}
	\norm{ Q^*_{r}(\baK,\overline x)[T_{r/2},T_{0}]- \Qf^*[T_{r/2},T_{0}]}_{{\rm TV}}=O(e^{-ur}),
\end{equation}
uniformly in $(\baK,\overline x)$. We call $\Qf^*$ a quasi-invariant measure on NIBM's from infinity.

\medskip
We finish this short section by recording a useful reverse separation lemma for NIBM's. We use the same notation $\widetilde \Delta_r$ for quality here, but one should note that the associated Brownian motions are from outside to inside.
\begin{lemma}[Reverse separation lemma for NIBM's]\label{lem:rev-sep-BM}
	There exists a universal constant $c>0$ such that for any initial configuration $(\baK,\overline x)$ introduced previously with $r\ge 1$,
	\begin{equation}
		\Pb\{ \widetilde\Delta_0\ge 1/10 \mid \widetilde A^*_r(\baK,\overline x) \}\ge c.
	\end{equation}
\end{lemma}
It can be either proved in a similar way to that of Lemma~\ref{lem:sep-BM}, or simply using the invariance of Brownian motion under inversion (see Proposition 2.2 of \cite{mink_cont}). A corresponding reverse version for NIRW's will be given in Lemma~\ref{sep} later.

\section{Moment bounds on the number of random walk cut points}\label{sec:moment}
In this section we derive some  moment bounds on the number of random walk cut points.

Recall that $B_n$ is the discrete ball of radius $n$ around $0$ in $\Zb^d$. Let $z_1,\dots,z_k$ be $k$ points in $B_n$ which are not necessarily distinct.
We write $\vec z=(z_1,\ldots,z_k)$. Define $z_0=0$, $z_{k+1}=\partial B_n$ and 
\[
r_i^{\vec z}= d(z_i)\wedge |z_{i-1}-z_i| \wedge |z_i-z_{i+1}| \quad \text{ with } d(z_i)=\dist (z_i,\partial B_n).
\]
let $\Pi_k$ be the symmetric group on $\{1,\ldots,k\}$ and for each element $\pi\in\Pi_k$ we write $\pi(\vec z) = (z_{\pi(1)},\ldots,z_{\pi(k)})$ for the corresponding permutation of $\vec z$.
Suppose $S$ is a simple random walk started from $0$ and let $\iota_n$ be the exit time of $B_n$ by $S$. Denote the set of cut points of $S[0,\iota_n]$ by $\Af_n$. 
Let $M_n$ be the cardinality of $\Af_n$.
We first derive an upper bound on $k$-point function\footnote{In this work we actually only need one- and two-point functions, but we still write down this proposition for general $k$ for future reference	.} of random walk cut points. As a corollary, we give an upper bound for higher moments of $M_n$ for $d=3$.
\begin{proposition}\label{prop:eb}
	For $d=2,3$, there exists a constant $C=C(d)>0$ such that for all $\vec z=(z_1,\ldots,z_k)$,
		\begin{equation}\label{eq:eb-1}
		\Pb( \{ z_1,\ldots, z_k\} \subseteq \Af_n ) \le C^k \sum_{\pi\in\Pi_k}\prod_{i=1}^{k} G_{B_n}(z_{\pi(i-1)},z_{\pi(i)}) \big(r_{\pi(i)}^{\pi(\vec z)}\big)^{-\xi}.
		\end{equation}
	Moreover, if $d=3$, then
		\begin{equation}\label{eq:eb-2}
		\sum_{z_k\in B_n} G_{B_n}(z_{k-1},z_k) \big(r_{k-1}^{\vec z}\big)^{-\xi} \big(r_{k}^{\vec z}\big)^{-\xi} \le C (|z_{k-1}-z_{k-2}|\wedge d(z_{k-1}))^{-\xi} n^{2-\xi}\mbox{, and }
		\end{equation}
		\begin{equation}\label{eq:eb-3}
		\Eb[ M_n^k ] \le C^k k! (n^{2-\xi})^k.
		\end{equation}
\end{proposition}
\begin{proof}
	We start with \eqref{eq:eb-1}. We only deal the case when $z_1,\ldots,z_k$ are all distinct, otherwise we can use a simple induction argument as the second part of the proof of Proposition 5.5 in \cite{MR2683633}.
	It suffices to show that 
	\begin{equation}\label{eq:eb-4}
	\Pb( E_n(\vec z) ) \le  C^k \prod_{i=1}^{k} G_{B_n}(z_{i-1},z_{i}) \big(r_i^{\vec z}\big)^{-\xi},
	\end{equation}
	where 
	\begin{equation}\label{eq:E_ndef}
	E_n(\vec z):=\Big\{\{ z_1,\ldots, z_k\} \subseteq \Af_n\Big\} \bigcap\Big\{\mbox{$S$ visits $z_1,\ldots,z_k$ successively}\Big\}.
	\end{equation}
	From the path decomposition point of view, $\Pb( E_n(\vec z) )$ is bounded by the multiplication of total mass of the following paths:
	\begin{itemize}
		\item For each $1\le i\le k$, let $(\gamma_i^1,\gamma_i^2)$ be a pair of NIRW's from $z_i$ to $\partial B(z_i,r_i^{\vec z}/4)$ with total mass $\asymp\big(r_i^{\vec z}\big)^{-\xi}$.
		\item Denote the endpoint of $\gamma_i^j$ by $x_i^j$. Let $\omega_1$ be the path from $0$ to $x_1^1$ with total mass $G_{B_n}(0,x_1^1)$; $\omega_i$ be the path from $x_{i-1}^2$ to $x_i^1$ with total mass $G_{B_n}(x_{i-1}^2,x_i^1)$ for $2\le i\le k$; $\omega_{k+1}$ be the simple random walk from $x_k^2$ to the exit of $B_n$ with total mass $1$.
	\end{itemize}
In the above procedure, we can concatenate $\Big(\oplus_{i=1}^k \big(\omega_i\oplus [\gamma_i^1]^R\oplus \gamma_i^2\big)\Big)\oplus \omega_{k+1}$ to recover $S[0,\iota_n]$ restricted to an event that contains $E_n(\vec z)$. Here we use $\oplus_{i=1}^k$ to denote consecutive concatenations.
By the Harnack principle (see \cite[Theorem 6.3.9]{RWintro} for this), $G_{B_n}(0,x_1^1)\asymp G_{B_n}(0,z_1)$ and $G_{B_n}(x_{i-1}^2,x_i^1)\asymp G_{B_n}(z_{i-1},z_i)$ for all $2\le i\le k$. Then, \eqref{eq:eb-1} follows immediately.

Once we get \eqref{eq:eb-1}, the inequalities \eqref{eq:eb-2} and \eqref{eq:eb-3} follow from the same summing argument as in the proof of Theorem 5.6 of \cite{MR2683633} and Theorem 8.4 of \cite{LERW3exp}, which deal with the moments bounds of 2D and 3D loop-erased random walk (LERW) respectively. 
\end{proof}

\begin{remark}
For LERW, there is an exponent $\alpha=\alpha_d<1$ for $d=2,3$ associated with the escape probability (of a simple random walk started from the tip of an independent LERW), which plays the same role as $\xi=\xi_d$ here. In the summing argument in Theorem 5.6 of \cite{MR2683633} and Theorem 8.4 of \cite{LERW3exp} for the moment bounds of LERW, the fact that  $\alpha_d<1$ is essential.
 However, for our case, we only have $\xi_3<1$ but $\xi_2=5/4>1$. This is the reason that we only get the inequalities \eqref{eq:eb-2} and \eqref{eq:eb-3} for $d=3$. 
 \end{remark}

We derive the up-to-constants estimate for the one-point in the following lemma. Recall $\eta=\xi+d-2$ in \eqref{eq:eta}.
 \begin{lemma}\label{lem:biop}
	For all $z\in B_n$ with $r:=\dist(0,z,\partial B_n)$,
	\[
	\Pb( z\in\Af_n ) \asymp r^{1-\xi} n^{1-d} 1_{|z|\ge n/2} + r^{-\eta} (\log(n/r))^{d-3} 1_{0<|z|<n/2}.
	\]	
\end{lemma}
\begin{proof}
	Assume $|z|\ge n/2$.
	For the upper bound, by applying \eqref{eq:eb-1} to $k=1$, we have 
	\[
	\Pb( z\in\Af_n ) \lesssim G_{B_n}(0,z) r^{-\xi}.
	\]
	The upper bound follows from the estimate $G_{B_n}(0,z)\lesssim r n^{1-d}$.
		
	As for the lower bound, the key observation is that we can construct the random path $S[0,\iota_n]$ to satisfy $z\in\Af_n$ in the following way:
	\begin{itemize}
		\item Let $(\gamma_1,\gamma_2)$ be a pair of well-separated non-intersecting random walks (see Lemma~\ref{lem:sep-SRW}) from $z$ to $\partial B(z,r/2)$ with total mass $\asymp r^{-\xi}$.
		\item Let $\eta_1$ be the SRW from $0$ to its first hitting of $B(z,r/2)$ satisfying that $\eta_1\subseteq B_n$ and the distance between the endpoints of $\eta_1$ and $\gamma_1$ is of order $r$, which has total mass $\asymp (r/n)^{d-1}$.
		\item Let $\omega$ be the path connecting the endpoints of $\eta_1$ and $\gamma_1$ which stays in a local ball of radius of order $r$ centered at the endpoint of $\gamma_1$. Then, the total mass of $\omega$ is $\asymp r^{2-d}$;
		\item Let $\eta_2$ be the simple random walk from the endpoint of $\gamma_2$ to $\partial B_n$, which has total mass $\asymp 1$.
		\item The concatenation $\eta_1\oplus \omega\oplus [\gamma_1]^R\oplus \gamma_2\oplus \eta_2$ recovers $S[0,\iota_n]$.
	\end{itemize}
Therefore, the total mass of $S[0,\iota_n]$ restricted to the event $z\in\Af_n$ is given by the multiplication of the total mass of these five pieces which leads to the conclusion.
There is a caveat that
on the event $z\in\Af_n$ there are still some mild restrictions on the paths $\eta_1,\omega$ and $\eta_2$. However, as one can easily check by using the separation lemma, these further restrictions will only change the mass up to a constant. 

Next, assume $|z|< n/2$. We will see that in this case the upper bound given by \eqref{eq:eb-1} is no longer tight when $d=2$ and $z$ is close to $0$. We now construct a path which satisfies $z\in\Af_n$:
\begin{itemize}
	\item Let $(\gamma_1,\gamma_2)$ be a pair of well-separated non-intersecting random walks from $z$ to $\partial B(z,r/2)$, whose total mass is $\asymp r^{-\xi}$.
	\item Let $\eta_1$ be the path from $0$ to the endpoint of $\gamma_1$ staying in a well-chosen tube (so that it avoids $\gamma_2$ by definition) of size of order $r$ with total mass $\asymp r^{2-d}$,
	\item Let $\eta_2$ be the SRW from the endpoint of $\gamma_2$ to its first hitting of $\partial B(z,2r)$ staying in a well-chosen tube (so that it avoids $\gamma_1$ and $\eta_1$ by definition) with total mass $\asymp 1$.
	\item Let $\eta_3$ be the SRW started from the endpoint of $\eta_2$ that hits $\partial B_n$ before $\partial B(z,3r/2)$ (so that it does not hit $\gamma_1$, $\gamma_2$ and $\eta_1$ by definition) with total mass $(\log(n/r))^{d-3}$ given by the gambler's ruin estimate.
\end{itemize}
With this construction the concatenated path $\eta_1\oplus [\gamma_1]^R\oplus \gamma_2\oplus \eta_2\oplus\eta_3$ indeed contains a cut point at $z$. We obtain the lower bound by multiplying all these masses. The upper bound follows immediately by using the same decomposition (but omitting the requirement that $\eta_1$ and $\eta_2$ need to stay in a well-chosen tube).
Thus, we conclude the proof. 
\end{proof}

Next, we give an upper bound in the two-point case. Recall the definition of $E_n$ in \eqref{eq:E_ndef}.
\begin{lemma}\label{lem:two-p}
	Let $z_1,z_2\in B_n$ with $r_1:=|z_1|$ and $r_2:=|z_1-z_2|$. Then,
	\[	\Pb( E_n((z_1,z_2)) ) =		O(r_1^{-\eta} r_2^{-\eta}).\]
\end{lemma}
\begin{proof}
We start with the ``bulk'' cases. We only illustrate in the  case of $r_2 \leq r_1\le \frac{n}{3}$ how to obtain the order of the mass of all the pieces that are used to construct $S[0,\iota_n]$ such that $\{ z_1,z_2\} \subseteq \Af_n$ and $\tau_{z_1}<\tau_{z_2}$ occur, imitating the proof of the one-point cases, as all ther ``bulk'' cases can be dealt with in a similar fashion.
	\begin{itemize}
		\item Let $(\gamma_1,\gamma_2)$ be a pair of non-intersecting random walks from $z_1$ to $\partial B(z_1,r_2/4)$ and  $(\gamma_3,\gamma_4)$ be a pair of non-intersecting random walks from $z_2$ to $\partial B(z_2,r_2/4)$. Both of them have total mass $\asymp r_2^{-\xi}$.
		\item Let $(\zeta_1,\zeta_2)$ be a pair of non-intersecting random walks started from the endpoints of $\gamma_1$ and $\gamma_4$ respectively and stopped upon hitting $\partial B(z,r_1/2)$ with $z:=\frac{z_1+z_2}{2}$, which has total mass $\asymp (r_1/r_2)^{-\xi}$.
		\item Let $\omega$ be the path connecting the endpoints of $\gamma_2$ and $\gamma_3$ with total mass $\asymp r_2^{2-d}$.
		\item Let $\eta_1$ be the path connecting $0$ and the endpoint of $\zeta_1$ with total mass $\asymp r_1^{2-d}$.
		\item Let $\eta_2$ be the simple random walk from the endpoint of $\zeta_2$ to $\partial B_n$, which has total mass $1$.
		\item Concatenate $\eta_1 \oplus [\zeta_1\oplus\gamma_1]^R\oplus \gamma_2\oplus \omega\oplus [\gamma_3]^R\oplus \gamma_4\oplus\zeta_2\oplus \eta_2$ to recover $S$.
	\end{itemize}
	See Figure \ref{fig:concatenate} for an illustration.
	\begin{figure}[!h]
		\begin{center}
			\includegraphics[scale=0.6]{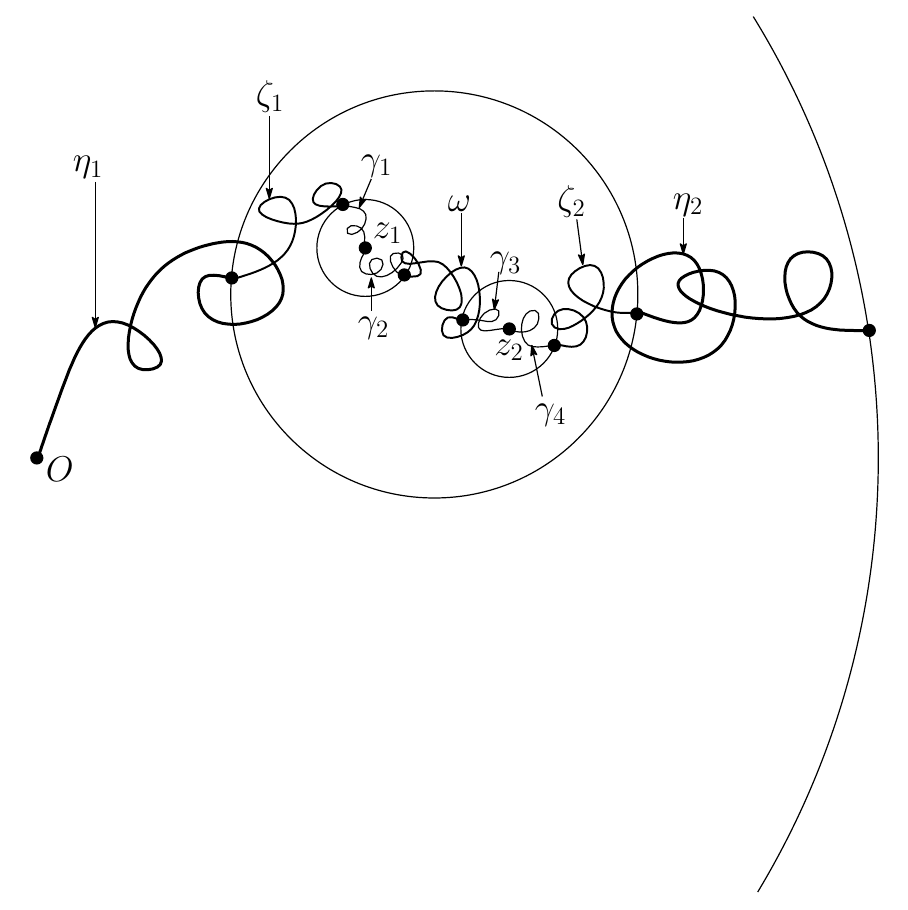}
			\caption{Proof of Lemma \ref{lem:two-p}.}
			\label{fig:concatenate}
		\end{center}
	\end{figure}
	Compute the multiplication of the above total mass:
	\[
	r_2^{-2\xi} (r_1/r_2)^{-\xi} r_2^{2-d} r_1^{2-d} = r_1^{-\eta} r_2^{-\eta},
	\]
	which gives the desired order. 

We now turn to the ``boundary'' cases. Note that when $z_1$ or $z_2$ gets close to $\partial B_n$, we need to calculate half-plane/space non-intersection probability when performing path decomposition.  Nevertheless we still obtain an upper bound at least as good as $O(r_1^{-\eta} r_2^{-\eta})$, as we now explain. Let $\widetilde\xi=\widetilde\xi_d$ stand for the half-plane intersection exponent defined in \cite{MR1879850}; we know that 
	\begin{equation}\label{eq:tildexi}
	 \widetilde\xi>\eta\quad \mbox{ for both $d=2,3$}:
	\end{equation}
as calculated in \cite{MR1879850}, $\widetilde\xi_2=\frac{10}{3}>\frac54=\xi_2=\eta_2$, and $\widetilde\xi_3\ge 2>1+\xi_3=\eta_3$ by the gambler's ruin estimate and the fact that $\xi_3<1$. 
We now take the case where $n\asymp r_1\asymp r_2 \asymp n-|z_2| \gg n-r_1$ (i.e., $z_1$ is close to $\partial B_n$ but $z_2$ is in the bulk of $B_n$) as an example. In this case, the total mass of all the pieces that are used to construct $S[0,\iota_n]$ are of order 
$$
\big(r_1/(n-r_1)\big)^{-\eta} (n-r_1)^{-\widetilde\xi} r_2^{-\eta} 
$$
which is $O(r_1^{-\eta} r_2^{-\eta})$ by \eqref{eq:tildexi}. In all other ``boundary'' cases, we obtain $O(r_1^{-\eta} r_2^{-\eta})$ for the same reason.
\end{proof}

Let $A_n(\eps):=B_n\setminus B_{n-\eps n}$. Define the sets 
\[
U^1_n(\eps)=A_{n}(\eps)\times A_{n}(\eps), 
\quad\quad  U^2_n(\eps)=\{ (z_1,z_2)\in B_n\times B_n: |z_1-z_2|\le \eps n \}.
\]
Then, as a quick corollary of Lemma~\ref{lem:two-p},  we have the following estimate, which shows that cut points in $U^1_n(\eps)\cup U^2_n(\eps)$ comprise a small portion of $\Eb[M_n^2]$ when $\eps$ is small. 
\begin{lemma}\label{lem:boundary}
	For all $0<\eps< 1/2$,
	\[
	 \sum_{(z_1,z_2)\in U^1_n(\eps)\cup U^2_n(\eps)} \Pb( \{ z_1,z_2\} \subseteq \Af_n )= O(\eps^{2-\xi}n^{2(2-\xi)}).
	\]
\end{lemma}

\section{Description of cut-point Green's function via cut balls}\label{sec:gcb}
We now turn to cut points of Brownian motion. The main goal of this section is to introduce cut ball, which in some sense approximates cut points in mesoscopic scales and show that the cut-point Green's function can also be expressed through the asymptotic probability of cut ball events. We first review the known results in \cite{mink_cont} about the one-point and two-point Green's functions in Section~\ref{subsec:review}, and provide some alternative ways to describe the Green's function of cut points. Then, in Section~\ref{subsec:cbd}, we give the definition of cut balls and relate to Green's function to asymptotic probability of cut ball events.

\subsection{Green's function of cut points: review}\label{subsec:review}
We recall some facts about Brownian cut points from \cite{mink_cont}. For this, we need to introduce some notation from \cite{mink_cont}. 
Recall $\delta$ and $\eta$ from \eqref{eq:eta}. Recall in the beginning of the introduction that $\Ac_{\gamma}$ is the set of cut points of a curve $\gamma$.
Let $\gamma:=W[0,T_0]$ be the standard Brownian motion stopped upon reaching the unit sphere $\partial \Dc$.
Define 
\begin{equation}\label{eq:Hsz}
H_s(z):=\{ \dist(z,\Ac_{\gamma})\le e^{-s} \} \text{ and } J_s(z):=e^{\eta s} 1_{H_s(z)}.
\end{equation}
Let $\Vc$
be the set of dyadic cubes $V\subset\Dc$ of the form
\begin{equation}\label{eq:nice-box}
V=\Big[\frac{k_{1}}{2^{n}},\frac{k_{1}+1}{2^{n}}\Big]\times\ldots\times\Big[\frac{k_{d}}{2^{n}},\frac{k_{d}+1}{2^{n}}\Big]
\end{equation}
for integers $n,k_{1},\ldots,k_{d}$ such that $\dist(0, V) \wedge \dist(\partial\Dc, V) \geq 2 \operatorname{diam}(V)=\sqrt{d} 2^{-n+1}$.
{The following two theorems, one establishing asymptotics for cut-point Green's function and the other essentially proving the existence of  Minkowski content for  Brownian cut points, are variants of Theorems 1.1 through 1.3 in \cite{mink_cont}, where the authors construct the natural fractal measure on Brownian cut points via Minkowski content, and be proved in a fashion very similar to the methods in \cite{mink_cont}. See Remark \ref{rem:setup} for more discussions.}

\begin{theorem}\label{thm:gr}
	For all $z, w, \partial \Dc$ such that  $\dist(0,z,w,\partial\Dc)>0$, the following limits exist
	\begin{equation}\label{eq:H-one-point}
		G^{\cut}_{\Dc}(z)=\lim_{s\rightarrow\infty} \Eb[ J_s(z) ],
	\end{equation}
	and 
	\[
	G^{\cut}_{\Dc}(z,w)=\lim_{s\rightarrow\infty}\Eb[ J_s(z) J_s(w) ].
	\]
	Moreover, there exists $u>0$ such that if $e^{-b}=\dist( 0, z, \partial \Dc )>0$ and $s\ge b+1$, then
	\begin{equation}\label{eq:S-one-point}
		\Eb[ J_s(z) ]= G^{\cut}_{\Dc}(z) [ 1+O(e^{(b-s)u}) ].
	\end{equation}
	Moreover, there exists $u>0$ such that if $V\in \Vc$, $z,w\in V$ with $e^{-b}=\dist( 0, z, w, \partial \Dc )>0$, $s\ge b+1$ and $0\le \rho\le 1$, then 
	\begin{equation}\label{eq:S-two-point}
		\Eb[ J_s(z)J_{s+\rho}(w) ]
		= G^{\cut}_{\Dc}(z,w)  \, 
		[1+O( e^{(b-s)u})],
	\end{equation}  
	and there exist $0<c(V)<C(V) < \infty$ such that  
	\begin{equation}\label{eq:Gcut-zw}
		c(V) \, |z-w|^{-\eta}\leq G^{\cut}_{\Dc}(z,w) \leq C(V) \, |z-w|^{-\eta}. 
	\end{equation}
\end{theorem}

\begin{theorem}\label{thm:mink}
	Suppose $V$ is a bounded Borel subset of $\Dc$ such that $\partial V$ has zero $(d-\eps)$-Minkowski content for some $\eps>0$. The following limit exist
	\[
	J_V=\lim_{s\to \infty} J_{s,V}  \text{ with } J_{s,V}:=\int_{V} J_s(z) dz,
	\]
	and there exist $c,u>0$ such that
	\begin{equation}\label{eq:jsv}
		\Eb [ ( J_{s,V} - J_V )^2 ] \le  c e^{-us}.
	\end{equation}
	Moreover, almost surely the Minkowski content $\nu(V):=\Cont_{\delta}(\Ac_{\gamma}\cap V)$ exists and equals $J_V$, and 
	\[
	\Eb [ \nu(V) ] = \Eb [ J_V ] = \int_V G^{\cut}_{\Dc}(z)\, dz,
	\]
	\[
	\Eb [ \nu(V)^2 ] =  \Eb [ J_V^2 ] = \int_V\int_V G^{\cut}_{\Dc}(z,w)\, dz\, dw.
	\]
\end{theorem}

\begin{remark}\label{rem:setup}
As introduced at the beginning of this subsection, the work \cite{mink_cont} establishes results of the same flavor as Theorems \ref{thm:gr} and \ref{thm:mink} above but in the cases of interior-to-interior Brownian path measures and half-plane excursions, neither of which can yield Theorems \ref{thm:gr} and \ref{thm:mink} directly. However, the setup in this section can be regarded as a ``mixed'' case of interior-to-boundary path measure in the terminology of the said paper, and this setup can be treated by combining the techniques of Sections 4.6 and 4.7, ibid.
\end{remark}

In the following, we give some alternative ways to describe the Green's function of cut points, which will be useful later.
The first one has been given in the proof of Theorem 1.1 in \cite{mink_cont} {(in the setup of Brownian path measures). In the lemma below we present a version in our setup} and refer the reader to the said paper for details.
Let $c_*$ be the universal constant defined in  (3.10) of \cite{mink_cont} and $S=\partial \Dc_{-s}(z)$, $D=\Dc \setminus\overline{\Dc_{-s}(z)}$ and define: 
	\begin{equation}\label{eq:gkdefinition1}
		{\gf_{s}=\gf_{s}(z)}:=\int_{S } \int_{S } \mu^D_{0,x} \otimes \mu_ {\partial\Dc,y}^D \left[ \gamma_1\cap \gamma_2=\emptyset\right]\sigma(dx,dy),
	\end{equation}
{where we recall from Section~\ref{subsec:pm} that $\sigma$ denotes the surface measure on spheres (area if $d = 3$ and length if $d = 2$).}
\begin{lemma}\label{lem:gs}
	There {exist} universal positive constants $c_*$ and $u$ such that the following holds for all $z\in\Dc$ with $e^{-b}:=\dist( 0, z, \partial \Dc )>0$ and $s\ge b+1$
	\begin{equation}\label{eq:Gcut1}
		c_* e^{2s(d-2)} e^{\xi s} \gf_{s}=G_{\Dc}^{\cut}(z) [1+O(e^{-u(s-b)})].
	\end{equation}
\end{lemma}
\begin{proof}
	{Analogous to} (4.2) in \cite{mink_cont}\footnote{We note that in fact there is a typo in (4.2) of \cite{mink_cont}, where {a factor $e^{-s(d-2)}$ is missing.}
	 This factor is due to the scaling relation for Brownian path measures in three dimensions.},
	we know that for some constant $1\le a\le 2$ and $u>0$, 
	\begin{equation}\label{eq:gz}
		e^{-s(d-2)} \Pb \{H_{s+a}(z)\} \gf_s^{-1} = c_* e^{-a\eta} [1+O(e^{-u(s-b)})].
	\end{equation}By \eqref{eq:S-one-point}, we have 
	\[
	e^{\eta (s+a)} \Pb \{H_{s+a}(z)\}=G^{\cut}_{\Dc}(z) [ 1+O(e^{-u(s+a-b)}) ].
	\] 
	This combined with \eqref{eq:gz} completes the proof.
\end{proof}

{We will need a continuity result for $G_{\Dc}^{\cut}(z)$ later, which follows from the above result immediately. Let us state it below. 
\begin{lemma}\label{lem:continuous-G}
	There exists a universal constant $u>0$ such that the following holds for all $z\in\Dc$ with $e^{-b}:=\dist( 0, z, \partial \Dc )>0$ and $|z-w|<e^{-b-2k}$ with $k\ge 1$,
	\[
	G_{\Dc}^{\cut}(z)=G_{\Dc}^{\cut}(w)[1+O(e^{-uk})].
	\]
\end{lemma}
\begin{proof}
	Note that 
	\[
	\Dc_{-b-k+\eps_-(k)}(w)\subset \Dc_{-b-k}(z) \subset \Dc_{-b-k+\eps_+(k)}(w) \subset \Dc,
	\]
    where $\eps_{\pm}(k)=\log(1\pm e^{-k})$. Then, by monotonicity, we have 
    \[
    \gf_{b+k-\eps_-(k)}(w)\le \gf_{-b-k}(z) \le \gf_{b+k-\eps_+(k)}(w).
    \]
    Therefore, by Lemma~\ref{lem:gs}, we obtain that
    \[
    (1-e^{-k})^{2(d-2)+\xi} G_{\Dc}^{\cut}(w) [1+O(e^{-u'k})]
    \le G_{\Dc}^{\cut}(z)\le (1+e^{-k})^{2(d-2)+\xi} G_{\Dc}^{\cut}(w) [1+O(e^{-u''k})].
    \]
    This implies the result.
\end{proof}
}

As a consequence of {Lemma~\ref{lem:gs}}, we have the following estimate for the Green's function $G_{\Dc}^{\cut}(z)$. 
\begin{lemma}\label{lem:Gcut-z}
	Let $d_z=\dist( 0, z, \partial \Dc )$. Then,
	\begin{equation}\label{eq:g-z}
	G_{\Dc}^{\cut}(z)\asymp a(z), \quad \text{ with } \  a(z):=d_z^{1-\xi} 1_{|z|\ge 1/2} + d_z^{-\eta} [ \log(d_z^{-1}) ]^{d-3} 1_{|z|<1/2}.
	\end{equation}
\end{lemma}
\begin{proof}
	This can be proved in a similar way as Lemma~\ref{lem:biop}. To get the up-to-constants estimate for $\gf_s$ in \eqref{eq:gkdefinition1}, we only need to replace the pair of NIRW's $(\gamma_1,\gamma_2)$ from $z$ to $\partial B(z,r/2)$ in the proof of Lemma~\ref{lem:biop}
	by a pair of NIBM's $(\gamma^1,\gamma^2)$. Here the pair $(\gamma^1,\gamma^2)$ is sampled from the boundary-to-boundary excursion measures in the annulus $D(z,d_z/2)\setminus\overline D(z,e^{-s})$,
	i.e. $$\mu^{D(z,d_z/2)\setminus\overline D(z,e^{-s})}_{\partial D(z,e^{-s}),\partial D(z,d_z/2)}\otimes \mu^{D(z,d_z/2)\setminus\overline D(z,e^{-s})}_{\partial D(z,e^{-s}),\partial D(z,d_z/2)}.$$ and then restricted to be non-intersecting.
	The total mass of non-intersecting $(\gamma^1,\gamma^2)$ has been estimated in Lemma~\ref{lem:ex-ni}, which is {of order} $e^{-2(d-2)s} e^{-\xi s} d_z^{-\xi}$.  Therefore, we have 
	\begin{equation}\label{eq:gs}
	\gf_s\asymp e^{-2s(d-2)} e^{-\xi s}\, a(z).
	\end{equation}
	This combined with \eqref{eq:Gcut1} concludes the proof.
\end{proof}

The above lemma will be used to derive Proposition~\ref{prop:one-point}, which provides another description for $G^{\cut}_{\Dc}(z)$ via cut-ball events.
Next, we present a version that is tailored for our use in the next section where we obtain the sharp one-point estimate for the simple random walk.

\begin{figure}[h!]
	\centering
	\includegraphics[width=.5\textwidth]{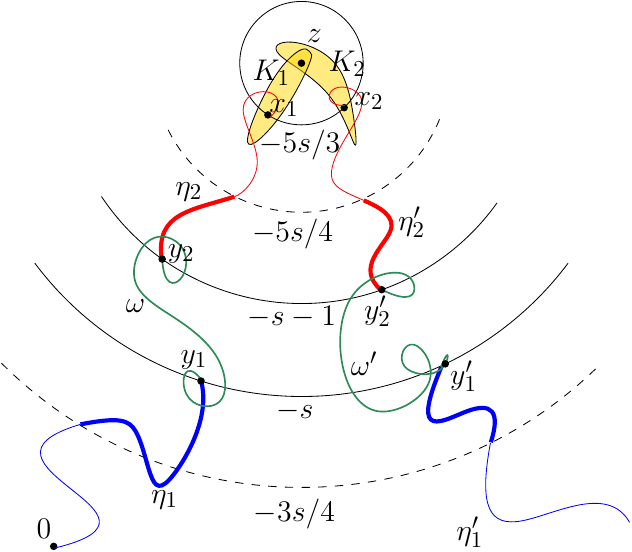}
	\caption{Proof of Proposition~\ref{prop:gbar-gs}. $K_1$ and $K_2$ are the yellow parts. $\eta_1$ ($\eta_1'$), $\eta_2$ ($\eta_2'$) and $\omega$ ($\omega'$) are in blue, red and green, respectively. The heavy red pair and the heavy blue pair are close to quasi-invariant measures.}
	\label{img:one-point}
\end{figure}

Suppose $K_1$ and $K_2$ are two compact sets in $\Dc_{-5s/3+1}(z)$ and $x_i\in K_i\cap\partial \Dc_{-5s/3}(z)$ for $i=1,2$. Let 
	\begin{equation}\label{eq:bar A}
	\overline A(\overline K,\overline x):=\{ (\gamma_1,\gamma_2)\in \widetilde\Gamma_{0,x_1}\times \widetilde\Gamma_{\partial\Dc,x_2}: \gamma_1\cap(\gamma_2\cup K_{2})=(\gamma_1\cup K_{1})\cap \gamma_2=\emptyset \},
	\end{equation}
	\[
	\overline g(\overline K,\overline x):=\mu^{\Dc}_{0,x_1} \otimes \mu_ {\partial\Dc,x_2}^{\Dc} \left\{ \overline A(\overline K,\overline x)\right\}.
	\]
	Moreover, recall the definition of {non-intersection event} $\widetilde A_{z,-s}(\overline K,\overline x)$ in \eqref{eq:Ar} and define 
	\[
	g_s(\overline K,\overline x):=\Pb\{\widetilde A_{z,-s}(\overline K,\overline x)\}.
	\]
\begin{proposition}\label{prop:gbar-gs}
	There exist universal constants $c_2, u>0$ such that the following holds for all $\dist( 0, z, \partial \Dc )\ge e^{-2s/3}$,
	\[
	e^{\xi s}\, \overline g(\overline K,\overline x)=c_2\, g_s(\overline K,\overline x) \, G_{\Dc}^{\cut}(z) [1+O(e^{-us})].
	\]
\end{proposition}
\begin{proof}
	We decompose the $(\gamma_1,\gamma_2)$ in the following way:
	\[
	\gamma_1=\eta_1\oplus \omega \oplus [\eta_2]^R, \quad  \gamma_2=\eta'_1\oplus \omega' \oplus [\eta'_2]^R,
	\]
	where $\eta_1$ is $\gamma_1$ from $0$ to its first visit of $\partial \Dc_{-s}(z)=S$, $\eta_2$ is $\gamma_1^{R}$ from $x_1$ to its first visit of $\partial\Dc_{-s-1}(z)$, and $\omega$ is the rest part of $\gamma_1$, and $\gamma_2$ is decomposed in the same way (with $0$ replaced by $\partial\Dc$). See {Figure}~\ref{img:one-point} for an illustration.
	
	Denote $\overline\eta_i:=(\eta_i,\eta'_i)$ for $i=1,2$, and $(y_i,y'_i):=(\eta_i\cap\partial \Dc_{-s-i+1}(z),\eta'_i\cap\partial \Dc_{-s-i+1}(z))$. 
	We now sample $\overline\eta_1$ by the following law
	\[
	 \frac{1}{\gf_{s}}\int_{S } \int_{S } \mu^D_{0,y_1} \otimes \mu_ {\partial\Dc,y'_1}^D \left[ 1_{\eta_1\cap \eta'_1=\emptyset} (\eta_1,\eta'_1)\in \cdot \right]\sigma(dy_1,dy'_1),
	\]
	where $D$ and $\gf_s$ are defined above Lemma~\ref{lem:gs}. By \eqref{eq:qi_bm_definition-inv}, the law of $\overline\eta_1[T_{\partial\Dc_{-3s/4}(z)},T_{\partial\Dc_{-s}(z)}]$ has total variation distance $O(e^{-us})$ to the probability measure $\Qf^*_{-s}(z)[T_{\partial\Dc_{-3s/4}(z)},T_{\partial\Dc_{-s}(z)}]$, where $\Qf^*_{-s}(z)$ is the pushforward of the quasi-invariant measure $\Qf^*$ on NIBM's from infinity under the map $f_1(x)=e^{-s}(x+z)$. 
	
	Now we sample $\overline\eta_2$ from the following law:
	\[
	 e^{-s-1}\circ \left(Q_{2s/3-1}(e^{5s/3}\baK,e^{5s/3}\overline x)\right),
	\]
	where $Q_{2s/3-1}(e^{5s/3}\baK,e^{5s/3}\overline x)$ is defined above \eqref{eq:qi_bm_definition}. By \eqref{eq:qi_bm_definition}, the law of $\overline\eta_2[T_{\partial\Dc_{-5s/4}(z)},T_{\partial\Dc_{-s-1}(z)}]$ has total variation distance $O(e^{-us})$ to the probability measure $\Qf_{-s-1}(z)[T_{\partial\Dc_{-5s/4}(z)},T_{\partial\Dc_{-s-1}(z)}]$, where $\Qf_{-s-1}(z)$ is the pushforward of the quasi-invariant measure $\Qf$ on NIBM's from origin under the map $f_2(x)=e^{-s-1}(x+z)$. 
	
	Write $\overline\Eb$ for the product measure of those of $\overline\eta_1$ and $\overline\eta_2$.
	Given $\overline\eta_1, \overline\eta_2$, we define the set
	\[
	\Wc(\overline\eta_1, \overline\eta_2):=\left\{ (\omega,\omega')\in \widetilde\Gamma^{\Dc}_{y_1,y_2}\times \widetilde\Gamma^{\Dc}_{y'_1,y'_2} : (\eta_1\oplus \omega \oplus [\eta_2]^R,\eta'_1\oplus \omega' \oplus [\eta'_2]^R)\in \overline A(\overline K,\overline x) \right\},
	\]
	and write
	\[
	\Upsilon(\overline\eta_1, \overline\eta_2):= \mu^{\Dc}_{y_1,y_2} \otimes \mu_ {y'_1,y'_2}^{\Dc} \left\{ \Wc(\overline\eta_1, \overline\eta_2) \right\}.
	\]
	Then, we have
	\begin{equation}\label{eq:bargs}
		\overline g(\overline K,\overline x)=g_{s+1}(\overline K,\overline x)\, \gf_s\, \overline \Eb[ \Upsilon(\overline\eta_1, \overline\eta_2)].
	\end{equation}
	
    Furthermore, we consider a subset of $\Wc(\overline\eta_1, \overline\eta_2)$:
    \[
    \Wc'(\overline\eta_1, \overline\eta_2):=\left\{ (\omega,\omega')\in \Wc(\overline\eta_1, \overline\eta_2) : (\omega\cup\omega') \subset \Dc_{-3s/4}(z)\setminus\Dc_{-5s/4}(z) \right\},
    \]
    and write
    \[
    \Upsilon'(\overline\eta_1, \overline\eta_2):= \mu^{\Dc}_{y_1,y_2} \otimes \mu_ {y'_1,y'_2}^{\Dc} \left\{ \Wc'(\overline\eta_1, \overline\eta_2) \right\}.
    \]
	Then, by applying the Beurling estimate (Proposition~\ref{p:Beurling}) when $d=2$; and Lemma~\ref{l:trans_recur} when $d=3$ to control the probability that $\omega$ or $\omega'$ leaves the annulus $\Dc_{-3s/4}(z)\setminus\Dc_{-5s/4}(z)$, we have 
	\begin{equation}\label{eq:Upsi}
		\overline\Eb[ \Upsilon(\overline\eta_1, \overline\eta_2)] \simeq \overline\Eb[ \Upsilon'(\overline\eta_1, \overline\eta_2)].
	\end{equation}

    The advantage to work with $\Upsilon'(\overline\eta_1, \overline\eta_2)$ rather than $\Upsilon(\overline\eta_1, \overline\eta_2)$ is that the previous one only depends on $\overline\eta_1, \overline\eta_2$ inside the annulus $\Dc_{-3s/4}(z)\setminus\Dc_{-5s/4}(z)$. According to our previous observations, inside this annulus, the distributions of $\overline\eta_1$ and $\overline\eta_2$ are close to quasi-invariant measures respectively. Therefore, 
	\begin{equation}\label{eq:UQ}
		\overline\Eb[ \Upsilon'(\overline\eta_1, \overline\eta_2)] \simeq
		\Qf^*_{-s}(z)\otimes \Qf_{-s-1}(z) [ \Upsilon''(\overline\beta_1, \overline\beta_2)],
	\end{equation}
where $\Upsilon''(\overline\beta_1, \overline\beta_2)$ is defined in a similar manner, that is, the total mass of $(\omega,\omega')$ such that the concatenation of $\omega$ with $\beta_1$ and $\beta_2$ does not intersect the concatenation of $\omega'$ with $\beta'_1$ and $\beta'_2$.  Then, by translation invariance ($-z$) and scaling invariance ($\times e^{s}$), we have
	\begin{equation}\label{eq:UQ'}
	\Qf^*_{-s}(z)\otimes \Qf_{-s-1}(z) [ \Upsilon''(\overline\beta_1, \overline\beta_2)]=e^{2(d-2)s}\cdot  \Qf^*\otimes \Qf_{-1} [ \Upsilon''(\overline\beta'_1, \overline\beta'_2)].
\end{equation}
One should be careful that the total mass of the path measure under the scaling multiplied by $e^s$ for the intermediate parts $\omega$ and $\omega'$ will decrease by $e^{-(d-2)s}$  each (this explains why we have an extra term $e^{2(d-2)s}$ on the right hand side). 
Letting $\mathfrak{q}:=\Qf^*\otimes \Qf_{-1} [ \Upsilon''(\overline\beta'_1, \overline\beta'_2)]$, we see that $\mathfrak{q}$ is bounded away from $0$ by using the separation lemma (for both $\overline\beta'_1$ and $\overline\beta'_2$), and bounded away from infinity by using Lemma~\ref{l:paths_discon_beurling} when $d=2$; and by using a bound of Green's function $\sup_{w_1\in\partial\Dc, w_2\in \partial\Dc_{-1}} \wt G_{\Rb^3}(w_1,w_2)<\infty$ when $d=3$. Thus, $\mathfrak{q}$ is a universal constant in $(0,\infty)$. Combining \eqref{eq:Upsi}, \eqref{eq:UQ} and \eqref{eq:UQ'}, we obtain that 
	\begin{equation}\label{eq:Upsi'}
	\overline\Eb[ \Upsilon(\overline\eta_1, \overline\eta_2)] \simeq \mathfrak{q}\, e^{2(d-2)s}.
\end{equation}
Moreover, it follows from \eqref{eq:An} that 
	\begin{equation}\label{eq:gs+1}
		g_{s+1}(\overline K,\overline x)\simeq g_{s}(\overline K,\overline x)\, e^{\xi}.
	\end{equation}
	Plugging \eqref{eq:Gcut1}, \eqref{eq:Upsi'} and  \eqref{eq:gs+1} into \eqref{eq:bargs}, we obtain that 
	\[
	\overline g(\overline K,\overline x)\simeq g_{s}(\overline K,\overline x) \, e^{\xi} \, c_*^{-1} \, e^{-2s(d-2)} \, e^{-\xi s} \, G_{\Dc}^{\cut}(z)\, \mathfrak{q}\, e^{2(d-2)s}.
	\]
We finish the proof of Proposition~\ref{prop:gbar-gs} by setting $c_2=\mathfrak{q}\, c_*^{-1} \, e^{\xi}$.
\end{proof}

In the rest of this section, we will present the two-point counterparts of Lemma~\ref{lem:gs} and Proposition~\ref{prop:gbar-gs}. To this end, we introduce some notation first.
Let $V\in\Vc$, $z,w\in V$ with $\dist(0,z,w,\partial\Dc)=e^{-b}>0$ and $s\ge b+1$. Let $S=\partial \Dc_{-s}(z)$, $S'=\partial \Dc_{-s}(w)$.
We construct a measure $\Uc$ on the triple $\overline\gamma:=(\gamma_1,\gamma_2,\gamma_*)$ as follows:
\begin{itemize}
	\item sample $(\gamma_1,\gamma_2)$ from the measure $\int_{S }  \int_{S' } \mu^{\Dc\setminus \overline{\Dc_{-s}(z)}}_{0,x} \otimes \mu_ {\partial\Dc,y}^{\Dc\setminus \overline{\Dc_{-s}(w)}} \sigma(dx,dy)$,
	\item sample $\gamma_*$ from the excursion measure in $U:=\Dc\setminus (\overline{\Dc_{-s}(z)}\cup \overline{\Dc_{-s}(w)})$ from $S$ to $S'$, i.e.,
	\[
	\int_S \int_{S^{\prime}} \mu_{x^{\prime}, y^{\prime}}^{U} \sigma\left(d x^{\prime}\right) \sigma\left(d y^{\prime}\right),
	\]
\end{itemize}
and then restricted to the event $N_0$
\[
N_0:=\left\{\gamma_1 \cap \gamma_*=\gamma_1 \cap \gamma_2=\gamma_* \cap \gamma_2=\emptyset\right\}.
\]
Let $\Uc'$ be the measure as $\Uc$ defined above with $z,w$ interchanged.

The following lemma is a consequence of Theorem 1.2 in \cite{mink_cont} (in the setup of Brownian path measures), which is the two-point counterpart of Lemma~\ref{lem:gs}.

\begin{lemma}
	There exists $u>0$ such that the following holds for all $V\in\Dc$, $z,w\in V$ with $\dist(0,z,w,\partial\Dc)=e^{-b}>0$  and $s\ge b+1$.
	\begin{equation}\label{eq:Gcut2}
		c_*^2 e^{2s(d-2)} e^{2s\eta} (\| \Uc \|+\| \Uc' \|)= G_{\Dc}^{\cut}(z,w)[1+O_V(e^{-u(s-b)})].
	\end{equation}
\end{lemma}
\begin{proof}
	Analogous to (4.31) of \cite{mink_cont}\footnote{We note that the factor $e^{-2s(d-2)}$ comes from the scaling of Green's function in three dimensions (compared with \eqref{eq:gz}, here we have two of them), and this factor is also missing in (4.31) of \cite{mink_cont}.}, for some constant $1\le a\le 2$ and $u>0$,
	\[
	e^{-2s(d-2)} \Pb\left\{H_{s+a}(z) \cap H_{s +a}(w) \right\}  = c_*^2 e^{-2a\eta} (\| \Uc \|+\| \Uc' \|) [1+O_{V}( e^{(b-s)u})].
	\]
	By \eqref{eq:S-two-point},
	\[
	e^{2(s+a)\eta}\Pb\left\{H_{s+a}(z) \cap H_{s +a}(w) \right\} 
	= G^{\cut}_{\Dc}(z,w)  \, 
	[1+O_{V}( e^{(b-s)u})].
	\]
	We finish the proof by combining these two estimates.
\end{proof}

Next, we present the counterpart of Proposition~\ref{prop:gbar-gs}.

\begin{proposition}\label{prop:gbar-gs2}
	There exists $u>0$ such that the following holds for all $V\in\Vc$ and $z,w\in V$  with $\dist(0,z,w,\partial\Dc)\ge e^{-2s/3}$. Suppose $K_1$ and $K_2$ are two closed sets in $\Dc_{-5s/3+1}(z)$ and $x_i\in K_i\cap\partial \Dc_{-5s/3}(z)$ for $i=1,2$. 
	Suppose $K'_1$ and $K'_2$ are two closed sets in $\Dc_{-5s/3+1}(w)$ and $x'_i\in K'_i\cap\partial \Dc_{-5s/3}(w)$ for $i=1,2$.
	Let $\widehat K=(\overline K,\overline K')$ with $\overline K=(K_1,K_2)$ and $\overline K'=(K'_1,K'_2)$. Also, let $\widehat x=(\overline x,\overline x')$ with $\overline x=(x_1,x_2)$ and $\overline x'=(x'_1,x'_2)$. Define
	\[
	\widehat A(\widehat K,\widehat x)= \left\{\begin{array}{l}
		(\gamma_1,\gamma_2,\gamma_*)\in \widetilde\Gamma_{0,x_1}\times \widetilde\Gamma_{\partial\Dc,x'_1}\times \widetilde\Gamma_{x_2,x'_2}: \\
		\gamma_1\cap(\gamma_2\cup\gamma_*\cup K_2\cup \Dc_{-5 s/3+1}(w))=\emptyset, \\
		\gamma_2\cap(\gamma_1\cup\gamma_*\cup K'_2\cup \Dc_{-5 s/3+1}(z))=\emptyset, \\
		\gamma_*\cap (K_1\cup K'_1)=\emptyset
	\end{array} \right\}.
	\]
	Let $\widehat A'(\widehat K,\widehat x)$ be defined as $\widehat A(\widehat K,\widehat x)$ with $z$ and $w$ interchanged.  Define
	\[
	\widehat g(\widehat K,\widehat x)=\mu^{\Dc}_{0,x_1} \otimes \mu_ {\partial\Dc,x'_1}^{\Dc} \otimes \mu_ {x_2,x'_2}^{\Dc} \left[ \widehat A(\widehat K,\widehat x)\cup \widehat A'(\widehat K,\widehat x)\right],
	\]
	and
	\[
	g_s(\widehat K,\widehat x)=\Pb(\widetilde A_{z,-s}(\overline K,\overline x))\, \Pb(\widetilde A_{w,-s}(\overline K',\overline x')).
	\]
	Then, we have 
	\begin{equation}\label{eq:hatg}
		e^{2\xi s}\, \widehat g(\widehat K,\widehat x)=c_2^2\, g_s(\widehat K,\widehat x) \, G_{\Dc}^{\cut}(z,w) [1+O_V(e^{-us})],
	\end{equation}
	where $c_2$ is the same constant in Proposition~\ref{prop:gbar-gs}.
\end{proposition}

As one can observe from \eqref{eq:hatg}, the constant $2$ in the exponent explains it is merely a duo of Proposition~\ref{prop:gbar-gs}.
We omit the proof 
of this proposition, as it is quite similar to that of Proposition~\ref{prop:gbar-gs}.

\subsection{Description via cut balls}\label{subsec:cbd}
{The event $H_s(z)$ defined in \eqref{eq:Hsz} is hard to relate to the event that $z^{(n)}$ is a cut point of the simple random walk in $\Dc\cap\Zc_n$, since we need to handle the microscopic scale.} To this end, we introduce the alternative cut-ball event for the Brownian motion and establish {results of the same flavor} as in the previous subsection. It turns out that such cut-ball events have a nice counterpart in the discrete side.

Abbreviate $\gamma:=W[0,T_0]$.
For any set $B$ in the unit disk $\Dc$ with $0\notin B$, if $\gamma\cap B\neq\emptyset$, we can decompose $\gamma$ by first-entry and last-exit of $B$ as follows:
\begin{equation}\label{eq:decomp2}
	\gamma=\gamma_1\oplus\omega\oplus[\gamma_2]^R,
\end{equation}
where $\gamma_1$ (resp.\ $\gamma_2$) is the part of $\gamma$ (resp.\ $[\gamma]^R$) from its starting point to its first entry of $B$ and $\omega$ is the rest part of $\gamma$ from the ending point of $\gamma_1$ to that of $\gamma_2$.
\begin{definition}[Cut ball for BM]\label{def:cut-bm}
	For $s>0$, 
	we say that the ball $B:=\Dc_{-s}(z)\subseteq \Dc$ with $0\notin B$ is a cut ball for $\gamma:=W[0,T_0]$ if $\gamma\cap B\neq\emptyset$, $\gamma_1\cap\gamma_2=\emptyset$ and $\omega\in \Dc_{-2s/3}(z)$ where $\gamma=\gamma_1\oplus\omega\oplus[\gamma_2]^R$ is decomposed as in \eqref{eq:decomp2} with $B=\Dc_{-s}(z)$. 
	Denote by $\Kt_s(z)$ the event that $\Dc_{-s}(z)$ is a cut ball for $W[0,T_0]$.
	See Figure~\ref{img:cut-ball} for an illustration.
\end{definition}

\begin{figure}[h!]
	\centering
	\includegraphics[width=.3\textwidth]{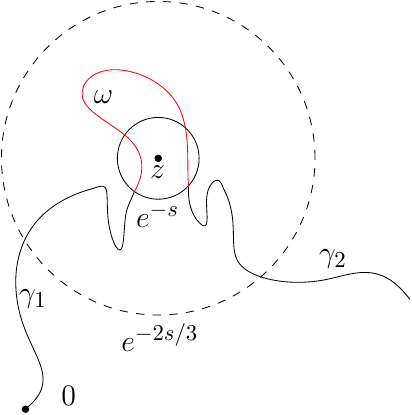}
	\caption{The event $\Kt_s(z)$. The ball (with a solid boundary) of radius $e^{-s}$ about $z$ is a cut ball. The ball (with a dashed boundary) of radius $e^{-2s/3}$ about $z$ contains the intermediate part $\omega$ (in red).}
	\label{img:cut-ball}
\end{figure}

\begin{remark}
	We choose to use the mesoscopic ball of radius $e^{-2s/3}$ in order to use the naive bound $3\eta/2<d$ later; see the last paragraph in the proof of Theorem~\ref{thm:tvv}.
\end{remark}

Let $\Qf^*$ be the quasi-invariant measure on NIBM's from infinity introduced in Section~\ref{subsec:qi-infty}. 
Denote by $\Qf^*(dx,dy)$ the distribution of endpoints of the pair of Brownian motions on $\partial \Dc$ induced from $\Qf^*$.
We will approximate the Minkowski content of Brownian cut points $\nu$ provided in Theorem~\ref{thm:mink} by the following measure
\begin{equation}\label{eq:tnus}
\wt \nu_{s}(V) =  \int_V L_s(z) \, dz, \quad \text{ with } \ 
L_s(z):=c_* \Qf^*[\Psi_{s}]^{-1} e^{\eta s} \,1_{\Kt_s(z)},
\end{equation}
where $\Psi_{s}$ is defined in \eqref{eq:Phidef} later, and $\Qf^*[\Psi_{s}]$ is a compensating factor (coming from the definition of cut ball on the intermediate part $\omega$) that only depends on $s$ and is of order $s^{3-d}$ by Lemma~\ref{lem:QPA}.

\begin{proposition}\label{prop:one-point}
	There exists $u>0$ such that if $\dist( 0, z, \partial \Dc )\ge e^{-2s/3}$, then 
	\begin{equation*}
		\Eb[ L_s(z) ]= G^{\cut}_{\Dc}(z) [ 1+O(e^{-us}) ].
	\end{equation*}
\end{proposition}

\begin{proposition}\label{prop:two-point}
	There exists $u>0$ such that for all $V\in\Vc$, $z,w\in V$ with $|z-w|\ge e^{-2s/3}$,
	\begin{equation*}
		\Eb[ L_s(z) L_s(w) ]= G^{\cut}_{\Dc}(z,w) [ 1+O_V(e^{-us}) ],
	\end{equation*}
	and 
	\begin{equation*}
		\Eb[ L_s(z) J_s(w) ]= G^{\cut}_{\Dc}(z,w) [ 1+O_V(e^{-us}) ].
	\end{equation*}
\end{proposition}

The above two propositions provide us with another way to describe the one-point and the  two-point Green's functions of the cut points respectively. {Their proof is very similar to that of} Theorem~\ref{thm:gr}. Therefore, we will only present a detailed proof of Proposition~\ref{prop:one-point} in the following subsection and leave the details of the proof of Proposition~\ref{prop:two-point} to the reader.

{In the next theorem,  we will show that the Minkowski content of Brownian cut points is well approximated by $\widetilde{\nu}_{s}$ defined in \eqref{eq:tnus}, assuming Propositions~\ref{prop:one-point} and~\ref{prop:two-point}.}

\begin{theorem}\label{thm:tvv}
	There exists $u>0$ such that if $V\in \Vc$ and $\dist(0,V,\partial \Dc)\ge e^{-2s/3}$, then
	\begin{equation}\label{eq:tv-v}
		\Eb[(\wt \nu_{s}(V)-\nu(V))^2]=O_V(e^{-us}).
	\end{equation}
\end{theorem}

\begin{proof}
	By \eqref{eq:jsv}, it suffices to prove
	\begin{equation}
		\Eb[|\wt \nu_{s}(V)-J_{s,V}|^2]= O_V(e^{-us}).
	\end{equation}
	Note that 
	\begin{align*}
		\Eb[|\wt \nu_{s}(V)-J_{s,V}|^2]=\Eb [ \wt \nu_{s}(V)^2 ]-2\Eb[ \wt \nu_{s}(V)J_{s,V} ]+ \Eb[ J_{s,V}^2 ].
	\end{align*}
	By \eqref{eq:jsv}, we have $\Eb[ J_{s,V}^2 ]=\Eb [ \nu(V)^2 ] [1+O_V(e^{-us})]$. It remains to show that {similar estimates also hold} with the left side replaced by $\Eb [ \wt \nu_{s}(V)^2 ]$ or $\Eb[ \wt \nu_{s}(V)J_{s,V} ]$. Since these {two cases are almost} the same, we will only deal with $\Eb [ \wt \nu_{s}(V)^2 ]$. 
	By splitting the integral according to the distance $|z-w|$ in which regime Proposition~\ref{prop:two-point} is valid, we obtain
	\begin{align*}
		\Eb [ \wt \nu_{s}(V)^2 ]\simeq\int_V\int_V G^{\cut}_{\Dc}(z,w) 1_{\{|z-w|\ge e^{-2s/3}\}} \, dz \, dw +
		\int_V\int_V \Eb[ L_s(z) L_s(w)] 1_{\{|z-w|\le e^{-2s/3}\}} \, dz \, dw .
	\end{align*}
	By \eqref{eq:Gcut-zw}, $G^{\cut}_{\Dc}(z,w)\asymp |z-w|^{-\eta}$. This implies that 
	\[
	\int_V\int_V G^{\cut}_{\Dc}(z,w) 1_{\{|z-w|\ge e^{-2s/3}\}} \, dz \, dw=
	\Eb [ \nu(V)^2 ] [1+O(e^{-us})].
	\]
	For the second term, By Proposition~\ref{prop:one-point} and Lemma~\ref{lem:Gcut-z}, we deduce that
	\begin{align*}
		\int_V\int_V \Eb[ L_s(z) L_s(w)] 1_{\{|z-w|\le e^{-2s/3}\}} \, dz \, dw 
		&\lesssim \int_V\int_V G^{\cut}_{\Dc}(z) c_* \Qf^* [\Psi_{s}]^{-1} e^{\eta s} 1_{\{|z-w|\le e^{-2s/3}\}} \, dz \, dw \\
		&\lesssim_V C c_* \Qf^* [\Psi_{s}]^{-1} e^{\eta s} |V| e^{-2sd/3} = O_V(e^{-us}),
	\end{align*}
	where we used $\Qf^* [\Psi_{s}]^{-1}=O(1)$ by Lemma~\ref{lem:QPA} and {$\eta=\eta_d<2d/3$ for both $d=2,3$} in the last equality. This finishes the proof.
\end{proof}

\subsection{Proof of Proposition~\ref{prop:one-point}}\label{subsec:BM-op}
{In this subsection, we give a proof of Proposition~\ref{prop:one-point}, the cut-ball approximation of the one-point Green's function for cut points, whose crucial ingredients are the scaling invariance of Brownian motion and the probability measure on NIBM's.}

Let $\wh \Xc_k$ denote the set of
ordered disjoint pairs of curves  $\overline \eta
= (\eta_1,\eta_2)$ starting in $\Dc_k^c$ and ending at their first
visit to $\partial \Dc$; let $\widetilde\state_k$ be the set of such
ordered pairs of  curves that start on $\partial \Dc_k$.  For each $\overline \eta \in \wh \state_k$,
there is a unique $\overline \eta^{(k)} \in \widetilde\state_k$ obtained by starting the
curves at their first visits to $\partial \Dc_k$.
For a probability measure $Q$ on $\wh \state_k$, let $Q^{(k)}$ denote the measure induced on $\widetilde\state_k$ by $Q$.

Suppose $k,s\ge 1$.
If $\overline \eta = (\eta_1,\eta_2)  \in \wh \state_k$ with terminal points $ x_1,x_2
\in \partial \Dc$, let $\mu^{\overline \eta}_s$
denote $\mu_{x_1,x_2}$ restricted to those curves $\omega_*$  such that 
\begin{itemize}
	\item $\omega_*\cap \Dc_{-1}\neq\emptyset$, and it {can be decomposed as} $\omega_*=\beta_1\oplus\omega\oplus [\beta_2]^R$ according to its first and last visits to $\Dc_{-1}$,
	\item under the first condition, we further require $(\eta_1\cup \beta_1)\cap (\eta_2\cup \beta_2)=\emptyset$ and $\omega\subset \Dc_{s/3-1}$.
\end{itemize}
\begin{remark}\label{rmk:additional-scale}
	One should view the above decomposition as the configuration of $\Kt_s(z)$ under the map $\phi(w)=e^{s-1}(w-z)$, where $\gamma_i$ in Definition~\ref{def:cut-bm} is decomposed into $\eta_i\oplus \beta_i$. See Figure~\ref{img:cut-ball-dec} for an illustration. We also remark that the reason why we spare an additional scale $e^{-s+1}$ (or $1$ on the right picture) here, compared with Figure~\ref{img:cut-ball}, is to make Lemma~\ref{lem:sup-Psi} more accessible (the total mass of $\mu_{x_1,x_2}$ will blow up as $|x_1-x_2|$ tends to $0$ if we do not restrict $\omega_*$ to go deep inside, see the first condition on $\omega_*$ above).
\end{remark}
Define the following functions on $\ol \eta$ that will be used later:
\begin{equation}\label{eq:Phidef}
	\Psi_s(\ol \eta) :=\|\mu^{\ol \eta}_s\|\quad\mbox{ and }\quad  
	\wh\Psi_s(\ol \eta) := \mu^{\ol \eta}_s[1_{\beta_1,\beta_2\subset \Dc_{k/2}}].
\end{equation}

\begin{figure}[h!]
	\centering
	\includegraphics[width=.8\textwidth]{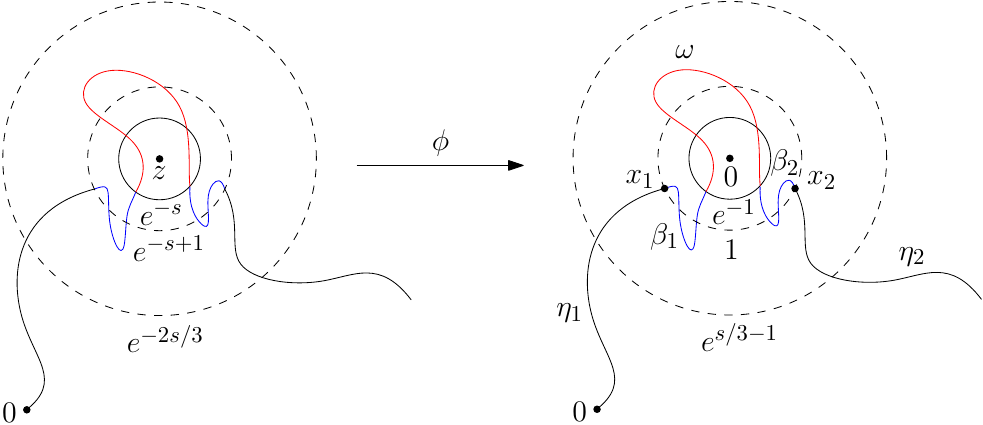}
	\caption{The picture of cut ball under the map $\phi(w)=e^{s-1}(w-z)$.}
	\label{img:cut-ball-dec}
\end{figure}

\begin{lemma}\label{lem:sup-Psi}
	There exists a constant $c>0$ such that for all $k,s\ge 1$,
	$$\sup_{\ol \eta\in\wh\state_{k/2}} \Psi_s(\ol \eta)\le cs^{3-d}.$$
\end{lemma}
\begin{proof}
	When $d=3$, observing that
	\[
	\sup_{\ol \eta\in\wh\state_{k/2}} \Psi_{s}(\ol \eta) \le \sup_{x_1,x_2\in \partial \Dc}\mu_{x_1,x_2} [ 1_{\omega_*\cap \Dc_{-1}\neq \emptyset} ]\le
	\sup_{x_1\in \partial \Dc, y\in \partial \Dc_{-1}} \wt G_{\Rb^3}(x_1,y) \le c.
	\]
	
	However, if $d=2$, the above argument does not work since the Green's function blows up in two dimensions. In this case, we use the fact that the total mass of $\beta_1$ is $1$, and there is a positive portion of $\beta_2$ will stay inside $\Dc_2$ by Lemma~\ref{l:trans_recur}. Therefore, we have 
	\[
	\sup_{\ol\eta\in\wh\state_{k/2}} \Psi_{s}(\ol \gamma) \lesssim \sup_{x_1\in \partial \Dc, y\in \partial \Dc_{-1}} \wt G_{s/3}(x_1,y) \le c\, s.
	\]
	This finishes the proof.
\end{proof}

\begin{lemma}\label{lem:wi-large}
	There exists a constant $c>0$ such that for all $k,s\ge 1$,
	\[
	\sup_{\ol \eta\in \widetilde\Xc_{k/2}} \mu^{\ol \eta}_s[1_{\beta_1\nsubseteq \Dc_{k/2}}] \le cs^{3-d}e^{-k/2}.
	\]
	It is also true if $\beta_1$ is replaced by $\beta_2$.
\end{lemma}
\begin{proof}
	Let $p$ be the probability that a Brownian motion $W$ that starts form $x_1\in \partial\Dc$ (the ending point of $\eta_1$), then hits $\partial \Dc_{k/2}$ before hitting $\Dc_{-1}$, and then returns to $\partial \Dc$. By the gambler's ruin estimate (Lemma~\ref{l:trans_recur}), we know that $p=O(e^{-k/2})$ when $d=3$. This combined with Lemma~\ref{lem:sup-Psi} implies the result when $d=3$.
	As for $d=2$, the probability that $W$ from $x_1$ to $\partial \Dc_{k/2}$ and then returns to $\partial \Dc$ such that it does not intersect $\eta_2$ along the whole way is $O(e^{-k/2})$ by using the Beurling estimate (Proposition~\ref{p:Beurling}) twice (once on the way out and once on the way in). This combined with Lemma~\ref{lem:sup-Psi} finishes the proof of the case when $d=2$.
\end{proof}

\begin{lemma}\label{lem:QPA}
	There exist universal constants $c,C>0$ such that for all $s$,
	$$cs^{3-d}\le \Qf^*[\Psi_{s}]\le Cs^{3-d}.$$
\end{lemma}
\begin{proof}
	The upper bound is a direct consequence of Lemma~\ref{lem:sup-Psi}.
	As for the lower bound, it can be proved by using the reverse separation lemma (see Lemma~\ref{lem:rev-sep-BM}). More precisely, by \eqref{eq:qi_bm_definition-inv} and Lemma~\ref{lem:rev-sep-BM}, we know that there is a positive portion of NIBM's $\overline\eta$ from infinity that have endpoints at distance at least $1/10$ from each other on $\partial\Dc$, for which we have $\Psi_{s}(\overline\eta)\ge cs^{3-d}$. This implies the result. 
\end{proof}

\begin{lemma}\label{lem:TV}
	There exists $c > 0$ such that for any probability measure $Q$ on $\wh \state_k$, and $k,s\ge 1$,
	\begin{equation}\label{eq101}
		\left|Q[\Psi_s]- \Qf^*[\Psi_{s}] \right| \leq  
		c s^{3-d} \left[ e^{-k/2} + d_{\op{TV} }\left(Q^{(k/2)} , \Qf^{*,(k/2)}\right)\right].
	\end{equation}
\end{lemma}
\begin{proof}
	Noting that $\wh\Psi_{s}(\ol \eta^{(k/2)})=\wh\Psi_{s}(\ol \eta)$, we have
	\[
	\Psi_{s}(\ol \eta)\le \Psi_{s}(\ol \eta^{(k/2)})
	\le \wh\Psi_{s}(\ol \eta) + \sum_{i=1}^2 \mu^{\ol \eta^{(k/2)}}_s[1_{\beta_i\notin \ball_{k/2}}].
	\]
	Since $\wh\Psi_{s}(\ol \eta)\le \Psi_{s}(\ol \eta)$, by Lemma~\ref{lem:wi-large}, 
	\[
	| \Psi_{s}(\ol \eta)- \Psi_{s}(\ol \eta^{(k/2)})|
	\le c s^{3-d} e^{-k/2}.
	\]
	By definition, $Q^{(k/2)}[\Psi_{s}]=Q[\Psi_{s}(\ol \eta^{(k/2)})]$. It follows that 
	\[
	|Q[\Psi_{s}]- \Qf^*[\Psi_{s}]|\le c s^{3-d} e^{-k/2}+
	d_{\op{TV} }\left(Q^{(k/2)} , \Qf^{*,(k/2)}\right) \sup_{\ol \eta\in \widetilde\Xc_{k/2}} \Psi_{s}(\ol \eta).
	\]
	This combined with Lemma~\ref{lem:sup-Psi} finishes the proof of the lemma.
\end{proof}

\begin{proof}[Proof of Proposition~\ref{prop:one-point}]
	Let $k=s-b-1$. 
	We let $S=\partial \Dc_{-s+1}(z)$, $D=\Dc \setminus(\Dc_{-s+1}(z)\cup S)$ and $\gf_{s-1}$ be defined as in \eqref{eq:gkdefinition1}.
	Let $Q_D$ be the probability measure on a pair of paths $(\eta_1,\eta_2)$ given by
	\begin{equation}
		Q_D((\eta_1,\eta_2)\in \cdot) =  \frac{1}{\gf_{s-1}}\int_{S } \int_{S } \mu^D_{0,x} \otimes \mu_ {y,\partial \Dc}^D \left[ 1_{\eta_1\cap \eta_2=\emptyset} (\eta_1,\eta_2)\in \cdot \right]\sigma(dx,dy).
	\end{equation}
	We let $Q$ be the pushforward of $Q_D$ under the map $\phi(w)=e^{s-1}(w-z)$ so that $Q$ is a probability measure on $\wh \state_k$. Since the Brownian motion is scaling and translation invariant, we have (See Figure~\ref{img:cut-ball-dec} and note the scaling covariance of path measures in $\Rb^3$)
	\[
	e^{-(s-1)(d-2)} \Pb \{ \widetilde K_{s} \} / \gf_{s-1} =Q [\mu^{\ol \eta}_s[1_{\beta_1,\beta_2\subset \phi(D)}]].
	\]
	By Lemma~\ref{lem:wi-large}, 
	\begin{equation*}
		| Q[\Psi_{s}]-Q [\mu^{\ol \eta}_s[1_{\beta_1,\beta_2\subset \phi(D)}]] |=O(s^{3-d}e^{-k/2}).
	\end{equation*}
	This combined with Lemma~\ref{lem:TV} shows 
	\begin{equation}\label{eq:muQ}
		| e^{-(s-1)(d-2)} \Pb \{ \widetilde K_{s} \} / \gf_{s-1} - \Qf^* [\Psi_{s}] | = O(e^{-us}).
	\end{equation}
	It follows that
	\begin{align*}
		\Eb[ L_{s}(z) ]&\;\;= \;c_* \Qf^* [\Psi_{s}]^{-1} e^{\eta (s-1)} \Pb \{ \widetilde K_{s} \}\\
		&\stackrel{\eqref{eq:muQ}}{=} c_* \Qf^* [\Psi_{s}]^{-1} e^{\eta (s-1)} e^{(s-1)(d-2)} \gf_{s-1} (\Qf^* [\Psi_{s}]+O(e^{-us}))\\
		&\,\stackrel{\eqref{eq:Gcut1}}{=} G_{\Dc}^{\cut}(z) [1+O(\Qf^* [\Psi_{s}]^{-1} e^{-us})] [1+O(e^{-u's})].
	\end{align*}
	By Lemma~\ref{lem:QPA}, we conclude the proof.
\end{proof}

\section{Sharp asymptotics for random walk cut-point Green's function \label{sec:pt_estim}}
{In this section, we will give in Theorems \ref{t:one_pt} and \ref{t:two_pt} sharp asymptotics for one- and two-point Green's function for the cut points of the simple random walk.}

Fix $z\in \Dc\setminus\{0\}$. Write
\begin{equation}\label{eq:dz-zn}
d_z=\dist(0,z,\partial\Dc)\mbox{ and }z_n=\floo{e^{n}z}
\end{equation}
and let
$$
\Ac_n(z):=\Big\{\mbox{$z_n$ is a cut point of the random walk path $S[0,\tau_n]$}\Big\}.
$$

The following theorem gives the sharp asymptotics on the one-point function.
\begin{theorem}
	\label{t:one_pt} There exists a universal constant $c_1>0$ such that the following holds for all $z\in \Dc$ with $d_z=\dist(0,z,\partial\Dc)\ge e^{-n/6}$,
	\[
	c_{1}\, e^{\eta n}\, \Pb\{\Ac_n(z)\} \simeq G^{\cut}_{\Dc}(z).
	\]
\end{theorem}
The following theorem gives the sharp asymptotics on two-point function.
\begin{theorem}
	\label{t:two_pt} With the same constant $c_1$ as in Theorem~\ref{t:one_pt}, for all $V\in\Dc$, $z,w\in V$ with $|z-w|\ge e^{-n/6}$,
	\[
	c_{1}^{2}\, e^{2\eta n}\, \Pb\{\Ac_n(z)\cap \Ac_n(w)\}\simeq_V G_{\Dc}^{\cut}(z,w).
	\]
\end{theorem}

\subsection{One-point estimate}
This section is devoted to proving Theorem~\ref{t:one_pt}, the sharp one-point estimate. 
To begin with, we present an up-to-constants estimate for $\Pb\{\Ac_n(z)\}$, which follows from Lemma~\ref{lem:biop} directly.
\begin{lemma}\label{lem:utc}
	For any $z\in \Dc\setminus \{0\}$, we have 
	\[
	\Pb\{\Ac_n(z)\} \asymp a(z) e^{-\eta n},
	\]
	where $a(z)$ is defined in \eqref{eq:g-z} and of the same order as $G^{\cut}_{\Dc}(z)$.
\end{lemma}

For any $\overline\beta=(\beta_1,\beta_2)\in \Xc_{n/6}$ (see \eqref{eq:tXr} for the definition) with ending points $y_i:=\beta_i(t_{\beta_i})$ for $i=1,2$, we use the following notation to denote the non-intersecting (abbreviated as ``NI'' below) event for BM's with initial configuration given by $(\Kc n)$-sausage of $\ol\beta$ (recall the definition of sausage in \eqref{eq:sausage}): 
\[
\wt A^\Delta_{n/2}(\overline\beta):= \wt A_{n/2}(\Theta(\overline\beta,\Kc n), \ol y ),
\]
where ${\cal K}$ is the  constant from the strong approximation (see Corollary~\ref{cor:sa}), $\ol y=(y_1,y_2)$, and the NI event $\wt A_{n/2}(\cdot,\cdot)$ is defined in \eqref{eq:Ar}.

Recall the definition of NI events for RW's $A_{m}$ in \eqref{eq:tilde-A'} without initial configuration as well as  $A_{m}(\overline\beta)$ in \eqref{eq:tilde-A} with initial configuration $\overline\beta$. 
The following proposition has been proved in \cite{2sidedwalk}, see (3.7), (3.9), (3.12) and Proposition 3.19 therein. 

\begin{proposition}\label{prop:An4}
	There is a set of ``nice'' configurations in $\Xc_{n/6}$, denoted by $\mathrm{NICE}_{n/6}$, such that if we let 
	$$\Nc_{n/6}:=\big\{\baS[0,\tau_{n/6}]\in\mathrm{NICE}_{n/6}\big\},$$ then 
	\begin{equation}\label{eq:nice34}
		\Pb\{ A_{n/2}\}\simeq \Pb\{ A_{n/2}\,\cap\, \Nc_{n/6} \}.
	\end{equation}
	Moreover, for any $\overline\beta\in \mathrm{NICE}_{n/6}$, we have 
	\begin{equation}\label{eq:A-tA}
	\Pb\{  A_{n/2}(\overline\beta) \} \simeq
	\Pb\{ \wt A^\Delta_{n/2}(\overline\beta) \}.
	\end{equation}
\end{proposition}
\begin{remark}\label{rem:nice}
	In fact, the set $\mathrm{NICE}_{n/6}$ is exactly delineated by events $F_m, G_m$ and $H_m$ defined in Section 3.2 of \cite{2sidedwalk} with $m=n/2$. A crucial property of such nice configuration is that they are well-separated (see the paragraph just before Lemma~\ref{lem:sep-SRW}). Since we do not need the exact definition of such events at this stage, we omit it for simplicity.
\end{remark}

In this section, we always assume that $z\in \Dc$ with $d_z=\dist(0,z,\partial\Dc)\ge e^{-n/6}$. 
Recall that $z_n=\floo{e^{n}z}$ in \eqref{eq:dz-zn} so that $\Bc_{5n/6}(z_n)\subset \Bc_n$.
For $\overline\zeta=(\zeta_1,\zeta_2)\in  \Xc_{n/6}(z_n)$ (recall this is the set of NI paths with center $z_n$) with starting point $z_n$ and ending points $(x_1,x_2)\in \partial\Bc_{n/6}(z_n)\times\partial\Bc_{n/6}(z_n)$, define the set of NI discrete paths with initial configuration $\ol\zeta$ that end at $0$ and $\partial\Bc_n$ respectively by
\begin{equation}\label{eq:Af'}
	A_{0,\partial\Bc_n}(\overline\zeta):=\{ (\lambda_1,\lambda_2)\in \Gamma^{\Bc_n}_{x_1,0}\times\Gamma^{\Bc_n}_{x_2,\partial\Bc_n}: \lambda_1\cap(\zeta_2\cup\lambda_2)=\lambda_2\cap(\zeta_1\cup\lambda_1)=\emptyset  \},
\end{equation} 
and its continuous analogue by
\begin{equation}\label{eq:Af}
\wt A^{\Delta}_{0,\partial\Dc_n}(\overline\zeta):=\left\{
	(\gamma_1,\gamma_2)\in \wt\Gamma^{\Dc_n}_{x_1,0}\times\wt\Gamma^{\Dc_n}_{x_2,\partial\Dc_n}: \gamma_1\cap(\wt\zeta_2\cup\gamma_2)=\gamma_2\cap(\wt\zeta_1\cup\gamma_1)=\emptyset \right\},
\end{equation}
where we write $\wt\zeta_i:=\Theta(\zeta_i,\Kc n)$ for the $(\Kc n)$-sausage of $\zeta_i$ for $i=1,2$ in the rest of this section. We also use the notation $\wt\zeta:=(\wt\zeta_1,\wt\zeta_2)$.

\begin{remark}\label{rmk:Ablow-up}
	Note that $\wt A^{\Delta}_{0,\partial\Dc_n}(\overline\zeta)$ can be viewed as the blow-up of the event $\ol A (e^{-n}\wt\zeta,e^{-n}\ol x)$ defined in \eqref{eq:bar A} with $s=n/2$.
\end{remark}

Let $\text{NICE}_{n/6}(z)=z_n+\text{NICE}_{n/6}$. Suppose $z_n$ is a cut point of $\lambda$ with $\lambda:=S[0,\tau_n]$, then we can extract a pair of paths $\overline\zeta\in  \Xc_{n/6}(z_n)$ from $\lambda$ by starting from $z_n$ and tracing along $\lambda$ chronologically and reverse-chronologically respectively at the first visits to $\partial\Bc_{n/6}(z_n)$, and $\mathcal{N}_{n/6}(z)$ is defined as the event that $\overline\zeta\in \text{NICE}_{n/6}(z)$. 
In other words, we use the following (unique) decomposition on the event $z_n\in \Ac_{\lambda}$:
\begin{equation}\label{eq:lbd-dec}
	\lambda=[\lambda_1]^R\oplus [\zeta_1]^R\oplus \zeta_2\oplus \lambda_2
\end{equation}
such that $(\zeta_1,\zeta_2)\in\Xc_{n/6}(z_n)$ and $(\lambda_1,\lambda_2)$ is a pair in $A_{0,\partial\Bc_n}(\overline\zeta)$.
Recall that $\nu^{\Bc_n}_{\cdot,\cdot}$ and $\mu^{\Dc_n}_{\cdot,\cdot}$ are the random walk path measure and the Brownian path measure introduced in Section~\ref{subsec:pm}, respectively. 

\begin{proposition}\label{prop:bridge-nice}
	Let $z\in\Dc$ with $d_z\ge e^{-n/6}$.
	\begin{enumerate}
		\item[(1)] With the definition given above, we have 
		\begin{equation}\label{eq:Anz}
		\Pb\{\Ac_n(z)\}\simeq\Pb\{\Ac_n(z)\,\cap\, \mathcal{N}_{n/6}(z)\}.
		\end{equation}
		\item[(2)] For all $\overline\zeta\in \mathrm{NICE}_{n/6}(z)$, we have 
		\begin{equation}\label{eq:A0p}
		\nu^{\Bc_n}_{x_1,0}\otimes\nu^{\Bc_n}_{x_2,\partial\Bc_n}\{ A_{0,\partial\Bc_n}(\overline\zeta) \} \simeq \mu^{\Dc_n}_{x_1,0}\otimes\mu^{\Dc_n}_{x_2,\partial\Dc_n}\{ \wt A^{\Delta}_{0,\partial\Dc_n}(\overline\zeta) \}.
		\end{equation}
	\end{enumerate}
	
\end{proposition}

We first show how to prove Theorem~\ref{t:one_pt} by using the above two propositions, together with Proposition~\ref{prop:gbar-gs} by setting $s=n/2$ (see Remark~\ref{rmk:Ablow-up}). 
\begin{proof}[Proof of Theorem~\ref{t:one_pt} assuming Proposition  \ref{prop:bridge-nice}]
    From \eqref{eq:lbd-dec},
	we obtain that
	\begin{equation}\label{eq:one}
	\Pb\{\Ac_n(z)\}=\sum_{\overline\zeta\in   \Xc_{n/6}(z_n)} \nu^{\Bc_n}_{x_1,0}\otimes\nu^{\Bc_n}_{x_2,\partial\Bc_n}\{ A_{0,\partial\Bc_n}(\overline\zeta) \} p(\overline\zeta),
	\end{equation}
	where 
	\begin{equation}\label{eq:p-zeta}
	p(\overline\zeta):=\Pb^{z_n,z_n}\{ \baS[0,\tau_{n/6}]=\overline\zeta \}.
	\end{equation}
	By Proposition~\ref{prop:bridge-nice}, we obtain that 
	\begin{align*}
		  \Pb\{\Ac_n(z)\}\ 		\overset{\eqref{eq:Anz}}{\simeq} \ & \sum_{\overline\zeta\in \text{NICE}_{n/6}(z)}  \nu^{\Bc_n}_{x_1,0}\otimes\nu^{\Bc_n}_{x_2,\partial\Bc_n}\{ A_{0,\partial\Bc_n}(\overline\zeta) \}  \, p(\overline\zeta)\\
		\overset{\eqref{eq:A0p}}{\simeq} \ &\sum_{\overline\zeta\in \text{NICE}_{n/6}(z)}  \mu^{\Dc_n}_{x_1,0}\otimes\mu^{\Dc_n}_{x_2,\partial\Dc_n}\{ \wt A^{\Delta}_{0,\partial\Dc_n}(\overline\zeta) \} \, p(\overline\zeta).
	\end{align*}
	
	Plugging $s=n/2$ in Proposition~\ref{prop:gbar-gs} and rescaling by $e^n$ (see Remark~\ref{rmk:Ablow-up}), we obtain that for all $\overline\zeta\in \text{NICE}_{n/6}(z)$,
	\begin{equation}\label{eq:sim}
		e^{\xi n/2} e^{(d-2)n} \mu^{\Dc_n}_{x_1,0}\otimes\mu^{\Dc_n}_{x_2,\partial\Dc_n}\{ \wt A^{\Delta}_{0,\partial\Dc_n}(\overline\zeta) \}
		\simeq c_2\, G_{\Dc}^{\cut}(z)\, \Pb\{ \wt A^{\Delta}_{n/2}(\overline\zeta) \}.
	\end{equation} 
By Proposition~\ref{prop:An4}, 
	\begin{align}\label{eq:sum1}
		\sum_{\overline\zeta\in \text{NICE}_{n/6}(z)} \Pb\{ \wt A^{\Delta}_{n/2}(\overline\zeta) \} \, p(\overline\zeta)
		 \overset{\eqref{eq:A-tA}}{\simeq} \sum_{\overline\zeta\in \text{NICE}_{n/6}(z)} \Pb\{ A_{n/2}(\overline\zeta) \} \, p(\overline\zeta) 
		 \overset{\eqref{eq:nice34}}{\simeq} \Pb\{ A_{n/2} \}
		 \simeq q\, e^{-\xi n/2},
	\end{align}
    where we used Proposition \ref{prop:sharp-rw} in the last inequality.	
    Finally, we get 
	\[
	e^{(d-2)n} \Pb\{\Ac_n(z)\}
	\simeq c_2 \, G_{\Dc}^{\cut}(z)  \, q\, e^{-\xi n}.
	\]
	This concludes the proof of the theorem by choosing $c_1=(c_2\, q)^{-1}$.
\end{proof}

\medskip 

The rest of this section is devoted to proving Proposition~\ref{prop:bridge-nice}. 
We want to use the strong approximation (see Corollary~\ref{cor:sa}) to prove it. 
We can do it for the second marginal $\nu^{\Bc_n}_{x_2,\partial\Bc_n}$ since this is just the law of a SRW started from $x_2$ and stopped upon reaching $\partial\Bc_n$, see \eqref{eq:S-stop}. 
However, we do not have a nice coupling for the first marginal, the path measure $\nu^{\Bc_n}_{x_1,0}$ (in fact, the KMT coupling is still in force for two dimensional bridges, see Corollary 3.2 in \cite{RWLS}; however, the analogue for bridges in three dimensions have not been proved yet). This issue will make our argument more involved. 

To overcome the issue, we decompose the path sampled from $\nu^{\Bc_n}_{x_1,0}$ into two parts: the first part is a SRW started from $x_1$ and stopped upon reaching the mesoscopic ball $\Bc_{n/6}$ about the origin; and the second part is a path sample from the path measure between the endpoint of the first part and $0$ restricting to a certain NI event. On the one hand, the first part can be coupled with the counterpart of BM by the strong approximation. On the other hand, we can compare the total mass of the second part between the discrete and the continuous, which are just the Green's function in a {``pricked''} ball, i.e. (say, in the discrete case), $\Bc_n$  with the union of $\zeta_2$ and the trace of SRW started from $x_2$ stopped on reaching $\partial\Bc_n$ subtracted (recall that $\zeta_2$ is the second marginal of the initial configuration and $x_2$ is its endpoint). Since with high probability SRW from $x_2$ will not get close the the ball $\Bc_{n/6}$ on the NI event, we can show that the Green's functions are close to each other (see Lemma~\ref{lem:Gset} for the more involved case $d=2$). 

\medskip

We give some notation first, and one can see {Figure}~\ref{img:Ec} for an illustration. The following notion applies to $i=1,2$ respectively.
\begin{itemize}
	\item Let $\overline\zeta=(\zeta_1,\zeta_2)$ be a pair of paths in $ \text{NICE}_{n/6}(z)$ with ending points $(x_1,x_2)$. Write $\wt\zeta:=(\wt\zeta_1,\wt\zeta_2)$ for the $(\Kc n)$-sausage of $\ol\zeta$ as before.
	\item Let $S_i$ and $W_i$ be the simple random walk and the Brownian motion started from $x_i$, respectively.
	\item Let $\xi_1=S_1[0,\tau_{n/6}]$, $\xi_2=S_2[0,\tau_{n}]$ and 
	$\widetilde\xi_1=W_1[0,T_{n/6}]$, $\widetilde\xi_2=W_2[0,T_{n}]$.
	\item Write $\xi_i=\eta_i\oplus \eta_{i+2}$ where $\eta_i$ is the part of $\xi_i$ that started from $x_i$ stopped upon reaching $B(z_n,d_z e^{n-1})$.
	\item Let $\widetilde \eta_i$ and $\widetilde \eta_{i+2}$ be defined as above with $\xi_i$ replaced by $\wt\xi_i$ and $B(z_n,d_z e^{n-1})$ replaced by $D(z_n,d_z e^{n-1})$.
    \item Let $\Ec_n(\overline\zeta)$ be the event that $\xi_1\cap(\zeta_2\cup\xi_2)=\xi_2\cap(\zeta_1\cup\xi_1)=\emptyset$ and $\eta_3\subseteq \Bc_n$. 
    \item Let $\widetilde \Ec_n(\overline\zeta)$ be the event that $\wt\xi_1\cap(\wt\zeta_2\cup\wt\xi_2)=\wt\xi_2\cap(\wt\zeta_1\cup\wt\xi_1)=\emptyset$ and $\wt\eta_3\subseteq \Dc_n$. 
\end{itemize}

\begin{figure}[h!]
	\centering
	\includegraphics[width=.5\textwidth]{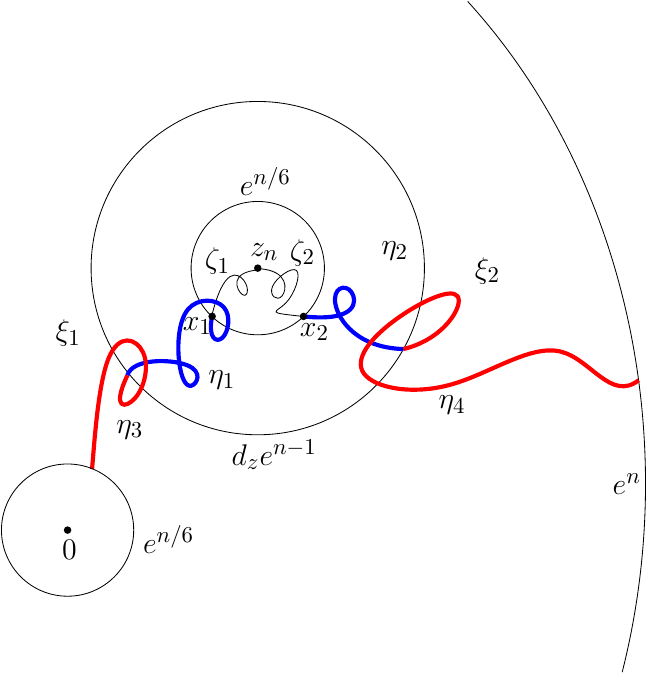}
	\caption{$(\zeta_1,\zeta_2)$ is in black, $(\eta_1,\eta_2)$ is in blue,  $(\eta_3,\eta_4)$ is in red, and $(\xi_1,\xi_2)$ is the pair of concatenated thick curves that satisfies $\xi_i=\eta_i\oplus \eta_{i+2}$.}
	\label{img:Ec}
\end{figure}

\begin{lemma}\label{lem:utc-E}
	For all $\overline\zeta\in \mathrm{NICE}_{n/6}(z)$ with $d_z\ge e^{-n/6}$, we have 
	\[
	\Pb\{ \Ec_n(\overline\zeta) \} \asymp a(z) n^{d-3} e^{-5\eta n/6}, \quad\quad 
	\Pb\{ \widetilde\Ec_n(\overline\zeta) \} \asymp a(z) n^{d-3} e^{-5\eta n/6},
	\]
	where $a(z)$ is defined in \eqref{eq:g-z}.
\end{lemma}
\begin{proof}

	We will only illustrate the lower bound. In fact, the upper bound is much easier to obtain. 
	We first deal with the case when $|z|>1/2$, that is to say, $B(z_n,d_z e^{n-1})$ is close to the boundary of $\partial\Dc$. 
	We use the same proof strategy as Lemma~\ref{lem:biop}. 
	\begin{itemize}
		\item The total mass of $(\eta_1,\eta_2)$ restricted to the event $\eta_1\cap(\zeta_2\cup\eta_2)=\eta_2\cap(\zeta_1\cup\eta_1)=\emptyset$ and they are well-separated at $\partial B(z_n,d_z e^{n-1})$ is of order $d_z^{-\xi} e^{-5\xi n/6}$, noting that the nice configuration $\overline\zeta$ is well-separated (see Remark~\ref{rem:nice}).
		\item The total mass of $\eta_4$ staying in a well-chosen tube (of order $d_z e^{n-1}$) is of order $1$. 
		\item Let $\eta_3^1$ be the simple random walk started from the endpoint of $\eta_1$ that hits $\Bc_{n-3}$ before exiting $\Bc_n\setminus (\zeta_2\cup\eta_2\cup\eta_4)$, which occurs with probability of order $d_z$ by a version of the gambler's ruin estimate (see (7.24) in \cite{RWintro}).
		\item let $\eta_3^2$ be the simple random walk started from the endpoint of $\eta_3^1$ and hits $\Bc_{n/6}$ before exiting $\Bc_{n-2}$, which occurs with probability of order $n^{d-3} e^{-5(d-2)n/6}$.
		\item The concatenation $(\eta_1\oplus\eta_3^1\oplus \eta_3^2, \eta_2\oplus\eta_4)$ satisfies the event $\Ec_n(\overline\zeta)$.
	\end{itemize}
    We multiple the above total masses to obtain the desired lower bound $d_z^{1-\xi} n^{d-3} e^{-5\eta n/6}$.
    Next, we deal with the case $|z|\le 1/2$.
    \begin{itemize}
    	\item Let $(\lambda_1,\lambda_2)$ be a pair of NIRW's from $(x_1,x_2)$ to $\partial B(z_n,19d_z e^n/20)$ such that the quality (see \eqref{eq:RW-quality} for the definition where we should translate the whole configuration here by $-z_n$) is greater than $1/9$ and the endpoint of $\lambda_1$ is in $B(0,d_z e^n/10)\cap\partial B(z_n,19d_z e^n/20)$. The total mass of such pairs is of order $d_z^{-\xi} e^{-5\xi n/6}$.
    	\item Let $\lambda_3$ be the simple random walk started from the endpoint of $\lambda_1$ and hits $\Bc_{n/6}$ before exiting $B(0,e^n/9)$, which occurs with probability of order $[\log(d_z e^n/9)-n/6]^{-1}\asymp n^{-1}$ when $d=2$, $|z|\ge e^{-n/6}$ and of order $d_z^{-1} e^{-5n/6}$ when $d=3$.
    	\item Let $\lambda_3$ be the simple random walk started from the endpoint of $\lambda_2$ and hits $\partial\Bc_n$ before hitting $\zeta_1\cup\lambda_1\cup\lambda_3$ which occurs with probability of order $[\log(d_z^{-1})]^{d-3}$.
    	\item The concatenation $(\lambda_1\oplus\lambda_3, \lambda_2\oplus\lambda_4)$ satisfies the event $\Ec_n(\overline\zeta)$.
    \end{itemize}
	The multiplication gives the order $d_z^{-\eta} [\log(d_z^{-1})]^{d-3} n^{d-3} e^{-5\eta n/6}$. This finishes the proof.
\end{proof}

	In fact, we can use Lemma~\ref{lem:utc} to show Lemma~\ref{lem:utc-E} directly, however, in the above proof, we choose to make a detour and use the path decomposition idea to prove it. In this way, one can have a clear picture in mind how this order comes from and why $a(z)$ takes such a form. More importantly, we can analyze the path piece by piece thanks to the strong Markov property, which will become essential for us to prove the next proposition.

\begin{proposition}\label{prop:sharp-E}
	For all $\overline\zeta\in \mathrm{NICE}_{n/6}(z)$ with $d_z\ge e^{-n/6}$, we have 
	\[
		\Pb\{ \Ec_n(\overline\zeta) \} \simeq \Pb\{ \widetilde \Ec_n(\overline\zeta) \}.
	\]
\end{proposition}

\begin{proof}
	By Corollary~\ref{cor:sa}, there is a constant ${\cal K}>0$ such that we can couple $S_i$ and $W_i$ via a common probability $\Pb$ under which the following event $F_0$ holds with probability at least $1-e^{-10n}$:
	\begin{equation}\label{eq:coup}
	F_0:=\left\{|S_i(t)-W_i(t)|\le {\cal K}n, \quad 0\le t\le \tau_{n+1}(S_i)\vee  T_{n+1}(W_i), \quad i=1,2.\right\}
	\end{equation}
	Recall various definitions above Lemma \ref{lem:utc-E}. We write $\Ec$ and $\widetilde\Ec$ for $\Ec_n(\overline\zeta)$ and $\widetilde \Ec_n(\overline\zeta)$, respectively.
	Under this coupling, we will estimate the probability of the event $ \Ec\setminus \widetilde\Ec$. Let $\Ec^1$ be the corresponding event of $\Ec$ with $(\xi_1,\xi_2)$ replaced by $(\eta_1,\eta_2)$. Let $\widetilde\Ec^1$ be the corresponding event of $\widetilde\Ec$ with $(\widetilde\xi_1,\widetilde\xi_2)$ replaced by $(\widetilde\eta_1,\widetilde\eta_2)$. We refer readers to Figure~\ref{img:Ec} for an illustration of various paths. On the event $\Ec\setminus \widetilde \Ec$, one of the following six events will happen: 
	\begin{itemize}
		\item $F_1$: $(\Ec^1\setminus \widetilde \Ec^1)\cap \{\eta_3\subseteq \Bc_n\}$.
		\item $F_2$: $\Ec^1\cap \{\eta_3\subseteq \Bc_n\}\cap\{(\eta_3\cup\eta_4)\cap B(z_n, e^{n/6}+\Kc n)\neq \emptyset\}\cap\{\eta_3\cap\eta_2=\emptyset, \eta_4\cap\eta_1=\emptyset\}$.
		\item $F_3$: $\Ec^1\cap \{\eta_3\subseteq \Bc_n\}\cap \{\eta_4\cap \Bc_{n/3}\neq\emptyset\}\cap\{\eta_4\cap \eta_3=\emptyset\}$.
		\item $F_4$: $\Ec\cap\{\eta_4\cap\Bc_{n/3}=\emptyset\}\cap\{ \widetilde\eta_3\cap\widetilde\xi_2\neq\emptyset \}$.
		\item $F_5$: $\Ec\cap\{ \widetilde\eta_4\cap\widetilde\eta_1\neq\emptyset \}$.
		\item $F_6$: $\Ec\setminus \{\widetilde\eta_3\subseteq \Dc_n\}$.
	\end{itemize}
We will only prove the case when $|z|\ge 1/2$ and the case $|z|<1/2$ can be shown in a similar way.
We will show that there exists $u>0$ such that 
\begin{equation}\label{eq:Fi}
	\Pb\{ F_i \} \lesssim a(z) e^{-un} e^{-5\eta n/6} \quad \text{ for all } 1\le i\le 6.
\end{equation}
The goal is to obtain the extra cost $O(e^{-un})$, for which the path decomposition idea (or the Markov property) plays a key role as illustrated in the proof of Lemma \ref{lem:utc-E} before.

It is easy to see that \eqref{eq:Fi} holds for $F_1$ by Proposition~\ref{prop:An4}, and also for $F_2$ and $F_3$ by the strong Markov property of SRW and the Beurling estimate for $d=2$ (Proposition~\ref{p:Beurling}) or the gambler's ruin estimate for $d=3$ (Lemma~\ref{l:trans_recur}), noting that $d_z e^{n} \ge e^{5n/6}$ since $d_z\ge e^{-n/6}$. 

We now focus on $F_4$. We decompose $S_i$ in a different way to incorporate the coupling error in \eqref{eq:coup}.
\begin{itemize}
	\item Let $\xi_1^1$ be the part of $S_1$ started from $x_1$ stopped at the first hitting of $\partial B(z_n,d_z e^{n-1}-{\cal K}n)$.
	\item Let $\xi_1^2$ be the part of  $S_1$ from the endpoint of $\xi_1^1$ to the first hitting of $\Bc_{n/6}$.
	\item Let $\xi_1^+$ be the part of  $S_1$ from the endpoint of $\xi_1^2$ to the first hitting of $B(0,e^{n/6}-{\cal K}n)$.
	\item Let $\xi_2^1$ be the part of $S_2$ started from $x_2$ stopped at the first hitting of $\partial B(z_n,d_z e^{n-1}-{\cal K}n)$.
	\item Let $\xi_2^2$ be the part of  $S_2$ from the endpoint of $\xi_2^1$ to the first hitting of $\Bc_{n}$.
	\item Let $\xi_{2}^+$ be the part of $S_2$ started from the endpoint of $\xi_2^2$ to the exit of $B(0,e^n+{\cal K}n)$.
\end{itemize} 
In the above decomposition, we have $\xi_i=\xi_i^1\oplus \xi_i^2$, and $\xi_i^+$ is a small additional part to $\xi_i$.
On the event $F_0\cap F_4$, noting that $\widetilde\eta_3\subseteq \Theta(\xi_1^2\cup \xi_1^+, \Kc n)$ and $\widetilde\xi_2\subseteq \Theta(\xi_2\cup\xi_{2}^+, \Kc n)$ and $\widetilde\eta_3\cap\widetilde\xi_2\neq\emptyset$, we have 
\begin{equation}\label{eq:+}
	\dist (\xi_1^2\cup \xi_1^+,\xi_2\cup\xi_{2}^+)\le 2{\cal K}n.
\end{equation}
Using the gambler's ruin estimate (see (7.24) in \cite{RWintro}), we know that 
\begin{equation*}
	\Pb\{ \xi_{2}^+\subset \Bc_{n/10}(S_2(\tau_n)) \} \ge 1-e^{-n/8}, \quad 
	\Pb\{ \xi_{1}^+\subset \Bc_{n/10}(S_1(\tau_{n/6})) \} \ge 1-e^{-n/8}.
\end{equation*}
Therefore, conditioning on $F_0\cap F_4$, with probability greater than $1-2e^{-n/8}$, 
\begin{equation*}
	\dist (\xi_1^2,\xi_2)\le e^{n/9}.
\end{equation*}
Letting $\Ec^2$ be the corresponding event of $\Ec$ with $(\xi_1,\xi_2)$ replaced by $(\xi_1^1,\xi_2^1)$ which has probability $O(d_z^{-\xi}e^{-5\xi n/6})$, and using the strong Markov property together with {the freezing lemma (Lemma \ref{l:freezing}) when $d=3$ or the Beurling estimate (Proposition~\ref{p:Beurling}) when $d=2$}, we obtain that 
\begin{align*}
	\Pb\{ F_4 \} &\lesssim 
	\Pb\{ F_0\cap F_4 \} \\
	&\lesssim 
	\Pb\{ F_0\cap F_4\cap \{\dist (\xi_1^2,\xi_2)\le e^{n/9}\} \} \\
	 &\le
	\Pb\{ \Ec^2\cap \{\xi_2\cap\Bc_{n/3}=\emptyset\} \cap \{0<\dist (\xi_1^2,\xi_2)\le e^{n/9}\}\cap \{\xi_1^2\subseteq \Bc_n\} \} \\
	&\lesssim 
	a(z) e^{-un} e^{-5\eta n/6}.
\end{align*}
More precisely, we decompose $\xi_1^2$ according to its distance to $0$ as $\gamma_0\oplus\gamma_1\oplus\cdots\oplus\gamma_k$, where $k=\lfloor\log(d_z e^{n-1})\rfloor-n/6$, $\gamma_0$ is the part from the starting point of $\xi_1^2$ to its first hitting of $\partial \Bc_{n/6+k}$, and $\gamma_i$ is the part from $\partial \Bc_{n/6+k-i+2}$ to first hitting of $\partial \Bc_{n/6+k-i+1}$ for any $1\le i\le k$. By the freezing lemma (Lemma \ref{l:freezing}), except for an event of probability $O(e^{-10n})$, $\xi_2$ satisfies the property that
\begin{gather*}
\Pb\{ \gamma_{i+1}\cap\xi_2=\emptyset, \dist (\gamma_i,\xi_2)\le e^{n/9} \mid \xi_2 \} \lesssim e^{-un} \quad \text{ for all } 0\le i\le k-1, \\
\quad \text{ and }\quad  \Pb\{ 0<\dist (\gamma_k,\xi_2)\le e^{n/9} \mid \xi_2 \}\lesssim e^{-un}.
\end{gather*}
In a word, at each scale when $\dist (\xi_1^2,\xi_2)\le e^{n/9}$, we get an extra cost. 
Therefore, noting that $k\le 5n/6$, by the union bound, we have
\begin{align*}
&\Pb\{ \{0<\dist (\xi_1^2,\xi_2)\le e^{n/9}\}\cap \{\xi_1^2\subseteq \Bc_n\}  \}\\
\lesssim \;& \sum_{i=0}^k d_z n^{d-3} e^{-5(d-2)n/6} e^{-un} + e^{-10n} \lesssim d_z e^{-5(d-2)n/6} e^{-u'n}.
\end{align*}
We also refer the reader to Lemma 3.1 of \cite{lawler2000intersection} for a similar derivation. 

The probability of $F_5$ can also be bounded in a similar fashion by using the freezing lemma (Lemma \ref{l:freezing}) as above. We omit the details.

As for $F_6$, note that $F_0\cap \{\widetilde\eta_3\subseteq \Dc_n\}^c$ implies that $\dist(\xi_1^2\cup\xi_1^+,\partial\Bc_n)\le {\cal K}n$. Using the gambler's ruin estimate,
\[
\Pb\{ 0<\dist(\xi_1^2,\partial\Bc_n)\le {\cal K}n \} \lesssim e^{-un} \Pb\{ \eta_3\subseteq \Bc_n \}, 
\]
and also 
\[
\Pb\{ \dist(\xi_1^+,\partial\Bc_n)\le {\cal K}n \} \lesssim e^{-un}.
\]
This combined with the strong Markov property shows that \eqref{eq:Fi} holds for $F_6$.
Now, we have finished the proof of \eqref{eq:Fi}, which implies that 
\[
\Pb\{ \Ec\setminus \widetilde \Ec \} \lesssim a(z) e^{-un} e^{-5\eta n/6}.
\]
Using a similar argument, we can also conclude 
\[
\Pb\{ \widetilde \Ec\setminus \Ec \} \lesssim a(z) e^{-un} e^{-5\eta n/6}.
\]
This combined with Lemma~\ref{lem:utc-E} concludes the proof of the proposition.
\end{proof}

As discussed between Proposition \ref{prop:bridge-nice} and Lemma \ref{lem:utc-E}, we need to compare the Green's function for both the SRW and BM in a ``pricked'' ball. The following lemmas deal with the more involved case of $d=2$.

\begin{lemma}\label{lem:logS}
	Let $d=2$. Let $\gamma$ (resp.\ $\widetilde\gamma$) be a discrete (resp.\ continuous) path in $\Bc_n$ (resp.\ $\Dc_n$) from the interior to the boundary such that $\dist(\gamma,0)\ge e^{n/3}$ (resp.\ $\dist(\widetilde\gamma,0)\ge e^{n/3}$). Suppose $d_H(\gamma,\widetilde\gamma)\le e^{n/90}$ (see \eqref{eq:d-H} for the definition of Hausdorff distance). Let $\sigma$ be the exit time of $\Bc_n\setminus \gamma$ by the simple random walk $S$. Let $\widetilde\sigma$ be the exit time of $\Dc_n\setminus \widetilde\gamma$ by the Brownian motion $W$. Then, for any $y\in\Bc_{n/3}$,
	\[
	| \Eb^y[ \log|S_{\sigma}| ] - \Eb^y[ \log|W_{\widetilde\sigma}| ] | \lesssim e^{-n/10}.
	\]
\end{lemma}
\begin{proof}
	Let $G$ be the event that $|S_{\sigma}-W_{\widetilde\sigma}|\le e^{2n/9+n/90}$.
	Using Proposition 3.3 of \cite{MR2191635}, we can couple $S$ and $W$ (both starting from $y$) under $\Pb$ such that $\Pb\{G^c\}\lesssim e^{-n/9}$. Since $\dist(\gamma,0)\ge e^{n/3}$, we have $|S_{\sigma}|\ge e^{n/3}$. Thus, on the event $G$,
	\[
	| \log|S_{\sigma}|- \log|W_{\widetilde\sigma}| | \lesssim \frac{e^{2n/9+n/90}}{e^{n/3}}=e^{-n/10}.
	\]
	Moreover, noting that ${\log}|S_{\sigma}|\le n, \log|W_{\widetilde\sigma}|\le n$, we have
	\[
	| \Eb^y[ \log|S_{\sigma}| 1_{G^c}] | \lesssim n e^{-n/9} \lesssim e^{-n/10},
	\]
	and a similar bound holds for $\log|W_{\widetilde\sigma}|$. This finishes the proof.
\end{proof}

\begin{lemma}\label{lem:Gset}
	Let $d=2$. Let $\gamma$ and $\widetilde\gamma$ be defined as in Lemma~\ref{lem:logS} with the same condition. For any $y_1\in\partial\Bc_{n/6}$ and $y_2\in \partial \Dc_{n/6}$ with $|y_1-y_2|\le e^{n/90}$, 
	\[
	G_{\Bc_n\setminus \gamma}(y_1) \simeq \widetilde G_{\Dc_n\setminus \widetilde\gamma}(y_2).
	\]
\end{lemma}
\begin{proof}
	Using the difference estimate for Green's function (see Lemma 6.3.3 in \cite{RWintro}), 
	\begin{equation}\label{eq:difference}
	| \widetilde G_{\Dc_n\setminus \widetilde\gamma}(y_1)-\widetilde G_{\Dc_n\setminus \widetilde\gamma}(y_2) | \lesssim e^{n/90} e^{-n/6}.
	\end{equation}
	Moreover (we refer readers to the proof of Theorem 1.2 of \cite{MR2191635} for details), by Lemma~\ref{lem:logS},
	\[
	\left|G_{\Bc_n\setminus \gamma}(y_1) - \widetilde G_{\Dc_n\setminus \widetilde\gamma}(y_1)\right|
	= \left|\frac{2}{\pi}\Eb^{y_1}[ \log|W_{\widetilde\sigma}| ] - 
	 \frac{2}{\pi} \Eb^y[ \log|S_{\sigma}| ] + O(|y_1|^{-2})\right|=O(e^{-n/10}).
	\]
	Therefore, 
	\[
	\left|G_{\Bc_n\setminus \gamma}(y_1) - \widetilde G_{\Dc_n\setminus \widetilde\gamma}(y_2)\right| = O(e^{-n/10}).
	\]
	This combined with the fact that $G_{\Bc_n\setminus \gamma}(y_1)\ge G_{n/3}(y_1)\gtrsim n$ concludes the proof.
\end{proof}

Finally we are ready to present the proof of Proposition~\ref{prop:bridge-nice}.
\begin{proof}[Proof of Proposition~\ref{prop:bridge-nice}]
	Item (1) can be deduced in the same way as \eqref{eq:nice34} by following routine modifications of Section 3.2 in \cite{2sidedwalk}. Thus, we omit the proof.
	
	We now prove item (2). We will use the notation given in Proposition~\ref{prop:sharp-E}. Let $y_1$ and $\wt y_1$ be the endpoints of $\xi_1$ and $\widetilde\xi_1$, respectively. Using the strong Markov property, we obtain the following equations: 
	\[
	\nu^{\Bc_n}_{x_1,0}\otimes\nu^{\Bc_n}_{x_2,\partial\Dc_n}\{ A_{0,\partial\Bc_n}(\overline\zeta)  \}
	=\Eb \Big[1_{(\xi_1,\xi_2) \in \Ec_n(\overline\zeta)} \Eb[ G_{\Bc_n\setminus (\zeta_2\cup\xi_2)}(y_1) \mid \xi_1,\xi_2 ] \Big],
	\]
	and 
	\[
	\mu^{\Dc_n}_{x_1,0}\otimes\mu^{\Dc_n}_{x_2,\partial\Dc_n}\{ \wt A^{\Delta}_{0,\partial\Dc_n}(\overline\zeta) \}
	=\Eb \Big[1_{(\widetilde\xi_1,\widetilde\xi_2) \in \widetilde\Ec_n(\overline\zeta)} \Eb[ \widetilde G_{\Dc_n\setminus (\widetilde\zeta_{2}\cup\widetilde\xi_2)}(\wt y_1) \mid \widetilde\xi_1,\widetilde\xi_2 ] \Big],
	\]
	where we recall that $\widetilde\zeta_2=\Theta(\zeta_2,\Kc n)$.
	From the estimates derived in the proof of Proposition~\ref{prop:sharp-E}, except for an event with negligible probability (i.e., have an upper bound same as \eqref{eq:Fi}), we have $\dist(\xi_2,0)\ge e^{n/3}$ (from $F_3$), $d_H(\xi_2,\widetilde\xi_2)\le e^{n/90}$ (from $F_0$) and $|y_1-\wt y_1|\le e^{n/90}$ (from $F_0$), and it remains to show under such conditions we have
	\begin{equation}\label{eq:GB}
	G_{\Bc_n\setminus (\zeta_2\cup\xi_2)}(y_1)\simeq 
	\widetilde G_{\Dc_n\setminus (\widetilde\zeta_{2}\cup\widetilde\xi_2)}(\wt y_1).
	\end{equation}
	This follows from Lemma~\ref{lem:Gset} when $d=2$ by setting $\gamma=\zeta_2\cup\xi_2$ and $\widetilde\gamma=\widetilde\zeta_{2}\cup\widetilde\xi_2$.
	
	We next consider $d=3$. This is much simpler since 
	\[
	G_{n/3}(y_1)\le G_{\Bc_n\setminus (\zeta_2\cup\xi_2)}(y_1) \le G_{n}(y_1)
	\]
	and
	\[
	G_{n/3}(y_1)\simeq c_g e^{-n/6}, \quad G_{n}(y_1)\simeq c_g e^{-n/6},
	\]
	where $c_g$ is some universal constant. And the same estimates also hold for $\widetilde G$. Therefore, 
	\[
	G_{\Bc_n\setminus (\zeta_2\cup\xi_2)}(y_1)\simeq 
	\widetilde G_{\Dc_n\setminus (\widetilde\zeta_{2}\cup\widetilde\xi_2)}(y_1)\simeq c_g e^{-n/6}.
	\]
	Using the difference estimate again (see \eqref{eq:difference}), we obtain that 
	\[
	\widetilde G_{\Dc_n\setminus (\widetilde\zeta_{2}\cup\widetilde\xi_2)}(y_1)-\widetilde G_{\Dc_n\setminus (\widetilde\zeta_{2}\cup\widetilde\xi_2)}(\wt y_1)\lesssim e^{n/90} e^{-n/3}.
	\]
	This implies \eqref{eq:GB} when $d=3$, and we complete the proof.
\end{proof}

\subsection{Two-point estimate}
The goal of this section is to deal with the proof of Theorem~\ref{t:two_pt}, the sharp two-point estimate. Since the proof is just the combination of two copies {of the coupling used in the one-point estimate}, we will not dive into the details and only sketch the ideas and ingredients that are used in the proof.

We begin with the counterpart of Proposition~\ref{prop:bridge-nice}. 
For $\overline\zeta=(\zeta_1,\zeta_2)\in\Xc_{n/6}(z_n)$ with endpoints $(x_1,x_2)$ and $\overline\zeta'=(\zeta'_1,\zeta'_2)\in\Xc_{n/6}(w_n)$ with endpoints $(x'_1,x'_2)$,  denote $\widehat\zeta=(\overline\zeta,\overline\zeta')$ and define the set of NI discrete triples:
\[
A_{z,w}(\widehat\zeta):=\left\{\begin{array}{l} (\lambda_1,\lambda_2,\lambda_*)\in \Gamma^{\Bc_n}_{x_1,0}\times\Gamma^{\Bc_n}_{x'_1,\partial\Bc_n} \times \Gamma^{\Bc_n}_{x_2,x'_2}: \\ \lambda_1\cap(\zeta_2\cup\lambda_*\cup\lambda_2\cup\Bc_{n/6}(w_n))=\lambda_2\cap(\zeta'_2\cup\Bc_{n/6}(z_n))=\lambda_*\cap(\zeta_1\cup\zeta'_1)=\emptyset
\end{array}	\right\},
\]
and its continuous analogue:
\[
\wt A^{\Delta}_{z,w}(\widehat\zeta):=\left\{\begin{array}{l} (\gamma_1,\gamma_2,\gamma_*)\in \wt\Gamma^{\Dc_n}_{x_1,0}\times\wt\Gamma^{\Dc_n}_{x'_1,\partial\Dc_n} \times \wt\Gamma^{\Dc_n}_{x_2,x'_2}: \\ \gamma_1  \cap (\wt\zeta_{2}\cup\gamma_*\cup\gamma_2\cup\Dc_{n/6}(w_n)) =
	\gamma_2  \cap (\wt\zeta'_{2}\cup\gamma_*\cup\Dc_{n/6}(z_n)) =
	\gamma_*\cap ( \wt\zeta_{1}\cup \wt\zeta'_{1} )=\emptyset
\end{array}	\right\},
\]
where we use the notation $\wt\zeta'_i:=\Theta(\zeta'_i,\Kc n)$ as before.
See {Figure} \ref{img:two_pt} for an illustration.
As one can see, $A_{z,w}(\widehat\zeta)$ and $\wt A^{\Delta}_{z,w}(\widehat\zeta)$ are just the two-point counterparts of \eqref{eq:Af'} and \eqref{eq:Af}, respectively.

\begin{figure}[h!]
	\centering
	\includegraphics{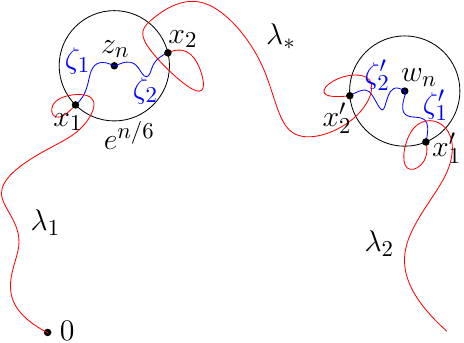}
	\caption{Illustration of the event $A_{z,w}(\widehat\zeta)$.}
	\label{img:two_pt}
\end{figure}

Recall the definition of $\text{NICE}_{r}(z)$ and $\Nc_r(z)$ just before  Proposition~\ref{prop:bridge-nice} and write
$$
\text{NICE}_{r}(z,w)=\text{NICE}_{r}(z)\times \text{NICE}_{r}(w).
$$
\begin{proposition}\label{prop:bridge-nice2}
	Let $V\in\Vc$ and $z,w\in V$ with $|z-w|\ge e^{-n/6}$.
	\begin{enumerate}
		\item[(1)] It holds that 
		\begin{equation}\label{eq:two-1}
		\Pb\{\Ac_n(z)\cap\Ac_n(w)\}\simeq\Pb\{\Ac_n(z)\,\cap\,\Ac_n(w)\cap\, \mathcal{N}_{n/6}(z)\cap\, \mathcal{N}_{n/6}(w)\}.
		\end{equation}
		\item[(2)] For all $\widehat\zeta=(\overline\zeta,\overline\zeta')\in {\rm NICE}_{n/6}(z,w)$, we have 
		\begin{equation}\label{eq:two-2}
		\nu^{\Bc_n}_{x_1,0}\otimes\nu^{\Bc_n}_{x'_1,\partial\Bc_n} \otimes \nu^{\Bc_n}_{x_2,x'_2}\{ A_{z,w}(\widehat\zeta) \} \simeq \mu^{\Dc_n}_{x_1,0}\otimes\mu^{\Dc_n}_{x'_1,\partial\Dc_n} \otimes \mu^{\Dc_n}_{x_2,x'_2}\{ \wt A^{\Delta}_{z,w}(\widehat\zeta) \}.
	\end{equation}
	\end{enumerate}
\end{proposition}

The above proposition can be proved in the same way as what we did for Proposition~\ref{prop:bridge-nice}. So we will only sketch necessary modifications. Now, we have three paths, namely $\lambda_1,\lambda_*$ from point-to-point path measures, and $\lambda_2$ a SRW started from the interior to the boundary. Therefore, we need to use the previous strategy to decompose $\lambda_1$ and $\lambda_*$ according to the first visit of some mesoscopic balls around the endpoints, i.e., $\Bc_{n/6}$ and $\Bc_{n/10}(x'_2)$, and obtain $\lambda_1=\xi_1\oplus\omega_1$, $\lambda_*=\xi_*\oplus\omega_*$. 
Then, we can use the argument developed in the proof of Proposition~\ref{prop:bridge-nice} to show that the probability that the truncated triple $(\xi_1,\xi_*,\lambda_2)$ satisfies a certain non-intersecting event, i.e., the counterpart of $\Ec_n(\overline\zeta)$, is close to that of the continuous truncated triple, in the $\simeq$ sense. Moreover, except for a negligible event, we can make   $\dist(\xi_1\cup\lambda_2,\Bc_{n/10}(x'_2))\ge e^{n/8}$, and $\dist(\xi_*\cup\lambda_2,\Bc_{n/6})\ge e^{n/3}$. Then, as before, one can show that the total masses of $\omega_1$ and $\omega_*$ are close to the continuous counterparts in the $\simeq$ sense by Lemma~\ref{lem:Gset}. This leads to the conclusion. We omit the details of the proof of Proposition~\ref{prop:bridge-nice2}, and complete the proof of Theorem \ref{t:two_pt} in the rest of this section.

\begin{proof}[Proof of Theorem \ref{t:two_pt}]
	The proof is similar to that of Theorem \ref{t:one_pt}. Plugging $s=n/2$ in Proposition~\ref{prop:gbar-gs2}, we obtain that for all $\widehat\zeta=(\overline\zeta,\overline\zeta')\in \text{NICE}_{n/6}(z,w)$,
	\begin{equation}\label{eq:sim2}
		e^{\xi n} e^{2(d-2)n} \mu^{\Dc_n}_{x_1,0}\otimes\mu^{\Dc_n}_{x'_1,\partial\Dc_n} \otimes \mu^{\Dc_n}_{x_2,x'_2}\{ \wt A^{\Delta}_{z,w}(\widehat\zeta) \}
		\simeq c_2^2\, G_{\Dc}^{\cut}(z,w)\, \Pb\{ \wt A^\Delta_{n/2}(\ol\zeta) \}\, \Pb\{ \wt A^\Delta_{n/2}(\ol\zeta') \},
	\end{equation}
	where $\wt A^\Delta_{n/2}(\ol\zeta)$ is defined in \eqref{eq:Af}.
	Analogous to the one-point case \eqref{eq:one}, we have the decomposition:
	\[
	\Pb\{\Ac_n(z)\cap \Ac_n(w)\}=\sum_{\widehat\zeta{\in \Xc_{n/6}(z_n)\times \Xc_{n/6}(w_n)}} \nu^{\Bc_n}_{x_1,0}\otimes\nu^{\Bc_n}_{x'_1,\partial\Bc_n} \otimes \nu^{\Bc_n}_{x_2,x'_2}\{ A_{z,w}(\widehat\zeta) \}\, p(\ol\zeta)\, p(\ol\zeta'),
	\]
	where $p(\ol\zeta)$ is defined in \eqref{eq:p-zeta}.
	By Proposition~\ref{prop:bridge-nice2}, we obtain
	\begin{align*}
		\ &\sum_{\widehat\zeta\in \text{NICE}_{n/6}(z,w)}  \mu^{\Dc_n}_{x_1,0}\otimes\mu^{\Dc_n}_{x'_1,\partial\Dc_n} \otimes \mu^{\Dc_n}_{x_2,x'_2}\{ \wt A^{\Delta}_{z,w}(\widehat\zeta) \} \, p(\ol\zeta)\, p(\ol\zeta')\\
		\overset{\eqref{eq:two-2}}{\simeq} \ & \sum_{\widehat\zeta\in \text{NICE}_{n/6}(z,w)}  \nu^{\Bc_n}_{x_1,0}\otimes\nu^{\Bc_n}_{x'_1,\partial\Bc_n} \otimes \nu^{\Bc_n}_{x_2,x'_2}\{ A_{z,w}(\widehat\zeta) \} \, p(\ol\zeta)\, p(\ol\zeta')\\
		\overset{\eqref{eq:two-1}}{\simeq} \ &  \Pb\{\Ac_n(z)\cap \Ac_n(w)\}.
	\end{align*}
	Moreover, by \eqref{eq:sum1},
	\begin{align*}
		\sum_{\widehat\zeta\in \text{NICE}_{n/6}(z,w)}  \left(\Pb\{ \wt A^\Delta_{n/2}(\ol\zeta) \} \, p(\ol\zeta) \right) \cdot \left(\Pb\{ \wt A^\Delta_{n/2}(\ol\zeta') \} \, p(\ol\zeta') \right)
		\simeq\ q^2 e^{-\xi n}.
	\end{align*}
    Combined, we get that 
	\[
	e^{\xi n} e^{2(d-2)n} \Pb\{\Ac_n(z)\cap \Ac_n(w)\}
	\simeq c_2^2\, G_{\Dc}^{\cut}(z,w)\,  q^2 \, e^{-\xi n}.
	\]
	The theorem follows immediately.
\end{proof}

\section{Inward coupling}\label{sec:coupling}
In this section, we prepare a coupling result {on two pairs of NIRW's with different ``initial configurations''} (see Theorem \ref{coupling} below) that is needed later. Since the proof of Theorem \ref{coupling} is quite similar to that of \cite[Theorem 3.7]{LERW3exp}, {we will only briefly sketch the proof at the end of this section. Some of the lemmas are also just stated without proof.} We refer the reader to \cite[Section 3.1]{LERW3exp} for details.

Suppose $0< l\le m\le n-1\le  K-2$. Set $V_0=\Bc_K^c$. Suppose that $V_1$ and $V_2$ are two paths lying in $\Bc_n^c$ starting from $V_0$ and ending at $\partial \Bc_n$ such that $V_1\cap V_2=\emptyset$ and for $j=1,2$, $V_j\cap \partial \Bc_n$ is a single point which is denoted by $x_j$. We refer to such quintuples $(V_0, V_1, V_2, x_1, x_2)$ colloquially as an initial configuration. (We often use  $(V_0, U_1, U_2, x'_{1}, x'_{2})$ to stand for another initial configuration.)

For $j=1,2$, let $S^j$ be SRW started from $x_j$, $\Lambda_l^j=V_j \cup S^j[0,\tau^j_l]$ with $\tau^j_l:=H_{\partial\Bc_l}(S^j)$ (cf. \eqref{eq:hitting-S}). Define the events
\begin{align*}
	&A_{l,n,K}^{1} = A_{l,n,K}^{1} (V_0, V_1, V_2, x_1, x_2) = \left\{(V_0\cup\Lambda^1_l)\cap \Lambda^2_l =\Lambda^1_l\cap (V_0\cup\Lambda^2_l) = \emptyset \right\}; \\ 
	&A_{l,n,K}^{2} =  \left\{ \tau^1_l<\tau^1_K, \tau^2_l<\tau^2_K \right\};  \  \ A_{l,n, K} =  A_{l,n,K} (V_0, V_1, V_2, x_1, x_2) = A_{l,n,K}^{1}  \cap A_{l,n,K}^{2} .
\end{align*}
See Figure \ref{fig:al} for an illustration the event $A_{l,n, K}$.

\begin{figure}[h]
	\begin{center}
		\includegraphics[scale=0.5]{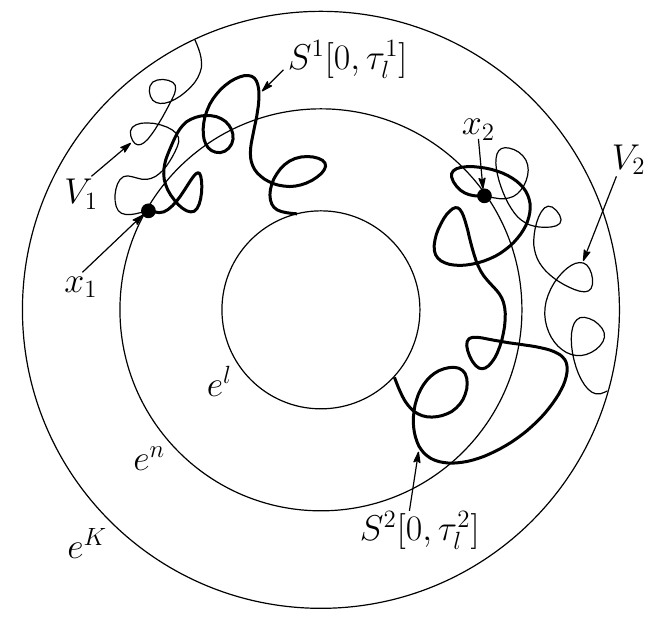}
		\caption{Illustration for the event $A_{l,n, K}$. Thick curves stand for SRW while the initial configuration is depicted by thin curves.}\label{fig:al}
	\end{center}
\end{figure}

Let  $\mu_{l,n,K} = \mu_{l, n, K} (V_0, V_1, V_2, x_1, x_2)$ be the probability measure induced by $(S^1[0,\tau_l^1],S^2[0,\tau_l^2])$ conditioned on the event $A_{l, n, K} (V_0, V_1, V_2, x_1, x_2)$. 
For two pairs of paths $\overline\lambda=(\lambda^1,\lambda^2)$ and $\overline\lambda^*=(\lambda^{1,*},\lambda^{2,*})$ in $\Gamma_{\Bc_b^c,\partial\Bc_a}\times\Gamma_{\Bc_b^c,\partial\Bc_a}$, we write $\overline\lambda=_{b}\overline\lambda^*$ if $\overline\lambda$ and $\overline\lambda^*$ merge before hitting $\partial \Bc_b$ for the first time, namely, $\lambda^{j} (H_b(\lambda^{j}) + t) = \lambda^{j,*} (H_b(\lambda^{j,*}) + t) $ for all $t \ge 0$ and $j=1,2$.

\begin{theorem}[Inward coupling of non-intersecting walks]\label{coupling}
	There exist $0 < c,\beta < \infty$ such that for all $0< l\le m\le n-1\le  K-2$ and for any pair of initial  configurations $(V_0, V_1, V_2, x_1, x_2)$, $(V_0, U_1, U_2, x'_{1}, x'_{2})$, we can define $\overline\lambda = (\lambda^{1}, \lambda^{2})$ and  $\overline\lambda^{\ast} = \big( \lambda^{1,\ast}, \lambda^{2,\ast})$ on the same probability space $(\Omega,\Fc,P)$ such that $\overline\lambda$ has the distribution $\mu_{l,n,K} (V_0, V_1, V_2, x_1, x_2)$ and $\overline\lambda^{\ast}$ has the distribution $\mu_{l,n,K} (V_0, U_1, U_2, x'_{1}, x'_{2})$, and that 
	\begin{equation*}
		P\{ \overline\lambda=_{m}\overline\lambda^{\ast} \} \ge 1-ce^{-\beta (n-m)}.
	\end{equation*}
\end{theorem}

In order to prove this theorem, we need several lemmas (Lemmas \ref{sep}-\ref{approx} below). We first begin with the reverse separation lemma as follows. Suppose that $l=n-1$ and condition that the event $A_{n-1, n, K}$ occurs. Similar to \eqref{eq:RW-quality}, we define the quality {of separation at scale $e^{n-1}$} by
$$
\Delta_{K, n} = e^{-n+1} \min_{i=1,2} \text{dist} \Big( S^{i} \big( \tau^{i}_{n-1} \big), \Lambda^{3-i}_{n-1} \Big) 
$$
as the rescaled distance between $S^{j} \big(\tau^{j}_{n-1} \big)$ and $\Lambda^{3-j}_{n-1}$. The reverse separation lemma is then stated as follows.

\begin{lemma}[Reverse Separation Lemma for NIRW's]\label{sep}
	There exists a constant $c > 0$ such that for all $1 \le n \le K-1$ and every initial configuration $(V_0, V_1, V_2, x_1, x_2)$
	$$
	\Pb \big\{ \Delta_{K,n} \ge 1/10 \ \big| \ A_{n-1, n, K}  \big\} \ge c.
	$$
\end{lemma}
The proof of this version of separation lemma is similar to that of \cite[Proposition 3.1]{LERW3exp}, so we will omit it here. Using the reverse separation lemma
and the same idea to derive \cite[Corollary 3.2]{LERW3exp}, one can also describe how the initial configuration affects the non-intersecting probability $\Pb \{A_{l,n,K}\}$, as described in the following lemma (for which the proof is omitted).

\begin{lemma}\label{decompose}
	There exists a constant $c > 0$ such that for all $0< l\le m\le n-1\le  K-2$ and every initial configuration $(V_0, V_1, V_2, x_1, x_2)$
	$$
	c \, e^{-\xi (n-l)}    \le \frac{\Pb \big\{ A_{l, n, K} \ \big| \ A^{2}_{l, n, K} \big\}}{ \Pb \big\{ A_{n-1, n, K} \ \big| \ A^{2}_{n-1, n, K} \big\} } \le c^{-1} e^{-\xi (n-l)}. 
	$$
\end{lemma}
From this lemma, we see that the effect of the initial configuration {can be controlled through the conditional non-intersection probability $P \big\{ A_{n-1, n, K} \ \big| \ A^{2}_{n-1, n, K} \big\}$ up to scale $e^{n-1}$}. The reason for this is that loosely speaking, once two simple random walks are well-separated {at scale $e^{n-1}$} in the sense of Lemma  \ref{sep}, they can forget the initial configuration. 

We will next show that conditioned on the event $A_{l, n, K}$ both $S^{1} [0, \tau^{1}_{l}]$ and $S^{2} [0, \tau^{2}_{l}]$ are contained in $\Bc_{n+q}$ with high (conditional) probability if $q > 0$ is large (recall that $x_{j} = S^{j} (0) \in \partial \Bc_{n}$). This conditional probability depends on the initial configuration, and unfortunately we cannot derive a uniform estimate for all configurations. Instead, we will choose a suitable subset of configurations for which some uniform estimate can be given. More precisely, our choice is as follows:
$$
\text{GOOD}_{K, n, q} := \Big\{ (V_0, V_1, V_2, x_1, x_2) \ \Big| \ \Pb \big\{ A_{n-1, n, K} \ \big| \ A^{2}_{n-1, n, K} \big\} \ge 10^{-1} \, e^{-q/2} \Big\}
$$
for $q > 0$.
According to Lemma \ref{decompose}, 
\begin{equation}\label{good}
\begin{split}
	&\text{if }  (V_0, V_1, V_2, x_1, x_2) \in \text{GOOD}_{K, n,q} \text{ then } \Pb \big\{ A_{l, n, K} \ \big| \ A^{2}_{l, n, K} \big\} \ge c e^{-q/2} \, e^{-\xi (n-l)} , \\
	&\text{if }  (V_0, V_1, V_2, x_1, x_2) \notin \text{GOOD}_{K, n,q} \text{ then } \Pb \big\{ A_{l, n, K} \ \big| \ A^{2}_{l, n, K} \big\} \le c^{-1} e^{-q/2}  \, e^{-\xi (n-l)}.
\end{split}
\end{equation}

Take $0< l\le m\le n-1\le  K-2$, $q > 0$ and an initial configuration $(V_0, V_1, V_2, x_1, x_2)$. We define the events $F_{1}$ and $F_{2}$ (we omit the dependence on various parameters in the notation below) by
\begin{align*}
	&F_{1} = \Big\{ \big( S^{1} [0, \tau^{1}_{m} ] \cup S^{2} [0, \tau^{2}_{m} ] \big) \subset \Bc_{ n+q} \Big\}, \\
	&F_{2} = \Big\{ \Big( V_{0}, V_{1} \oplus S^{1} [0, \tau^{1}_{m} ] , V_{2} \oplus S^{2} [0, \tau^{2}_{m} ], S^{1} (\tau^{1}_{m} ), S^{2} (\tau^{2}_{m} ) \Big) \in \text{GOOD}_{K, m, q} \Big\}.
\end{align*}
Note that if $K\le n+q$, then $F_1$ occurs on the event $A_{l, n, K}$.
Then we have
\begin{lemma}\label{no-travel}
	There exists a constant $C \in (0, \infty)$ such that for all $0< l\le m\le n-1\le  K-2$, $q > 0$ and every initial configuration $(V_0, V_1, V_2, x_1, x_2)  \in \mathrm{GOOD}_{K, n, q}$
	\begin{equation}\label{goodness}
		\Big| \Pb \big\{ A_{l, n, K} \cap F_{i} \big\} - \Pb \{ A_{l, n, K} \} \Big| \le C e^{-q/2}\, \Pb \{A_{l,n, K} \} 
	\end{equation}
	for each $i=1,2$.
	
\end{lemma}

\begin{proof}
	By dividing both sides of  \eqref{goodness} by $\Pb \{A^{2}_{l, n, K} \}$, we only have to deal with a random walk $S^{j}$ conditioned on the event that $\tau^{j}_{l} < \tau^{j}_{K}$. The proof of \eqref{goodness} for $i=2$ then easily follows from Lemma \ref{decompose} and \eqref{good}. So, we will only prove \eqref{goodness} for the case of $i=1$. We will consider the case $K\ge n+q$, otherwise it is trivial.
	
	We will first consider the case of $d=3$. The strong Markov property ensures that 
	\begin{align*}
		\Pb \big\{ A_{l, n, K} \cap F_{1}^{c}  \ \big| \ A^{2}_{l, n, K} \big\} \le \sum_{j=1}^{2} \Pb \Big\{ A_{l, n, K} \cap \big\{ S^{j} [0, \tau^{j}_{m} ] \not\subset \Bc_{n+q} \big\}  \ \Big| \ A^{2}_{l, n, K} \Big\} \le C e^{-q} e^{-\xi (n-l)}.
	\end{align*}
	Thus, the inequality \eqref{goodness} follows from \eqref{good} since $(V_0, V_1, V_2, x_1, x_2)  \in \text{GOOD}_{K, n, q}$.
	
	The case of $d=2$ is next. In this case, we can use the Beurling estimate (Proposition~\ref{p:Beurling}) to conclude that
	$$
	\Pb \Big\{ A_{l, n, K} \cap \big\{ S^{j} [0, \tau^{j}_{m} ] \not\subset \Bc_{n+q} \big\}  \ \Big| \ A^{2}_{l, n, K} \Big\} \le C e^{-q} e^{-\xi (n-l)}.
	$$
	This finishes the proof.
\end{proof}

Suppose that we have two initial configurations $(V_0, V_1, V_2, x_1, x_2) $ and $(V_0, U_1, U_2, x_1, x_2) $ such that $(V_{1}, V_{2} ) =_{n+q} (U_{1}, U_{2} )$ for some $q > 0$. Namely, $V_{j}$ and $U_{j}$ merge before hitting $\partial \Bc_{n+q}$ for the first time (recall that both $V_{j} $ and $U_{j}$ are paths starting from $\Bc_{K}^c$ and ending at $\partial \Bc_{n}$). Lemma \ref{no-travel} guarantees that a random walk $S^{j}$ conditioned on the event $A_{l,n,K}$ is contained in $\Bc_{n+q}$ with high (conditional) probability if the initial configuration is good (with conditional probability $1$ if $K\le n+q$). This implies that  if  $(V_{1}, V_{2} ) =_{n+q} (U_{1}, U_{2} )$ and $(V_0, V_1, V_2, x_1, x_2) $ is good, then the probability of the event $A_{l,n,K} (V_0, V_1, V_2, x_1, x_2) $ is close to that of the event $A_{l,n,K} (V_0, U_1, U_2, x_1, x_2) $. More precisely, we have the following as a direct consequence of Lemma \ref{no-travel}.

\begin{lemma}\label{approx}
	There exists a constant $C \in (0, \infty)$ such that for all $0< l\le m\le n-1\le  K-2$, $q > 0$ and all initial configurations $(V_0, V_1, V_2, x_1, x_2) $ and $(V_0, U_1, U_2, x_1, x_2) $ satisfying $(V_{1}, V_{2} ) =_{n+q} (U_{1}, U_{2} )$ and $(V_0, V_1, V_2, x_1, x_2)  \in \mathrm{GOOD}_{K, n, q}$
	\begin{equation}
		\Big| \Pb \Big\{ A_{l,n,K} (V_0, V_1, V_2, x_1, x_2) \Big\} - \Pb \Big\{ A_{l,n,K} (V_0, U_1, U_2, x_1, x_2) \Big\} \Big| \le C\, e^{-q/2}\, \Pb \Big\{ A_{l,n,K} (V_0, V_1, V_2, x_1, x_2) \Big\}. 
	\end{equation}
\end{lemma}

We are now ready to prove Theorem \ref{coupling}.

\begin{proof}[Proof of Theorem \ref{coupling}]
	Throughout the proof, we will assume $K\ge n+q$, otherwise it is simpler. Since we can take the constant $c$ in the theorem to be large enough, it suffices to prove the theorem when $n-m$ is large enough.
	Take two initial configurations $(V_0, V_1, V_2, x_1, x_2) $ and $(V_0, U_1, U_2, x_1, x_2) $ satisfying $(V_{1}, V_{2} ) =_{n+q} (U_{1}, U_{2} )$ and $(V_0, V_1, V_2, x_1, x_2)  \in \text{GOOD}_{K, n, q}$. Write $\overline{V} = (V_{1}, V_{2})$ and $\overline{U} = (U_{1}, U_{2} )$. Using Lemma \ref{approx}, we can define $\overline{\lambda}$ and $\overline{\lambda}^{\ast}$ on the same probability space {(we denote it by $P$ below)} such that $\overline{\lambda}$ has the distribution $\mu_{l,n,K} (V_0, V_1, V_2, x_1, x_2)$ and $\overline{\lambda}^{\ast}$ has the distribution $\mu_{l,n,K} (V_0, U_1, U_2, x_1, x_2)$, and that 
	\begin{align}\label{coup1}
		&P \Big\{ \overline{V} \oplus \overline{\lambda}_{m} =_{n+q} \overline{U} \oplus \overline{\lambda}^{\ast}_{m} \Big\} \ge 1 - C_{1} e^{-q/2},  \notag \\
		&P \Big\{ \big( V_{0}, V_{1} \oplus \lambda^{1}_{m}, V_{2} \oplus \lambda^{2}_{m}, y_{1}, y_{2} \big) \in \text{GOOD}_{K, m, q} \Big\} \ge 1- C_{1} e^{-q/2},
	\end{align}
	where 
	\begin{itemize}
		\item $\lambda^{j}_{m} = \lambda^{j} [0, \tau_{m} (\lambda^{j}) ]$ stands for the path $\lambda^{j}$ up to the first time that it hits $\Bc_m$,
		
		\item $y_{j} = \lambda^{j} (\tau_{m} (\lambda^{j}) )$ is the endpoint of $\lambda^{j}_{m}$, 
		
		\item $\overline{\lambda}_{m} = (\lambda^{1}_{m}, \lambda^{2}_{m} )$, and $ \overline{\lambda}^{\ast}_{m} = \big( \lambda^{1,\ast}_{m}, \lambda^{2,\ast}_{m} \big)$ is defined similarly.
	\end{itemize}
	See Figure \ref{fig:couple1} for this coupling.
	We will use the coupling \eqref{coup1} later when $q $ is large.
	
	\begin{figure}[!h]
		\begin{center}
			\includegraphics[scale=0.5]{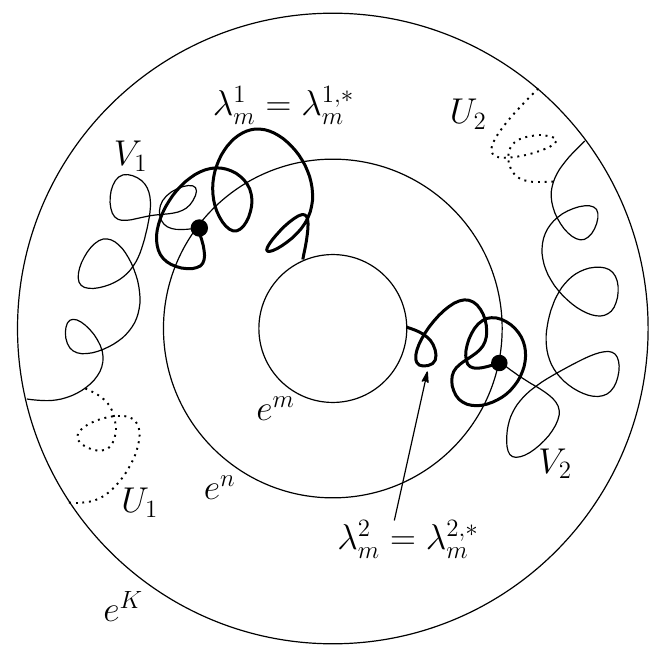}
			\caption{Illustration for the coupling \eqref{coup1}. A large part of $V_{j}$ (depicted by a thin curve) agrees with that of $U_{j}$ (depicted by a dotted curve). In this coupling, $\overline{\lambda}_{m} = \overline{\lambda}^{\ast}_{m}$ with high probability.}
			\label{fig:couple1}
		\end{center}
	\end{figure}

	On the other hand,  if $q$ is too small or even $V_{j}$ and $U_{j}$ do not share the same endpoint, using the reverse separation lemma (Lemma \ref{sep}), it is easy to derive the following coupling. Fix an integer $L > 2$ such that $C_{1} e^{-\frac{L-2}{2}} < 1/2$ where the constant $C_{1}$ is defined as in \eqref{coup1}. Take two arbitrary initial configurations $(V_0, V_1, V_2, x_1, x_2)$ and $(V_0, U_1, U_2, x'_1, x'_2)$. Then we can define $\overline{\lambda}$ and $\overline{\lambda}^{\ast}$ on the same probability space {(for ease of notation we still denote it by $P$)} such that $\overline{\lambda}$ has the distribution $\mu_{l,n,K} (V_0, V_1, V_2, x_1, x_2)$ and $\overline{\lambda}^{\ast}$ has the distribution $\mu_{l,n,K} (V_0, U_1, U_2, x'_1, x'_2)$, and that 
	\begin{align}\label{coup2}
		&P \Big\{ \overline{V} \oplus \overline{\lambda}_{n-L} =_{n-2}\  \overline{U} \oplus \overline{\lambda}^{\ast}_{n-L} \Big\} \ge c'> 0,  \notag \\
		&P \Big\{ \big( V_{0}, V_{1} \oplus \lambda^{1}_{n-L}, V_{2} \oplus \lambda^{2}_{n-L}, z_{1}, z_{2} \big) \in \text{GOOD}_{K, n-L, L-2} \Big\} \ge c',
	\end{align}
	where $z_{j} = \lambda^{j} (\tau_{n-L} (\lambda^{j}) )$ is the endpoint of $\lambda^{j}_{n-L}$. See Figure \ref{fig:couple2} for this coupling.

	\begin{figure}[!h]
		\begin{center}
			\includegraphics[scale=0.5]{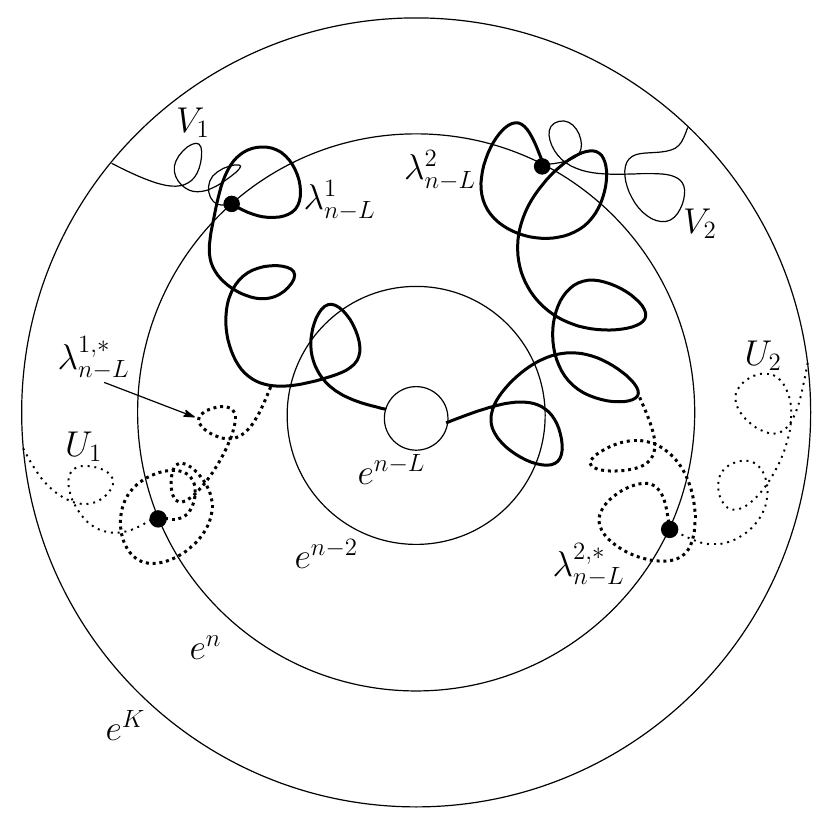}
			\caption{Illustration for the coupling \eqref{coup2}. There is a possibility that $V_{j} \cap U_{j} = \emptyset$ in this coupling. However, with positive probability, $\lambda^{j}_{n-L}$ (depicted by a thick curve) and $\lambda^{j,\ast}_{n-L}$ (drawn by a thick dotted curve)  merge before hitting $\partial {\cal B}_{n-2}$.}
			\label{fig:couple2}
		\end{center}
	\end{figure}

	Returning to the proof of the theorem, let us consider the following procedure for which we will use the coupling {leading to} \eqref{coup1} when $q \ge L-2$ and otherwise we will use that {leading to} \eqref{coup2}. Let  $J \ge 1$ be the largest integer satisfying $n-JL \ge m+1$. The first goal is to obtain a coupling of $\overline{\lambda}_{n-JL}$ and $\overline{\lambda}^{\ast}_{n-JL}$ where $\overline{\lambda}$ has the distribution $\mu_{l,n,K} (V_0, V_1, V_2, x_1, x_2)$ and $\overline{\lambda}^{\ast}$ has the distribution $\mu_{l,n,K} (V_0, U_1, U_2, x'_1, x'_2)$.
	To get it, we first consider a coupling of $\overline{\lambda}_{n-L}$ and $\overline{\lambda}^{\ast}_{n-L}$, then define a coupling of $\overline{\lambda}_{n-2L}$ and $\overline{\lambda}^{\ast}_{n-2L}$, and continue this procedure until the $J$-th step. Take $j \le J$. If there exists an integer $r \ge 1$ such that in the coupling above at the $j$-th step, 
	\begin{align}\label{sigma}
		\overline{\lambda}_{n-jL} =_{n-jL+r} \overline{\lambda}^{\ast}_{n-jL} \ \ \ \text{ and }  \ \ \  \overline{\lambda}_{n-jL}  \in \text{GOOD}_{K, n-jL, r}
	\end{align}
	then we let $\xi_{j}$ be the largest integer $r \ge 1$ satisfying \eqref{sigma}. Otherwise we set $\xi_{j} = 0$.
	
	Suppose that $\overline{\lambda}_{n-jL} $ and $ \overline{\lambda}^{\ast}_{n-jL}$ are given. To proceed to the next step, we define a coupling of $\overline{\lambda}_{n-(j+1)L} $ and $ \overline{\lambda}^{\ast}_{n-(j+1)L}$ either via {the coupling leading to} \eqref{coup1} or that leading to \eqref{coup2} according to whether $\xi_{j}  \ge L-2$. Note that by \eqref{coup1}, if $r \ge L-2$,
	$$
	P \big( \xi_{j+1} = r+L \ \big| \ \xi_{j} = r \big) \ge 1 - C_{1} e^{-r/2},
	$$
	while if $r < L-2$ the inequality \eqref{coup2} guarantees that 
	$$
	P \big( \xi_{j+1} \ge L-2 \ \big| \ \xi_{j} = r \big) \ge c'.
	$$
	
	Comparing $\{ \xi_{j} \}_{j \ge 1}$ with a one-dimensional Markov chain (see the proof of Theorem 4.1 in \cite{BMinvar} for this), it follows that there exist $ c, \beta \in (0, \infty)$ such that 
	$$
	P \big( \xi_{J} \le JL/2 \big) \le c\, e^{-\beta(n-m)}.
	$$
	This implies that we can couple $\overline{\lambda}_{n-JL}$ and $\overline{\lambda}^{\ast}_{n-JL}$ such that with probability at least $1 - c e^{-\beta(n-m)}$,
	\begin{align*}
		\overline{\lambda}_{n-JL} =_{n-JL/2} \overline{\lambda}^{\ast}_{n-JL} \ \ \ \text{ and }  \ \ \  \overline{\lambda}_{n-JL}  \in \text{GOOD}_{K, n-JL, JL/2}.
	\end{align*}
	Theorem \ref{coupling} then easily follows from this and we finish the proof.
\end{proof}

\section{Transfer between cut-point and cut-ball events}\label{sec:cpcb}
In this section, we obtain the transformation between the cut-point event and the discrete cut-ball event in Section~\ref{subsec:dc}, and as well between the discrete cut-ball event and the continuous cut-ball event in Section~\ref{subsec:dcc}. The results in this section are the key ingredients to prove our main theorems in the last section.

\subsection{Comparison in the discrete}\label{subsec:dc}
In this subsection, we define the discrete version of cut-ball events and compare the probability of {a} cut-point event with that of {a} cut-ball event in a very precise way (in the ``$\simeq$'' sense) with the help of the inward coupling obtained in the previous section.

Write $\lambda$ for $S[0,\tau_n]$, the simple random walk from $0$ to the exit of $\Bc_n$. For any set $B$ in $\Bc_n$ with $0\notin \ol B$, if $\lambda\cap B\neq\emptyset$, by \eqref{eq:decom1}, we can decompose $\lambda$ by the first-entry and the last-exit of $\overline B$ as follows:
\begin{equation}\label{eq:decomp}
	\lambda=\lambda_1\oplus\omega\oplus[\lambda_2]^R,
\end{equation}
where $\lambda_1=\lambda[0,H_{\partial B}(\lambda)]$, $\lambda_2=\lambda^R[0,H_{\partial B}(\lambda^R)]$, and $\omega=\lambda[H_{\partial B}(\lambda),H_{\partial B}(\lambda^R)]$. Note that we reversed the last part in the decomposition \eqref{eq:decom1} to simplify the notation later.
This is analogous to the continuous decomposition \eqref{eq:decomp2}. We recall that $z_n=\floo{e^{n}z}$ from \eqref{eq:dz-zn}.
\begin{definition}[Cut ball for SRW]\label{def:cut-Db}
    Let $z\in\Dc\setminus\{ 0 \}$.
	we say that the ball $\Bc_{3n/4}(z_n)$ is a cut ball for $\lambda$ if $\lambda\cap \Bc_{3n/4}(z_n)\neq\emptyset$, $\lambda_1\cap\lambda_2=\emptyset$ and $\omega\in \Bc_{5n/6}(z_n)$ where $\lambda=\lambda_1\oplus\omega\oplus[\lambda_2]^R$ is decomposed as in \eqref{eq:decomp} with $B=\Bc_{3n/4}(z_n)$.
	Denote by $K_{3n/4}(z)$ the event that $\Bc_{3n/4}(z_n)$ is a cut ball for $S[0,\tau_n]$.
\end{definition}

\begin{remark}\label{rem:blow-up}
	The choices for the ball sizes are not entirely arbitrary.
	The inner ball needs to be large compared to the scale of the coupling distance between Brownian and random walk paths with the Skorokhod embedding \eqref{eq:se}, which involves a spacial error of order $e^{(1/2+\eps)n}$ with $\eps > 0$. Thus, we choose the radius of the inner ball to be $e^{3n/4}$.
	The radius of the outer ball is picked to match the Definition~\ref{def:cut-bm} of a cut ball for a Brownian motion, modulo scaling by $e^n$. The inner radius is $e^{3n/4}$, which after downscaling becomes $e^{-s}$ with $s=n/4$, and we pick the outer ball of size $e^{-2s/3}$, which equals $e^{5n/6}$ after upscaling.
\end{remark}

In this section, we will prove the following propositions, which let us commute between the probabilities of cut points and that of cut balls (in $\simeq$ sense). As they can be proved in a similar way, we will only prove Proposition~\ref{prop:f2} and briefly discuss necessary adaptations for Proposition~\ref{prop:AKAA} at the end of this subsection.
\begin{proposition}\label{prop:f2}
	There is a function $f(n)$ of $n$ such that for all $z\in \Dc$ with $\dist(0,z,\partial\Dc)\ge e^{-n/6}$, 
	\begin{equation}
		\Pb\{ K_{3n/4}(z) \} \simeq f(n) \Pb\{ \Ac_n(z) \}. 
	\end{equation}
\end{proposition}

\begin{proposition}\label{prop:AKAA}
	For all $V\in\Vc$, $z,w\in V$ with $|z-w|\ge e^{-n/6}$,
	\begin{equation}\label{eq:AKAA}
		\Pb\{ K_{3n/4}(z)\cap \Ac_n(w) \} \simeq_V
		f(n) \Pb\{ \Ac_n(z) \cap \Ac_n(w) \}. 
	\end{equation}
\end{proposition}

We start with Proposition~\ref{prop:f2}. The key ingredient is the inward coupling Theorem~\ref{coupling}. We first show how to adapt it to the current setting. We  consider the balls of radii $3n/4<3n/4+1<37n/48<19n/24<5n/6$ about $z_n$, respectively.
We will decompose $\lambda=S[0,\tau_n]$  by its visits to $\mathcal{S}:=\partial\Bc_{3n/4+1}(z_n)$ instead of $\partial\Bc_{3n/4}(z_n)$ (as we already encountered, it is somewhat more convenient to spare an additional scale, see Remark~\ref{rmk:additional-scale} and Lemma~\ref{lem:hp} below). More precisely (see Figure~\ref{fig:zw}), 
\begin{itemize}
	\item let $\eta_1$ be the part of $\lambda$ from $0$ to the first hitting of $\Sc$,
	\item let  $\eta_2$ be the part of $\lambda^R$ from its starting point to its first visit of $\Sc$,
	\item and let $\omega_*$ be the remaining part of $\lambda$ from $\eta_1$ to $\eta_2$. 
\end{itemize}
Then, we can write $\lambda$ as $\eta_1\oplus\omega_*\oplus \eta_2^R$. 
By \eqref{eq:decom2}, to get $\lambda$, we can first sample $\ol\eta:=(\eta_1,\eta_2)$ from the measure $\nu^{U}_{0,\Sc}\otimes \nu^{U}_{\partial\Bc_n,\Sc}$ with $U:=\Bc_n\setminus \Bc_{3n/4+1}(z_n)$, then sample $\omega_*$ from the measure $\nu^{U}_{x_1,x_2}$ with $x_i$ being the endpoint of $\eta_i$, finally concatenate them in the way of \eqref{eq:decomp}. 

Let $p^z_{0,n}$ be the measure $\nu^{U}_{0,\Sc}\otimes \nu^{U}_{\partial\Bc_n,\Sc}$ restricted to the NI event $\eta_1\cap\eta_2=\emptyset$.
Then, let $\mu_0$ be some reference measure $\mu_{3n/4+1,19n/24,5n/6}(V_0,V_1,V_2,y_1,y_2)$ defined in the beginning of Section~\ref{sec:coupling}.
Let $\mu_0^z$ be the image of $\mu_0$ under the map $x\mapsto x+z_n$.
The following lemma says that we can use Theorem~\ref{coupling} to couple the probability measures $\widehat p^z_{0,n}$ (the normalized version of $p^z_{0,n}$) and $\mu^z_0$ with exponential rate.

\begin{figure}[h!]
	\centering
	\includegraphics[width=.3\textwidth]{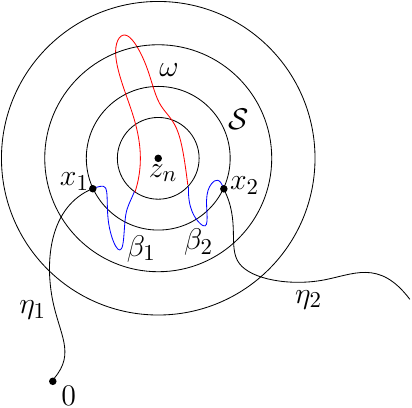}
	\caption{An illustration of path decomposition for $\lambda$. The four concentric balls about $z_n$ have radii $e^{3n/4}, e^{3n/4+1}$, $e^{37n/48}$ and $e^{5n/6}$, respectively and the second sphere is $\Sc$. $\bar\eta$ is in black with endpoints $\bar x$ and $\omega^*$ is the concatenation of $\bar\beta$ in blue and $\omega$ in red. }
	\label{fig:zw}
\end{figure}

\begin{lemma}\label{lem:coupling}
	There is a coupling $P$ of $\overline\eta$ and $\overline\eta'$ such that $\overline\eta\sim \widehat p^z_{0,n}$, $\overline\eta'\sim \mu^z_0$, and 
	\begin{equation*}
		P\{ \overline\eta=_{37n/48}\overline\eta' \} \ge 1-c e^{-un},
	\end{equation*}
	where $c,u$ are some universal positive constants.
\end{lemma}
\begin{proof}
	Let $\overline\eta=(\eta_1,\eta_2)\sim \widehat p^z_{0,n}$ and $\overline\eta'=(\eta'_1,\eta'_2)\sim \mu^z_0$. We decompose $\overline\eta=\overline \xi\oplus\overline\gamma$ with $\overline \xi=\overline\eta[0,H_{\partial\Bc_{19n/24}(z_n)}]$. By {the} strong Markov property, given $\overline\xi$ with endpoints $\overline y=(y_1,y_2)$, the distribution of $\overline\gamma$ is just a pair of independent simple random walks started from $\overline y$ and stopped upon hitting $\Sc$, conditioned on the event that they hit $\Sc$ before exiting $\Bc_n$ and $\eta_1\cap\eta_2=\emptyset$. By using the {g}ambler's ruin estimate (Lemma~\ref{l:trans_recur}) when $d=3$ or the Beurling estimate (Proposition~\ref{p:Beurling}) when $d=2$, we see that there is a $u>0$ such that for $i=1,2$,
	\[
	\Pb\{ \gamma_i\cap \Bc_{5n/6}(z_n)^c \neq\emptyset \mid \overline\xi \}=O(e^{-un}).
	\]
	Therefore, if we further condition on the event that $\gamma_1$ and $\gamma_2$ stay inside $\Bc_{5n/6}(z_n)$, then 
	\[
	\overline\gamma\sim \mu_{3n/4+1,19n/24,5n/6}(\Bc_{5n/6}(z_n)^c, \xi_1,\xi_2,y_1,y_2).
	\]
    By using the same decomposition for $\overline\eta'$ instead, we arrive at the same conclusion: conditioned on $\overline\xi'$ and $\overline\gamma'$ staying inside $\Bc_{5n/6}(z_n)$, we have 
    \[
    \overline\gamma'\sim \mu_{3n/4+1,19n/24,5n/6}(\Bc_{5n/6}(z_n)^c, \xi'_1,\xi'_2,y'_1,y'_2).
    \]
    Applying the inward coupling (Theorem~\ref{coupling}) to $\ol\gamma$ and $\ol\gamma'$, we obtain the result. 
\end{proof}

After sampling $\overline\eta$ according to $p^z_{0,n}$, we only need to consider admissible intermediate part $\omega_*$ to satisfy the cut-ball event $K_{3n/4}(z)$ or the cut-point event $\Ac_n(z)$, respectively. 
We begin with the cut-ball event.
This is very similar to the continuous (cut-ball) case, see Section~\ref{subsec:BM-op}. 
If $\overline\eta=(\eta_1,\eta_2)\in \Gamma^{U}_{0,\Sc}\times \Gamma^{U}_{\partial\Bc_n,\Sc}$ is a pair of NI paths with terminal points $x_1,x_2$ on $\Sc$, let $p^{\overline\eta}$ denote the measure $\nu^{\Bc_n}_{x_1,x_2}$ restricted to those paths $\omega_*$ such that 
\begin{itemize}
	\item $\omega_*\cap \Bc_{3n/4}(z_n)\neq \emptyset$, and it can be decomposed as $\omega_*=\beta_1\oplus\omega\oplus \beta_2^R$ according to its first and last visit to $\partial\Bc_{3n/4}(z_n)$ (see Figure~\ref{fig:zw} for an illustration),
	\item and $(\eta_1\cup \beta_1)\cap (\eta_2\cup \beta_2)=\emptyset$ and $\omega\subset \Bc_{5n/6}(z_n)$.
\end{itemize}
We then define 
\begin{equation*}
	\varphi_1(\overline\eta):=\|p^{\overline\eta}\|, \quad \varphi_2(\overline\eta):=p^{\overline\eta}[1_{\beta_1,\beta_2\subset \Bc_{37n/48}(z_n)}].
\end{equation*}
With the above definition, we can write 
\begin{equation}\label{eq:cp-dec}
	\Pb\{ K_{n/2}(z) \}=\|p^z_{0,n}\|\, \widehat p^z_{0,n} [\varphi_1 ].
\end{equation}

The next lemma lets us compare $\varphi_1$ with $\varphi_2$.
\begin{lemma}\label{lem:hp}
	The following {holds}:
	\begin{equation}\label{eq:hp-1}
		\|\varphi_1\|:=\max_{\overline\eta} \varphi_1(\overline\eta) \asymp 
		\widehat p^z_{0,n}[\varphi_2]\asymp
		n^{3-d} e^{-3(d-2)n/4},
	\end{equation}
where the maximum is over all pairs of $\overline\eta=(\eta_1,\eta_2)\in \Gamma^{U}_{0,\Sc}\times \Gamma^{U}_{\partial\Bc_n,\Sc}$ such that $\eta_1\cap\eta_2=\emptyset$.
	{Moreover, }
\begin{equation}\label{eq:hp-2}
	\widehat p^z_{0,n}[\varphi_1]\simeq
	\widehat p^z_{0,n}[\varphi_2].
\end{equation}
\end{lemma}
\begin{proof}
	We start by showing 
	\begin{equation}\label{eq:hatp}
		\|\varphi_1\|=O(n^{3-d} e^{-3(d-2)n/4}).
	\end{equation}
	If $d=2$, it follows from the standard estimate about the Green's function in $\Bc_n$. If $d=3$, since $\dist(\{x_1,x_2\},S)\ge e^{3n/4}$, we have $\nu_{x_1,x_2}[1_{\omega_*\cap S\neq\emptyset}]=O(e^{-3n/4})$. This confirms \eqref{eq:hatp} in both cases. 
	
	To prove \eqref{eq:hp-1}, it is now sufficient to show $\widehat p^z_{0,n}[\varphi_2]\gtrsim n^{3-d} e^{-3(d-2)n/4}$. By the reverse separation lemma (Lemma~\ref{sep}), with positive probability, $(\eta_1,\eta_2)$ sampled from $\widehat p^z_{0,n}$ are well-separated at $\Sc$. If $\eta_1$ and $\eta_2$ are well-separated, then with {a} universal positive probability we can attach $\beta_1$ and $\beta_2$ to $\eta_1$ and $\eta_2$ respectively such that $(\eta_1\cup \beta_1)\cap (\eta_2\cup \beta_2)=\emptyset$, $(\beta_1\cup\beta_2)\subset \Bc_{37n/48}(z_n)$, and they are also well-separated at $\partial\Bc_{3n/4}(z_n)$. For well-separated $\beta_1$ and $\beta_2$, the total mass of $\omega$ we can attach is greater than a constant multiple of $n^{3-d} e^{-3(d-2)n/4}$. This finishes the proof of \eqref{eq:hp-1}.
	
	It remains to prove \eqref{eq:hp-2}. To this end, it suffices to show
	\begin{equation*}
		\max_{\overline\eta} (\varphi_1(\overline\eta)-\varphi_2(\overline\eta) ) = O(e^{-un}) \|\varphi_1\|.
	\end{equation*}
	Note that $$\varphi_1(\overline\eta)-\varphi_2(\overline\eta)\le p^{\overline\eta}[1_{\beta_1\nsubseteq \Bc_{37n/48}(z_n)}]+p^{\overline\eta}[1_{\beta_2\nsubseteq \Bc_{37n/48}(z_n)}].$$
	It boils down to show 
	\[
	\max_{\overline\eta} p^{\overline\eta}[1_{\beta_1\nsubseteq \Bc_{37n/48}(z_n)}]= O(e^{-un}) \|\varphi_1\|.
	\]
	If $\beta_1\nsubseteq \Bc_{37n/48}(z_n)$, we can decompose $\beta_1$ into a concatenation of a long path from $\Sc$ to $\Sc$ that reaches $\Bc_{37n/48}(z_n)^c$ and another path from $\Sc$ to $\partial\Bc_{3n/4}(z_n)$. Then the extra cost $O(e^{-un})$ comes from the first part of $\beta_1$ when $d=2$ (it should avoid $\eta_2$) by the Beurling estimate (see Proposition~\ref{p:Beurling}), or the second part of $\beta_2$ when $d=3$ by the gambler's ruin estimate (Lemma~\ref{l:trans_recur}). Thus, we conclude the proof.
\end{proof}

Now, we have the following equivalent (in the $\simeq$ sense) expression for the probability of cut-ball events. Recall that $\mu_0=\mu_{3n/4+1,38n/48,5n/6}(V_0,V_1,V_2,y_1,y_2)$ is some fixed reference measure. 
\begin{proposition}\label{prop:ecb}
	It holds that
	\[
	\widehat p^z_{0,n}[\varphi_2] \simeq \mu^z_0 [\varphi_2 ], \quad 
	\Pb\{ K_{n/2}(z) \}\simeq \|p^z_{0,n}\|\, \mu^z_0 [\varphi_2 ].
	\]
\end{proposition}
\begin{proof}
	Noting that $\varphi_2(\overline\eta)$ only depends on the part of $\overline\eta$ inside $\Bc_{37n/48}(z_n)$,
	by Lemma~\ref{lem:coupling}, we have 
	\[
	| \widehat p^z_{0,n}[\varphi_2] - \mu^z_0 [\varphi_2 ] | \lesssim e^{-un} \|\varphi_1\|. 
	\]
	This combined with \eqref{eq:hp-1} gives that  
	\[
	\widehat p^z_{0,n}[\varphi_2] \simeq \mu^z_0 [\varphi_2 ].
	\]
	By \eqref{eq:cp-dec} and \eqref{eq:hp-2}, we have 
	\[
	\Pb\{ K_{n/2}(z) \}=\|p^z_{0,n}\|\, \widehat p^z_{0,n} [\varphi_1 ]
	\simeq  \|p^z_{0,n}\|\,  \widehat p^z_{0,n} [\varphi_2 ]
	\simeq  \|p^z_{0,n}\|\,  \mu^z_0 [\varphi_2 ].
	\]
	This completes the proof.
\end{proof}

Next, we will obtain a similar expression for the probability $\Pb\{ \Ac_n(z) \}$ by replacing $p^z_{0,n}$ with $\mu_0^z$. For this purpose, we consider admissible $\omega_*$ to satisfy the cut-point event instead.
If $\overline\eta=(\eta_1,\eta_2)$ is a pair of paths with terminal points $x_1,x_2$ on $\Sc$, let $q^{\overline\eta}$ denote $p_{x_1,x_2}$ restricted to those paths $\omega_*$ such that $z$ is a cut point for $\eta_1\oplus\omega_*\oplus \eta_2^R$. For such $\omega_*$, we can further decompose it as $\beta_1\oplus\omega\oplus\beta_2^R$ according to its first and last visit to $\partial\Bc_{3n/4}(z_n)$ as we did before.
We then define 
\begin{equation*}
	\psi_1(\overline\eta):=\|q^{\overline\eta}\|, \quad
	\psi_2(\overline\eta):=q^{\overline\eta}[1_{\beta_1,\beta_2\subset \Bc_{37n/48}(z_n)}].
\end{equation*}
With the above definition, we can write 
\begin{align*}
	\Pb\{ \Ac_n(z) \}=\|p^z_{0,n}\|\, \widehat p^z_{0,n} [\psi_1 ].
\end{align*}
We also want to compare $\psi_1$ with $\psi_2$.
\begin{lemma}
	The following holds:
	\begin{equation}\label{eq:q-hp}
			\|\psi_1\|:=\max_{\overline\eta} \psi_1(\overline\eta) \asymp 
		\widehat p^z_{0,n}[\psi_2]\asymp e^{-3(\xi+2(d-2))n/4},
	\end{equation}
where the maximum is over all pairs of $\overline\eta=(\eta_1,\eta_2)\in \Gamma^{U}_{0,\Sc}\times \Gamma^{U}_{\partial\Bc_n,\Sc}$ such that $\eta_1\cap\eta_2=\emptyset$. Moreover,
	\begin{equation}\label{eq:q-hp'}
		\widehat p^z_{0,n}[\psi_1]\simeq
		\widehat p^z_{0,n}[\psi_2].
	\end{equation}
\end{lemma}
\begin{proof}
	We start with the upper bound of $\|\psi_1\|$. We construct admissible $\omega_*$ in the following way:
	\begin{itemize}
		\item Let $(\xi_1,\xi_2)$ be a pair of NIRW's from $z_n$ to the exit of $\Bc_{3n/4}(z_n)$ with endpoints $(y_1,y_2)$, which has total mass $\asymp e^{-3\xi n/4}$.
		\item Sample $\beta_i$ from the measure $\nu^{\Bc_n}_{x_i,y_i}$ for $i=1,2$, and restrict $(\beta_1,\beta_2)$ to the event that $\beta_i\cap (\eta_{3-i}\cup \beta_{3-i}\cup \xi_{3-i})=\emptyset$, which has total mass $O(e^{-3(d-2)n/2})$ by Lemma~\ref{l:paths_discon_beurling}.
	\end{itemize}
The multiplication of these masses gives the desired upper bound for $\|\psi_1\|$. For the lower bound of $\widehat p^z_{0,n}[\psi_2]$, we can use the separation lemma for $\overline\eta$ and $\overline\xi$ from the first bullet above, respectively, and then we can attach them by a  pair $(\beta_1,\beta_2)$ from the second bullet with the same order of total mass. This concludes the proof of \eqref{eq:q-hp}. Moreover, the proof of \eqref{eq:q-hp'} is similar to that of \eqref{eq:hp-2}, and thus omitted.
\end{proof}

The following proposition is an analogue of Proposition~\ref{prop:ecb}, which can be proved in a similar fashion. Thus, we present it without proof.
\begin{proposition}\label{prop:phi}
	It holds that 
		\begin{align*}
			\widehat p^z_{0,n}[\psi_2] \simeq \mu^z_0 [\psi_2 ],\quad\mbox{ and } \quad 
			\Pb\{ \Ac_n(z) \}\simeq \|p^z_{0,n}\|\, \mu^z_{0} [\psi_2].
		\end{align*}
\end{proposition}

We are now ready to prove Proposition~\ref{prop:f2}.

\begin{proof}[Proof of Proposition~\ref{prop:f2}]
Note that both $\mu^z_{0} [\varphi_2]$ and $\mu^z_{0} [\psi_2]$ are translation invariant in $z$ by translation {invariance}. Let
\begin{equation}\label{eq:fn}
	f(n):=\mu^z_0[\varphi_2]/\mu^z_0[\psi_2]\asymp n^{3-d}e^{3\eta n/4}
\end{equation}
which only depends on $n$. 
Then, Proposition~\ref{prop:f2} follows from Propositions~\ref{prop:ecb} and \ref{prop:phi} directly.
\end{proof}

We now briefly discuss how to adapt the proof above for Proposition \ref{prop:AKAA}. The event $K_{3n/4}(z)\cap \Ac_n(w)$ can be decomposed into three sub-events:
\begin{itemize}
\item The cut-ball event $K_{3n/4}(z)$ happens before $\Ac_n(w)$ and the SRW $S$ no longer returns within the vicinity of $z$ after making a cut point at $w$;
\item The cut-ball event $K_{3n/4}(z)$ happens after $\Ac_n(w)$ and the SRW $S$ never enters the vicinity of $z$ before making a cut point at $w$;
\item The SRW $S$ approaches the vicinity of $z$ (at least) twice, but makes a cut point at $w$ between two visits and moreover forms a cut ball around $z$.
\end{itemize}
The event $\Ac_n(z) \cap \Ac_n(w)$ can be decomposed similarly.
For the first two sub-events, one can rerun the argument above (but sample, e.g.\ in the analysis of the first sub-event, $\eta_2$ in the proof of Proposition \ref{prop:f2} not from the original path measure but from the path measure restricted to those having a cut point at $w$ instead) to match the corresponding sub-events decomposed from $\Ac_n(z) \cap \Ac_n(w)$. For the third sub-event, by an argument similar to those in Section 4.4 of \cite{mink_cont} involving path decomposition and the probability decay due to extra backtracking (via disconnection exponent for $d=2$ and Green's function for $d=3$), we can show that it comprises a negligible proportion of $K_{3n/4}(z)\cap \Ac_n(w)$ in terms of probability. The same argument applies to the third sub-event decomposed from $\Ac_n(z) \cap \Ac_n(w)$.

\subsection{Comparison between the discrete and continuum}\label{subsec:dcc}
In this subsection, we couple the random walk $S[0,\tau_{n+1}]$ and the Brownian motion $W[0,T_{n+1}]$ in a common probability space via the {Skorokhod} embedding \eqref{eq:se} with $\eps=1/8$ such that
\begin{equation}\label{eq:se0}
	\Pb\{ H^c \} =O( e^{-10n} )\quad   \text{ with } \quad
	H:=\Big\{\max_{0\le t\le \tau_{n+1}\vee T_{n+1}} |S_t-W_t|\le e^{5n/8}\Big\}.
\end{equation}
We show that under this coupling we can approximate the discrete cut-ball event $K_{3n/4}(z)$ with the continuous one $\widetilde K^{(n)}_{3n/4}(z)$ very precisely, where $\widetilde K^{(n)}_{3n/4}(z)$ is the event $\widetilde K_{n/4}(z)$ after upscaling by $e^n$ (recall the cut-ball event $\widetilde K_{n/4}(z)$ for BM in Definition~\ref{def:cut-bm}). More precisely, 
$\widetilde K^{(n)}_{3n/4}(z)$ is the event that
$\Dc_{-n/4}(z)$ is a cut ball for $e^{-n}W[0,T_n]$, in which case (after upscaling by $e^n$) we also call $\Dc_{3n/4}(z)$ is a cut ball for $W[0,T_n]$. 
Since Brownian motion is scaling invariant, we know that (also see Remark~\ref{rem:blow-up} for this relation)
\begin{equation}\label{eq:K(n)}
	\Pb\{ \widetilde K_{n/4}(z) \}=\Pb\{ \widetilde K^{(n)}_{3n/4}(z) \}.
\end{equation} 
Because of the above equality, we will focus on $\widetilde K^{(n)}_{3n/4}(z)$ instead of $\widetilde K_{n/4}(z)$ in this section.

\begin{lemma}\label{lem:Kutc}
	For all $z\in \Dc$ with $\dist(0,z,\partial\Dc)\ge e^{-n/6}$,
	\[
	\Pb\{ K_{3n/4}(z) \} \asymp \Pb\{ \widetilde K_{n/4}(z) \} \asymp a(z)\, n^{3-d} e^{-\eta n/4},
	\]
	where $a(z)$ is given in \eqref{eq:g-z}.
\end{lemma}
We omit the proof as it is very similar to that of Lemma~\ref{lem:biop}, its counterpart in the continuum.

\begin{proposition}\label{prop:Ksim}
	For all $z\in \Dc$ with $\dist(0,z,\partial\Dc)\ge e^{-n/6}$,
		\begin{equation}
		\Pb\{ K_{3n/4}(z) \} \simeq  \Pb\{ \widetilde K^{(n)}_{3n/4}(z) \}. 
	\end{equation}
\end{proposition}

\begin{proof}
     By Lemma~\ref{lem:Kutc}, we only need to show that under the coupling \eqref{eq:se0} for some $u>0$,
     \begin{equation}
     	\Pb\{ K_{3n/4}(z)\Delta \widetilde K^{(n)}_{3n/4}(z) \} \lesssim  a(z)\, e^{-\eta n/4} \, e^{-un}.
     \end{equation}
 We will only deal with the event 
 \[
 E_0:=K_{3n/4}(z) \setminus \widetilde K_{3n/4}^{(n)}(z)
 \] 
 since $\widetilde K_{3n/4}^{(n)}(z)\setminus K_{3n/4}(z)$ can be analyzed in a similar way.
Let $B$ (resp.\ $B_{\pm}$) be the (discrete) ball of radius $e^{3n/4}$ (resp. $e^{3n/4}\pm 2e^{5n/8}$) around $z_n$. We decompose $\lambda=S[0,\tau_n]$ into $\lambda_1\oplus\omega\oplus\lambda_2^R$ (resp. $\lambda_1^{\pm}\oplus\omega^{\pm}\oplus[\lambda_2^{\pm}]^R$) according to its first and last visits to $B$ (resp. $B_{\pm}$). Recall the event $H$ from \eqref{eq:se0}. On the event $H\cap E_0$, at least one of the following three events should happen
 \begin{itemize}
 	\item $F_1$: $\lambda[\tau_n, \tau_{\partial B(e^n+e^{5n/8})}]\cap \Bc_{11n/16}(\lambda(\tau_n))^c\neq\emptyset$;
 	\item $F_2$:  $\omega\cap B_{-}=\emptyset$;
 	\item $F_3$: $\dist(\lambda_1^-,\lambda_2^-)\le e^{11n/16}$;
 	\item $F_4$: $\omega^+\nsubseteq B(z_n,e^{5n/6}-2e^{5n/8})$.
 \end{itemize}
 Then, $\Pb\{ H\cap E_0 \}\le \sum_{i=1}^{4}\Pb\{E_0\cap F_i\}$. 
 Thus, it suffices to show that for each $i$, there exists $u>0$ such that 
 \begin{equation}\label{eq:Ei}
 	\Pb\{E_0\cap F_i\}\lesssim a(z)\,e^{-\eta n/4} \, e^{-un}.
 \end{equation}
 The above estimate holds for $F_1$ by the {g}ambler's ruin estimate (Lemma~\ref{l:trans_recur}) and for  $F_2$ by standard estimate about the Green's function.
 
 By using Lemma 3.1 of \cite{disconnect2dRW}, we know that \eqref{eq:Ei} also holds for $F_3$.
 It remains to deal with $F_4$. For $x,y\in\partial B_+$, let $\omega^+$ be sampled according to $\wh\nu^{\Bc_n}_{x,y}$. Let $\beta_1\oplus\omega\oplus\beta_2^R$ be the decomposition of $\omega^+$ according to its first and last visits to $B$. 
 By the gambler's ruin estimate again, there exists  $u>0$ such that for all $x$ and $y$,
 \begin{equation}\label{eq:G}
 	\wh\nu^{\Bc_n}_{x,y}\{ \omega\subset\Bc_{5n/6}(z_n), \omega^+\nsubseteq B(z_n,e^{5n/6}-2e^{5n/8}) \}=O(e^{-un}).
 \end{equation}
This implies \eqref{eq:Ei} holds for $F_4$.
 Thus, we complete the proof of \eqref{eq:Ei} and the proof of the lemma as well.
\end{proof}

Now, we are able to get the sharp asymptotic for $f(n)$, which improves the up-to-constants estimate in \eqref{eq:fn}. 
\begin{corollary} \label{cor:sharp-f}
	It holds that
		\begin{equation}
		f(n)\simeq c_1 c_*^{-1} \Qf[\Psi_{n/4}] e^{3\eta n/4}.
	\end{equation}
\end{corollary}
\begin{proof}
	By \eqref{eq:K(n)} and Proposition~\ref{prop:Ksim}, 
	$
	\Pb\{ K_{3n/4}(z) \}\simeq \Pb\{ \widetilde K_{n/4}(z) \}.
	$
	By Proposition~\ref{prop:one-point},
	\[
	\Pb\{ \widetilde K_{n/4}(z) \}\simeq c_*^{-1} \Qf [\Psi_{n/4}] e^{-\eta n/4} G^{\cut}_{\Dc}(z).
	\]
	By Theorem~\ref{t:one_pt}, 
	\[
	\Pb\{ \Ac_n(z) \} \simeq c_1^{-1} e^{-\eta n} G^{\cut}_{\Dc}(z).
	\]
	By Proposition~\ref{prop:f2},
	\[
	f(n)\simeq \Pb\{ K_{3n/4}(z) \} \, \Pb\{ \Ac_n(z) \}^{-1}\simeq c_1 c_*^{-1} \Qf[\Psi_{n/4}] e^{3\eta n/4}.
	\]
	This finishes the proof.
\end{proof}

In the rest of this section, we will deal with the mixed (i.e., one cut point and one cut ball) case.
As before, we first give the up-to-constants estimate for the mixed two-point probability.
\begin{lemma}\label{lem:up-mix-1}
	For all $V\in\Vc$, $z,w\in V$ with $|z-w|\ge e^{-n/6}$,
	\begin{equation}\label{eq:up-mix-1}
	\Pb\{ K_{3n/4}(z) \cap \Ac_n(w) \}\asymp_V |z-w|^{-\eta} e^{-5\eta n/4}. 
	\end{equation}
\end{lemma}
\begin{proof}
	Let $d_V=\dist(0,V,\partial\Dc)$. Then $d_V\ge |z-w|\ge e^{-n/6}$ from the definition of a ``nice'' box. We only indicate how to get the total mass from the path-decomposition point of view and omit the technical details for brevity.
	\begin{itemize}
		\item Let $(\gamma_1,\gamma_2)$ be sampled from the boundary-to-boundary excursion measure in $B (z_n,|z-w|e^n/4)\setminus \ol{\Bc_{3n/4}(z_n)}$, i.e., the measure $\nu^{B (z_n,|z-w|e^n/4)\setminus \ol{\Bc_{3n/4}(z_n)}}_{\partial\Bc_{3n/4}(z_n), \partial B(z_n,|z-w|e^n/4)}$, such that $\gamma_1\cap\gamma_2=\emptyset$, which has total mass $e^{2(d-2)\frac34 n} (|z-w|e^{n/4})^{-\xi}$.
		\item Let $(\gamma_3,\gamma_4)$ be NIRW's from $w_n$ to $\partial B(w_n, |z-w|e^n/4)$ with total mass $\asymp(|z-w|e^n)^{-\xi}$. 
		\item Let $\xi$ be the path connecting the endpoints of $\gamma_2$ and $\gamma_3$. Let $(\eta_1,\eta_2)$ be NIRW's from the endpoints of $\gamma_1$ and $\gamma_4$ respectively to $\partial B((z_n+w_n)/2,d_V e^n/2)$. Then, the triple $(\xi,\eta_1,\eta_2)$ has total mass $$\asymp (|z-w|e^n)^{-(d-2)} \left( \frac{d_V e^n}{|z-w|e^n} \right)^{-\xi}.$$
		\item Let $\lambda_1$ be sampled from the path measure from the endpoint of $\eta_1$ to $0$ and $\lambda_2$ be the SRW from the endpoint of $\eta_2$ to its first visit of $\partial\Bc_n$, and restrict $\lambda_1\cap\lambda_2=\emptyset$. Then the total mass of such $(\lambda_1,\lambda_2)$ is $\asymp c(V) e^{-(d-2)n}$.
		\item Let $\omega$ be sampled from the path measure between the endpoints of $(\gamma_1,\gamma_2)$ and restricted in $\Bc_{5n/6}(z_n)$, which has total mass $\asymp n^{3-d} e^{-\frac34 (d-2)n}$.
	\end{itemize}
Multiplying all of the total masses above, we get that the probability that the SRW visits $\Bc_{3n/4}(z_n)$ before $w_n$ and $K_{3n/4}(z) \cap \Ac_n(w)$ occurs is 
\begin{align*}
	&\asymp e^{2(d-2)\frac34 n} (|z-w|e^{n/4})^{-\xi} \times (|z-w|e^n)^{-\xi} \times (|z-w|e^n)^{-(d-2)} \Big( \frac{d_V e^n}{|z-w|e^n} \Big)^{-\xi} \\
	&\qquad \quad \;\;\;\qquad \qquad \qquad \qquad \times c(V) e^{-(d-2)n} \times n^{3-d} e^{-\frac34 (d-2)n}\\
	&\asymp_V |z-w|^{-\eta} e^{-5\eta n/4}.
\end{align*}
By symmetry, we also know that {the probability of the} event that a SRW visits $w_n$ before $\Bc_{3n/4}(z_n)$ and $K_{3n/4}(z) \cap \Ac_n(w)$ occurs has the probability on the same order. Thus, we finish the proof.
\end{proof}

Next, we show that the same estimate also holds for BM cut-ball event in place of SRW cut-ball event in \eqref{eq:up-mix-1}, which is more complicated since it involves both BM and SRW.
\begin{lemma}\label{lem:up-mix}
	Under the coupling \eqref{eq:se0}, for all $V\in\Vc$, $z,w\in V$ with $|z-w|\ge e^{-n/6}$,
	\begin{equation}\label{eq:up-mix}
		 \Pb\{ \widetilde K^{(n)}_{3n/4}(z)\cap \Ac_n(w) \} \asymp_V |z-w|^{-\eta} e^{-5\eta n/4}.
	\end{equation}
\end{lemma}
\begin{proof}
	In this case, we need to use a certain type of strong Markov property derived in Lemma~\ref{lem:sem} to decouple BM and SRW. More precisely, let us assume  that the BM and the SRW are coupled according to \eqref{eq:se0} throughout the proof. 
	We also assume that $W$ visits $\Dc_{3n/4}(z_n)$ before $S$ visits $w_n$. The case that $S$ visits $w_n$ before $W$ visits $\Dc_{3n/4}(z_n)$ can be addressed in a similar way.
	
	We decompose $W[0,T_n]$ in the following way. We refer to {Figure}~\ref{fig:mix} for an illustration.

	\begin{itemize}
		\item Let $\gamma_1$ be $W$ started from $0$ until its first visit of $\partial\Dc_{3n/4}(z_n)$.
		\item Let $\gamma_2$ be $W$ started from the endpoint of $\gamma_1$ until hitting $\partial D(z_n,|z-w|e^n/4)$, and we denote this hitting time by $\sigma$ below.
		\item Let $\gamma_3$ be $W$ started form the endpoint of $\gamma_2$ until hitting $\partial D(w_n,|z-w|e^n/4)$ after its last visit of $\Dc_{5n/8}(w_n)$.
		\item $\gamma_4$ be $W$ started form the endpoint of $\gamma_3$ until its exit of $\Dc_n$.
	\end{itemize}
	
	\begin{figure}[h!]
		\centering
		\includegraphics[width=.6\textwidth]{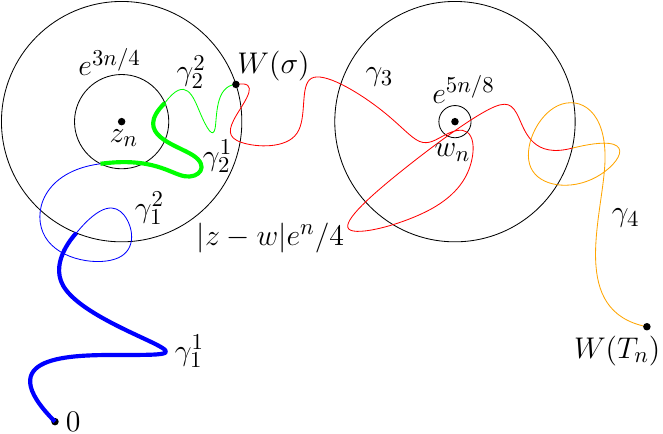}
		\caption{Proof of Lemma~\ref{lem:up-mix}. $\gamma_1$ is the concatenation of $\gamma^1_1$ (thick blue curve) and $\gamma_1^2$ (thin blue curve). $\gamma_2$ is the concatenation of $\gamma_2^1$ (thick green curve) and $\gamma^2_2$ (thin green curve). $\gamma_3$ is in red and $\gamma_4$ is in orange. The two big balls are of the same radius $|z-w|e^n/4$. The small ball around $z_n$ is of radius $e^{3n/4}$ and the small ball around $w_n$ is of radius $e^{5n/8}$. }
		\label{fig:mix}
	\end{figure}

	By Lemma~\ref{lem:sem} with $b=5/7>1/2+\eps$ with $\eps=1/8$, there exists an event $\Upsilon_{w,n}$ with
	\[
	\Pb\{  \Upsilon_{w,n}^c\cap \Ac_n(w) \}\lesssim e^{-10n},
	\]
	such that on $\Upsilon_{w,n}\cap \Ac_n(w)$, 
	\begin{equation}\label{eq:con-ind}
		\gamma_1\cup \gamma_2\cup \gamma_4\text{ and } S[\iota_1,\iota_2]\mbox{ 	are conditionally independent given $\gamma_3$,}
	\end{equation}
 where
	\[
	\iota_1=\inf\{ t> \sigma: S(t)\in B(w_n,|z-w|e^n/8) \},
	\]
	and 
	\[
	\iota_2=\inf\{ t>\tau_{w_n}: S(t)\in \partial B(w_n,|z-w|e^n/8) \}.
	\]
	If $w_n$ is a cut point for $S$, then it is also a cut point for $S[\iota_1,\iota_2]$, which satisfies 
	\begin{equation}\label{eq:AS}
		\Pb\{ w_n\in \Ac_{S[\iota_1,\iota_2]} \}\asymp (|z-w|e^n)^{-\eta}.
	\end{equation}
	Next, we compute the total mass of $(\gamma_1,\gamma_2,\gamma_4)$ that satisfies $\widetilde K^{(n)}_{3n/4}(z)$ as follows.
	\begin{itemize}
		\item Let $\gamma^1_1$ be the part of $\gamma_1$ started from $0$ stopped at its first visit of $\partial D(z_n,|z-w|e^n/4)$. Then the total mass of NI (non-intersecting) $(\gamma^1_1,\gamma_4)$ is  $\asymp_V \Big( \frac{d_V e^n}{|z-w|e^n} \Big)^{-\eta} $.
		\item Let $\gamma_2^1$ be the part of $\gamma_2$ from its starting point to its last visit of $\partial\Dc_{3n/4}(z_n)$. Let $\gamma_1^2=\gamma_1\setminus \gamma^1_1$ and $\gamma^2_2=\gamma_2\setminus\gamma_2^1$. By Lemma~\ref{lem:ex-ni}, the total mass of NI $(\gamma_1^2,\gamma^2_2)$ is
		\[
		\asymp e^{2(d-2)\frac34 n} \Big(\frac{|z-w|e^n}{e^{\frac34 n}}\Big)^{-\xi}\times (|z-w|e^n)^{-(d-2)}=|z-w|^{-\eta} e^{((d-2)/2-\xi/4)n}.
		\]
		\item The total mass of $\gamma_2^1\subseteq \Dc_{5n/6}(z_n)$ is $\asymp n^{3-d}e^{-\frac34 (d-2)n}$.
	\end{itemize}
	The multiplication of the above gives the total mass of admissible $(\gamma_1,\gamma_2,\gamma_4)$ that satisfies $\widetilde K^{(n)}_{3n/4}(z)$, which is  
	\[
	\asymp_V \Big( \frac{d_V}{|z-w|} \Big)^{-\eta} \cdot |z-w|^{-\eta} e^{((d-2)/2-\xi/4)n}
	\cdot n^{3-d}e^{-\frac34 (d-2)n}\asymp_V n^{3-d} e^{-\eta n/4}.
	\]
	Thus, by \eqref{eq:con-ind}, we can multiple the above by
	\eqref{eq:AS} to obtain the order of $|z-w|^{-\eta} e^{-5\eta n/4}$, which is exactly equal to the order of the probability on the left hand side of \eqref{eq:up-mix}. This concludes the proof of the lemma.
\end{proof}

\begin{remark}\label{sesa}
In Lemma~\ref{lem:up-mix}, we choose to use the Skorokhod embedding instead of the strong approximation (which provides with better error bounds that can facilitate some analysis when we prove Proposition \ref{prop:ct-3}) for a technical reason in the case $d=3$, which we now explain.

One core issue in the proof of Lemma~\ref{lem:up-mix} is to decouple the event around $z_n$ for $W$ and the event around $w_n$ for $S$ as \eqref{eq:con-ind}, or to obtain  a certain type of Markov property for the joint processes $(W,S)$, which is feasible with Skorokhod embedding but not with the strong approximation. 
Alternatively, one can try to bound $\Pb\{E_1\}$ by the probability of some event which is purely about $S$ or $W$, then one can obtain the strong Markov property from the marginal distribution of $S$ or $W$ solely. For example, the event $E_1$ implies that $\eta^1$ and $\eta^2$ intersect with each other but they can be made disjoint by perturbing them within a distance ${\cal K} n$ where $\eta^1$ and $\eta^2$ is obtained by applying a first-entry and last-exit decomposition of $\Bc_{3n/4}(z_n)$ for $S$, together with the event that $w_n$ is a cut point for $S$. These two events are all about $S$ and one can analyze it by the path-decomposition (or the strong Markov property) of $S$ as has been used many times in this paper. However, the problem is that under  this approach we do not know how to obtain the extra cost $O(e^{-un})$ from the requirement on $(\eta^1,\eta^2)$ when $d=3$, although we expect it should be true in some sense. In fact, this is true when $d=2$ by using Lemma 5.5 of \cite{MR3547746}.
\end{remark}

With up-to-constants estimates in hand, we are at the point to show a tight bound between the discrete and continuous cut-ball events, given a ``remote'' discrete cut point.

\begin{proposition}\label{prop:ct-3}
	Under the coupling \eqref{eq:se0}, for all $V\in\Vc$, $z,w\in V$ with $|z-w|\ge e^{-n/6}$,
	\begin{equation}\label{eq:K4}
		\Pb\{ K_{3n/4}(z) \cap \Ac_n(w) \} 
		\simeq_V \Pb\{ \widetilde K^{(n)}_{3n/4}(z)\cap \Ac_n(w) \}.
	\end{equation}
\end{proposition}

\begin{proof}
	We denote
	\[
	E_1=\Big(K_{3n/4}(z) \cap \Ac_n(w)\Big)\setminus \Big(\widetilde K^{(n)}_{3n/4}(z)\cap \Ac_n(w)\Big), \quad 
	E_2=\Big(\widetilde K^{(n)}_{3n/4}(z)\cap \Ac_n(w)\Big) \setminus \Big(K_{3n/4}(z) \cap \Ac_n(w)\Big).
	\]
	By Lemmas~\ref{lem:up-mix-1} and~\ref{lem:up-mix}, we only need to show for $i=1,2$ there exists $u>0$ such that 
	\begin{equation}\label{eq:Ei2}
	\Pb\{ E_i \}\lesssim_V |z-w|^{-\eta} e^{-5\eta n/4} e^{-un}.
	\end{equation}
    Let us deal with $E_2$ for an illustration. We will only sketch the proof, since it has a similar flavor to that of Proposition~\ref{prop:Ksim}.
	
	We use the notation introduced in the proof of Lemma~\ref{lem:up-mix}.
	Under the coupling \eqref{eq:se0}, on the event $E_1$, one of the following events will occur: \begin{enumerate}
	\item[(a)] the NI pairs $(\gamma^1_1,\gamma_4)$ and $(\gamma_1^2,\gamma^2_2)$ get close to each other within distance $O(e^{5n/8})$;
	\item[(b)] $\gamma_2^1$ get close to $\partial\Dc_{5n/6}(z_n)$ within distance $O(e^{5n/8})$;
	\item[(c)] the NI pair $(\gamma_1,\gamma_3)$ get close to each other within distance $O(e^{5n/8})$. 
\end{enumerate}
Now, using the same strategy as that of Proposition~\ref{prop:Ksim}, we can always obtain an extra cost $O(e^{-un})$ from (a) (b) and (c), respectively. Thus, we conclude that \eqref{eq:Ei2} holds for $E_1$. The event $E_2$ can be analyzed in a similar way. We thus finish the proof of this proposition.
\end{proof}

\section{Convergence of measures  \label{sec:measure_conv}}
In this section, we put everything together and prove Theorem~\ref{thm:rw_occ} and Theorem~\ref{thm:vg}. 

We begin with the following proposition.
\begin{proposition}\label{prop:vgtv}
	Under the coupling of \eqref{eq:se0}, there exists $u>0$ such that for all $V\in \Dc$ and $\dist(0,V,\partial \Dc)\ge e^{-n/6}$, 
	\begin{equation}\label{eq:vgtv}
		\Eb[|\nu_n(V)-\wt \nu_{n/4}(V)|^2]=O_V(e^{-un}).
	\end{equation}
\end{proposition}

\begin{proof}[Proof of Theorem \ref{thm:rw_occ}]
The claim follows from Proposition \ref{prop:vgtv} immediately since by Theorem~\ref{thm:tvv} we have $\Eb [ \wt \nu_{n/4}(V)^2 ] \simeq \Eb[\nu(V)^2]$ .
\end{proof}

\begin{proof}[Proof of Proposition \ref{prop:vgtv}]
	All the implied constants in the proof can depend on $V$.
	Combining Proposition~\ref{prop:ct-3} with Proposition~\ref{prop:AKAA}, we have for all $z,w\in V$ with $|z-w|\ge e^{-n/6}$, 
	\begin{equation}\label{eq:8-1}
		 \Pb \{ \Ac_n(z) \cap \widetilde K^{(n)}_{3n/4}(w) \} 
		  \simeq \Pb\{ \Ac_n(z) \cap K_{3n/4}(w) \} 
		 \simeq f(n)\, \Pb \{ \Ac_n(z) \cap \Ac_n(w) \}.
	\end{equation}
It then follows from Theorem~\ref{t:two_pt} and Corollary~\ref{cor:sharp-f} that for $|z-w|\ge e^{-n/6}$,
\[
 \Pb \{ \Ac_n(z) \cap  \widetilde K^{(n)}_{3n/4}(w) \} 
 \simeq f(n) c_1^{-2} e^{-2\eta n} G^{\cut}_{\Dc}(z,w)\simeq 
 c_1^{-1} c_*^{-1} G^{\cut}_{\Dc}(z,w) \Qf[\Psi_{n/4}] n^{3-d} e^{-5\eta n/4}.
\]
	Using the above estimates,
	\begin{align*}
		&\  \int_{V}\int_{V} \Pb \{ \Ac_n(z) \cap  \widetilde K^{(n)}_{3n/4}(w) \} \, dz\, dw  \\
		\ge &\  \int_{V}\int_{V} \Pb \{ \Ac_n(z) \cap  \widetilde K^{(n)}_{3n/4}(w) 1_{ |z-w|\ge e^{-n/6} } \} \, dz\, dw \\
		\simeq &\  c_1^{-1} c_*^{-1} \Qf[\Psi_{n/4}] n^{3-d} e^{-5\eta n/4}  
		\int_{V}\int_{V} G^{\cut}_{\Dc}(z,w) 1_{\{ |z-w|\ge e^{-n/6} \}} \, dz\, dw \\
		\simeq &\  c_1^{-1} c_*^{-1}  \Eb[\nu(V)^2] \,
		\Qf[\Psi_{n/4}] \, e^{-5\eta n/4}  \\
		\gtrsim &\ n^{3-d} e^{-5\eta n/4}.
	\end{align*}
	On the other hand,
	\begin{align*}
		&\int_{V}\int_{V} \Pb \{ \Ac_n(z) \cap \widetilde K^{(n)}_{3n/4}(w) \} 1_{\{ |z-w|\le e^{-n/6} \}} \, dz\, dw \\
		\lesssim\, & e^{-dn/6} \int_{V} \Pb \{ \Ac_n(z) \} dz
		\lesssim e^{-dn/6} e^{-\eta n} \int_{V} G^{\cut}_{\Dc}(z) dz
		\lesssim e^{-dn/6-\eta n}.
	\end{align*}
	Since {$\eta=\eta_d<2d/3$ for both $d=2,3$}, we have $d/6+\eta>5\eta/4$ and the integral over the near diagonal $\{ |z-w|\le e^{-n/6} \}$ only comprises an exponentially small portion {in the integral}. Therefore,
	\begin{align}
		\label{eq:and}
		\int_{V}\int_{V} \Pb \{ \Ac_n(z) \cap \widetilde K^{(n)}_{3n/4}(w) \} \, dz\, dw 
		\simeq c_1^{-1} c_*^{-1}  \Eb[\nu(V)^2] \,
		\Qf[\Psi_{n/4}] \, e^{-5\eta n/4}.
	\end{align}
	Recalling \eqref{eq:def_rwocc} for the definition of $\nu_n$, it follows that 
	\begin{align*}
		\Eb [ \nu_n(V) \wt \nu_{n/4}(V) ]&=c_1 e^{\eta n} c_* \Qf[\Psi_{n/4}]^{-1} e^{\eta n/4} \int_{V}\int_{V} \Pb \{ \Ac_n(z) \cap \widetilde K^{(n)}_{3n/4}(w) \} \, dz\, dw
		\simeq \Eb[\nu(V)^2].
	\end{align*}
    Next, we deal with $\Eb[\nu_n(V)^2]$. Write 
    \[ 
    \Eb[\nu_n(V)^2] = \sum_{z\in V^{(n)}}\sum_{w\in V^{(n)}} c_1^2 e^{-2(2-\xi)n} \Pb\{ \Ac_n(z) \cap \Ac_n(w) \}.
    \]
    By Lemma~\ref{lem:boundary}, we have 
    \[
    \sum_{z\in V^{(n)}}\sum_{w\in V^{(n)}} c_1^2 e^{-2(2-\xi)n} 1_{ |z-w|\le e^{-n/6} } \,  \Pb\{ \Ac_n(z) \cap \Ac_n(w) \} = O(e^{-(2-\xi)n/6}),
    \]
    and by Theorem~\ref{t:two_pt},
    \begin{align*}
    &\ \sum_{z\in V^{(n)}}\sum_{w\in V^{(n)}} c_1^2 e^{-2(2-\xi)n} 1_{ |z-w|\ge e^{-n/6} } \,  \Pb\{ \Ac_n(z) \cap \Ac_n(w) \} \\
    =&\ \sum_{z\in V^{(n)}}\sum_{w\in V^{(n)}} c_1^2 e^{-2(2-\xi)n} 1_{ |z-w|\ge e^{-n/6} } \, c_1^{-2} e^{-2\eta n} G^{\cut}_{\Dc}(z,w) \\
    \simeq&\ \int_{V}\int_{V} G^{\cut}_{\Dc}(z,w) 1_{\{ |z-w|\ge e^{-n/6} \}} \, dz\, dw,
    \end{align*}
where we use the Riemann sums to approximate the integral in the last line by applying the continuity result for the cut-point Green's function obtained in Lemma~\ref{lem:continuous-G}. It then follows that $\Eb[\nu_n(V)^2]\simeq \Eb[\nu(V)^2]$.
    Moreover, we have
	$\Eb [ \wt \nu_{n/4}(V)^2 ] \simeq \Eb[\nu(V)^2]$ by Theorem~\ref{thm:tvv}. Therefore, we conclude the proof of the proposition by splitting \eqref{eq:vgtv} in a similar fashion as in the proof of Theorem~\ref{thm:tvv} and noting that they cancel out since each of them is close to $\Eb[\nu(V)^2]$.
\end{proof}

We continue with the proof of Theorem~\ref{thm:vg}.
\begin{proof}[Proof of Theorem~\ref{thm:vg}]
	We first show \eqref{eq:nug}. Recall that as in the statement of  Theorem~\ref{thm:vg}, the function $g$ is continuous on $\overline\Dc$, for any $m\ge 1$, there exists $k_m\ge m$ such that 
	\[
	\sup_{|x-y|\le 4\cdot 2^{-k_m}} |g(x)-g(y)|\le 2^{-m}.
	\]
	Write $k$ for $k_m$ for brevity. Let $\{E_1, E_2,\cdots,E_{L_k}\}$ be the set of dyadic box of side length $2^{-k}$ in $\Dc$ that are contained in $\Dc$. Let $H_k$ be the union of these boxes. Then, for some constant $c>0$,
	\[
	\Dc\setminus H_k \subseteq B(c2^{-k})\cup (\Dc\setminus B(1-c2^{-k})).
	\]
	By Lemma~\ref{lem:boundary},
	\[
	\Eb [ \nu_n(\Dc\setminus H_k)^2 ]\lesssim 2^{-k/2}.
	\]
	
	Define $g_k:=\sum_{i=1}^{L_k} g(x_i) 1_{E_i}$ where $x_i$ is the center of $E_i$. Therefore,
	\begin{align*}
		\Eb( \nu_n(g) - \nu_n(g_k) )^2 &\le \Eb( \nu_n(g) - \nu_n(g 1_{H_k}) )^2 + \Eb( \nu_n(g 1_{H_k}) - \nu_n(g_k) )^2\\
		& \lesssim \|g\|^2 2^{-k/2} + 2^{-2m} \Eb ( \nu_n(\Dc) )^2 \\
		&\lesssim 2^{-k/2} + 2^{-2m},
	\end{align*}
where we use \eqref{eq:eb-3} to bound $\Eb ( \nu_n(\Dc) )^2=O(1)$ in the last inequality and we allow the implied constant to depend on $g$. It is easy to see that the same estimate also holds with $\nu_n$ replaced by $\nu$, i.e., 
\[
\Eb( \nu(g) - \nu(g_k) )^2 \lesssim 2^{-k/2} + 2^{-2m}.
\]
Using Theorem~\ref{thm:rw_occ},
\[
\Eb( \nu_n(g_k) - \nu(g_k) )^2 \lesssim \sum_{i=1}^{L_k} \Eb (\nu_n(E_i)-\nu(E_i))^2
\lesssim \sum_{i=1}^{L_k} c(E_i) e^{-un},
\]
where $c(E_i)$ is a constant depends on $E_i$. Three estimates above together yields
\[
\Eb[ (\nu_n(g)-\nu(g))^2 ] \lesssim 2^{-k/2} + 2^{-2m} + \sum_{i=1}^{L_k} c(E_i) e^{-un}.
\]
Letting $n$ tend to infinity first and $m$ tend to infinity afterwards and noticing that $k\ge m$, we complete the proof of \eqref{eq:nug}.

Finally, we show that $\nu_n$ converges in probability for the weak topology towards $\nu$. Given the kind of $L^2$ convergence that we have proven, we can readily infer this by using the portmanteau theorem, and the fact that $\nu$ is absolutely continuous with respect to Lebesgue measure by Theorem~\ref{thm:mink}. As the Portmanteau theorem is usually stated for probability measures, we refer the reader to Section 6 of \cite{MR3652040}\footnote{Note that although \cite{MR3652040} deals with Gaussian Mulplicative Chaos, the argument therein is not GMC-specific.} for a proof of the same flavor.
\end{proof}

\bibliography{refs.bib}
\end{document}